\documentclass[a4paper,reqno,10pt]{amsart}

\usepackage{epsf,graphicx,epsfig}
\usepackage{amscd}
\usepackage{amsmath,latexsym,amssymb,amsthm}
\usepackage[nospace,noadjust]{cite}
\usepackage{textcomp}
\usepackage{amsfonts}
\usepackage{setspace,cite}
\usepackage{lscape,fancyhdr,fancybox}
\usepackage{stmaryrd}
\usepackage{mathrsfs}
\usepackage[all,cmtip]{xy}
\usepackage{tikz}
\usepackage{cancel}
\usetikzlibrary{shapes,arrows,decorations.markings}
\usepackage{graphicx}
\usepackage{tikz-cd}
\usepackage{color}
\usepackage{url}
\usepackage{enumerate}
\usepackage[mathscr]{euscript}
  \theoremstyle{plain}

\swapnumbers
    \newtheorem{thm}{Theorem}[section]
    \newtheorem{prop}[thm]{Proposition}

    \newtheorem{corollary}[thm]{Corollary}
    
    \newtheorem{subsec}[thm]{}
\theoremstyle{definition}
    \newtheorem{defn}[thm]{Definition}
        \newtheorem{remark}[thm]{Remark}
    \newtheorem{exam}[thm]{Example}

\theoremstyle{remark}

\title{}
\author{}
\date{}
\usepackage{amssymb}

\usepackage{hyperref}
\hypersetup{
	colorlinks,
	citecolor=blue,
	filecolor=black,
	linkcolor=blue,
	urlcolor=black
}

\makeatletter
\newcommand*\xrightleftharpoons[2][]{%
  \mathrel{%
    \raise.22ex\hbox{%
      $\m@th\ext@arrow 0359\MT@rightharpoonup@fill{\phantom{#1}}{#2}$%
    }%
    \setbox0=\hbox{%
      $\m@th\ext@arrow 3095\MT@leftharpoondown@fill{#1}{\phantom{#2}}$%
    }%
    \kern-\wd0 \lower.22ex\box0 %
  }%
}
\def\MT@rightharpoonup@fill{%
  \arrowfill@\relbar\relbar\rightharpoonup
}
\def\MT@leftharpoondown@fill{%
  \arrowfill@\leftharpoondown\relbar\relbar
}

\@ifpackageloaded{amsmath}{}{%
  \providecommand*{\arrowfill@}[4]{%
    $\m@th\thickmuskip0mu\medmuskip\thickmuskip\thinmuskip\thickmuskip
     \relax#4#1\mkern-7mu%
     \cleaders\hbox{$#4\mkern-2mu#2\mkern-2mu$}\hfill
     \mkern-7mu#3$%
  }
  \providecommand*{\ext@arrow}[7]{%
    \mathrel{\mathop{%
      \setbox\z@\hbox{#5\displaystyle}%
      \setbox\tw@\vbox{\m@th
        \hbox{$\scriptstyle\mkern#3mu{#6}\mkern#4mu$}%
        \hbox{$\scriptstyle\mkern#3mu{#7}\mkern#4mu$}%
        \copy\z@
      }%
      \hbox to\wd\tw@{\unhbox\z@}}%
    \limits
      \@ifnotempty{#7}{^{\if0#1\else\mkern#1mu\fi
                         #7\if0#2\else\mkern#2mu\fi}}%
      \@ifnotempty{#6}{_{\if0#1\else\mkern#1mu\fi
                         #6\if0#2\else\mkern#2mu\fi}}}%
  }
  \providecommand{\@ifnotempty}[1]{%
    \@ifempty{#1}{}%
  }
  \providecommand{\@ifempty}[1]{\@xifempty#1@@..\@nil}
  \@ifundefined{@xifempty}{%
    \long\def\@xifempty#1#2@#3#4#5\@nil{%
      \ifx#3#4\expandafter\@firstoftwo\else\expandafter\@secondoftwo\fi
    }
  }{}
}
\makeatother

\begin{document}

\title[]{Deformation cohomology of Nijenhuis algebras and applications to extensions, inducibility of automorphisms and homotopy algebras}

\author{Apurba Das}
\address{Department of Mathematics,
Indian Institute of Technology, Kharagpur 721302, West Bengal, India}
\email{apurbadas348@gmail.com, apurbadas348@maths.iitkgp.ac.in}

%\author{Suman Majhi}
%\address{Department of Mathematics,
%Indian Institute of Technology, Kharagpur 721302, West Bengal, India}
%\email{majhisuman693@gmail.com}

%\author{Ramkrishna Mandal}
%\address{Department of Mathematics, Indian Institute of Technology, Kharagpur 721302, West Bengal, India}
%\email{ramkrishnamandal430@gmail.com}

\begin{abstract}
%In this paper, we first define the cohomology of a Nijenhuis operator and then the cohomology of a Nijenhuis algebra. 
Our primary aim in this paper is to introduce and study the cohomology of a Nijenhuis operator and of a Nijenhuis algebra.
Our cohomology of a Nijenhuis algebra controls the simultaneous deformations of the underlying associative structure and the Nijenhuis operator. We interpret the second cohomology group as the space of all isomorphism classes of abelian extensions. Then we study the inducibility of a pair of Nijenhuis algebra automorphisms in a given abelian extension and show that the corresponding obstruction can be seen as the image of a suitable Wells-type map. We also consider skeletal and strict $2$-term homotopy Nijenhuis algebras and characterize them by third cocycles of Nijenhuis algebras and crossed modules of Nijenhuis algebras, respectively. Finally, we introduce strict homotopy Nijenhuis operators on $A_\infty$-algebras and show that they induce $NS_\infty$-algebras. 
\end{abstract}

\maketitle

%\curraddr{}
%\email{}

%\subjclass[2010]{}
%\keywords{}

\medskip

\begin{center}

\end{center}

 %Quasi-twilled associative algebras, Deformation maps, Crossed homomorphisms, Rota-Baxter operators, Controlling algebras, Cohomology.
 
 %Rota-Baxter operators, twisted Rota-Baxter operators, Crossed homomorphism, Averaging operators, Reynolds operators, Cohomology, Deformation.

 %\medskip

%\noindent {\sf Date of resubmission:} July 26, 2021.

\thispagestyle{empty}

\tableofcontents

%\vspace{0.2cm}

\medskip

\section{Introduction}
Finding the cohomology of an algebraic structure and applications to deformations and extensions are traditional problems initiated by Hochschild \cite{hoch} and Gerstenhaber \cite{gers} for associative algebras. The same studies were subsequently generalized to the context of Lie algebras by Nijenhuis and Richardson \cite{nij-ric}. Since then, cohomologies of various other algebraic structures were extensively studied and their applications to deformations were obtained. On the other hand, algebras are often equipped with additional compatible structures such as homomorphisms, derivations, representations etc. In such cases, it is also meaningful to consider simultaneous deformations (i.e. deformations of algebras and additional structures). In \cite{gers-sch} Gerstenhaber and Schack first considered the cohomology and deformation theory of associative algebras endowed with homomorphisms. A further study on homomorphisms was also carried out using the minimal model of operads \cite{yau} and the method of derived brackets \cite{fre}. Similarly, the operadic treatment of algebras endowed with derivations is first considered by Loday in \cite{loday-der}.

%However, a modern approach for studying simultaneous deformations of algebras and homomorphisms was discovered by Fr\'{e}gier and Zambon 

\medskip

In recent times, Rota-Baxter operators and their various cousins (such as relative Rota-Baxter operators, Rota-Baxter operators of any scalar weight $\lambda \in {\bf k }$, twisted Rota-Baxter operators, modified Rota-Baxter operators, averaging operators etc.) attract much interest due to their importance in the study of the classical Yang-Baxter equation, classical $r$-matrices, splitting of operads, pre-Lie algebras, (tri)dendriform algebras, infinitesimal bialgebras and quasisymmetric functions (see \cite{guo-book}, \cite{ebrahimi-loday}, \cite{uchino}, \cite{lin} and the references therein). Rota-Baxter operators and all their cousins can be characterized by their graphs as subalgebras of some bigger algebras. In the realm of the study of algebras equipped with additional structures, recently the cohomology and deformation theory of Rota-Baxter algebras are extensively studied. The first step toward such a study was given in \cite{O-op}, \cite{das-rota} where the cohomology of a Rota-Baxter operator was defined. Subsequently, the authors in \cite{lazarev}, \cite{das-mishra-jmp} construct an $L_\infty$-algebra whose Maurer-Cartan elements correspond to Rota-Baxter algebra structures. Then twisting by the Maurer-Cartan element, one obtains the controlling $L_\infty$-algebra for the corresponding Rota-Baxter algebra. The differential map of this controlling algebra is precisely the coboundary operator of the cochain complex defining the cohomology of the Rota-Baxter algebra. The same technique can be adapted to obtain the cohomology of algebras endowed with any of the above cousins of the Rota-Baxter operator. However, an exciting operator that appears in the linear deformation theory of algebraic structures \cite{gers}, integrable systems and tensor hierarchies in mathematical physics \cite{koss}, quantum bi-Hamiltonian systems \cite{gra-bi} is the {\em Nijenhuis operator}. 
A Nijenhuis operator on an associative algebra gives rise to a new associative algebra structure on the underlying space, called the deformed associative algebra. From the algebraic perspective, Nijenhuis operators also got a lot of attention. They were studied in a wide class of algebraic structures  (see for instance \cite{azimi}, \cite{baishya-das2}, \cite{ebrahimi}, \cite{ebrahimi-leroux}, \cite{lei}, \cite{leroux}, \cite{liu}, \cite{liu-sheng}, \cite{ma}, \cite{saha}, \cite{peng}, \cite{wang}, \cite{yuan}).
 Unlike Rota-Baxter operators, it is important to remark that Nijenhuis operators cannot be characterized by their graphs as subalgebras of some bigger algebras. This makes the study of Nijenhuis operators a bit more intricate and therefore some results about Nijenhuis operators are not parallel to the study of Rota-Baxter operators. In \cite{leroux} Leroux introduced the notion of an {\em NS-algebra} as the algebraic structure induced by a Nijenhuis operator (see also \cite{lei}). NS-algebras are a generalization of dendriform algebras and are closely related to twisted Rota-Baxter operators \cite{uchino}. Recently, the present author has introduced the cohomology and deformation theory of an NS-algebra using the idea of nonsymmetric operads with multiplication \cite{das-ns}.

\medskip

Our primary aim in this paper is to introduce and extensively study the cohomology of a Nijenhuis operator (defined on an associative algebra) and also the cohomology of a Nijenhuis algebra. To begin with, we first recall the associative analogue of the Fr\"{o}licher-Nijenhuis bracket considered recently by Baishya and the present author \cite{baishya-das3}. Since a Nijenhuis operator $N$ is a Maurer-Cartan element for the Fr\"{o}licher-Nijenhuis bracket, the operator $N$ induces a cochain complex. We define the corresponding cohomology groups as the cohomology of the Nijenhuis operator $N$. It is important to note that, unlike the Rota-Baxter case, the cohomology of the Nijenhuis operator $N$ cannot be expressed as the Hochschild cohomology of the deformed associative algebra. We find a homomorphism from the cohomology of a Nijenhuis operator $N$ to the cohomology of the induced NS-algebra. Consequently, we also obtain a homomorphism from the cohomology of $N$ to the Hochschild cohomology of the deformed associative algebra. Next, we put our interest to define the cohomology of a Nijenhuis algebra. To do so, given a Nijenhuis algebra, we first obtain a homomorphism from the Hochschild cochain complex of the underlying associative algebra to the cochain complex induced by the Nijenhuis operator. The mapping cone corresponding to this homomorphism is defined to be the cochain complex associated with the Nijenhuis algebra. However, we are interested in a reduced version of this cochain complex to study deformations and abelian extensions of Nijenhuis algebras. We show that the cohomology groups thus obtained govern the deformations of the Nijenhuis algebra, i.e. simultaneous deformations of the underlying algebra and the Nijenhuis operator (cf. Theorem \ref{theorem-deformation1}, Theorem \ref{theorem-deformation2}). In the course of this paper, we explicitly define the cohomology of a Nijenhuis algebra with coefficients in a more general Nijenhuis bimodule. We observe that the second cohomology group with coefficients in a Nijenhuis bimodule parametrizes the isomorphism classes of all abelian extensions (cf. Theorem \ref{theorem-abelian}).

%Then as a byproduct with the cohomology of the underlying algebra, one establishes the cohomology of a Rota-Baxter algebra \cite{wang-zhou}. 

%A Nijenhuis operator on an associative algebra induces an NS-algebra structure \cite{leroux} in the same way a Rota-Baxter operator induces a dendriform algebra structure. 

\medskip

Another interesting study related to the cohomology of an algebraic structure is the inducibility problem for a pair of automorphisms. Given an extension of abstract groups, C. Wells \cite{wells} first formulated an exact sequence connecting various automorphism groups and obtained a criterion for the inducibility of a pair of automorphisms. His study was further explored in various specific cases and generalized to other algebraic structures \cite{jin}, \cite{bar-singh}. With Hazra and Mishra, the present author considered the extensions of Rota-Baxter Lie algebras and studied the inducibility problem \cite{das-hazra-mishra}. In particular, they defined the Wells map in the context of Rota-Baxter Lie algebras. They showed that a pair of Rota-Baxter Lie algebra automorphisms is inducible if and only if its image under the Wells map vanishes identically. Therefore, it is natural to ask whether the inducibility problem can be answered in the context of Nijenhuis algebras. Given an abelian extension of Nijenhuis algebras, we first define a suitable Wells-type map $\mathcal{W}$ and show that a pair of Nijenhuis algebra automorphisms is inducible if and only if the pair is `compatible' and its image under the map $\mathcal{W}$ vanishes identically (cf. Theorem \ref{theorem-wells}).

\medskip

On the other side, homotopy invariant algebraic structures play a prominent role in modern mathematical physics. Among others, $A_\infty$-algebras and $L_\infty$-algebras are the key ones (see, for example \cite{keller}, \cite{mapping}). Our final aim in this paper is to define and study homotopy Nijenhuis operators on $A_\infty$-algebras. At first, we
consider homotopy Nijenhuis operators on $2$-term $A_\infty$-algebras. A pair consisting of a $2$-term $A_\infty$-algebra and a homotopy Nijenhuis operator is called a {\em $2$-term homotopy Nijenhuis algebra}. Particular attention are given to skeletal and strict $2$-term homotopy Nijenhuis algebras. We show that skeletal $2$-term homotopy Nijenhuis algebras can be characterized by third cocycles of Nijenhuis algebras. Then we introduce crossed modules of Nijenhuis algebras and show that they are equivalent to strict $2$-term homotopy Nijenhuis algebras. Next, we define the notion of a {\em strict homotopy Nijenhuis operator} on an arbitrary $A_\infty$-algebra. However, a generic homotopy Nijenhuis operator on an $A_\infty$-algebra is yet to be found. At last, we define the concept of an {\em $NS_\infty$-algebra} (also called a {\em strongly homotopy NS-algebra}) and show that a strict homotopy Nijenhuis operator induces an $NS_\infty$-algebra structure. This generalizes the result of Leroux \cite{leroux} that a Nijenhuis operator on an associative algebra gives rise to an NS-algebra.

\medskip

The present paper is organized as follows. In Section \ref{sec2}, we define the cohomology of a Nijenhuis operator $N$ and obtain a homomorphism to the cohomology of the induced NS-algebra. Then using a byproduct of the Hochschild cochain complex of the underlying associative algebra and the cochain complex induced by the Nijenhuis operator, we define the cohomology of a Nijenhuis algebra in Section \ref{sec3}. Applications to deformations and abelian extensions of Nijenhuis algebras are provided in Section \ref{sec4}. Given an abelian extension, in Section \ref{sec5}, we study the inducibility of a pair of Nijenhuis algebra automorphisms and find a necessary and sufficient condition in terms of a suitable Wells-type map. Finally, we characterize skeletal and strict $2$-term homotopy Nijenhuis algebras in Section \ref{sec6}. We also introduce and study strict homotopy Nijenhuis operators and $NS_\infty$-algebras.

%\subsection{Nijenhuis operators and Nijenhuis algebras}

%\subsection{Cohomology and simultaneous deformations}

%\subsection{Inducibility of automorphisms and the Wells map}

%\subsection{Homotopy Nijenhuis algebras and $NS_\infty$-algebras}

\medskip

All algebras, vector spaces, (multi)linear maps and tensor products are over a field ${\bf k}$ of characteristic $0$ unless specified otherwise.

\section{Nijenhuis operators and the Fr\"{o}licher-Nijenhuis bracket}\label{sec2} 
In this section, we first recall some examples and basic properties of Nijenhuis operators. In particular, given an associative algebra $A$, we describe the Fr\"{o}licher-Nijenhuis bracket (that yields a graded Lie bracket) whose Maurer-Cartan elements correspond to Nijenhuis operators on $A$. Then we define and study the cohomology of a Nijenhuis operator. In the end, we also obtain a homomorphism from the cohomology of a Nijenhuis operator $N$ to the cohomology of the induced NS-algebra. As a consequence, we connect the cohomology of the operator $N$ to the Hochschild cohomology of the deformed associative algebra.

\begin{defn}
   Let $A = (A, ~ \! \cdot ~ \! )$ be an associative algebra. A linear map $N: A \rightarrow A$ is said to be a {\bf Nijenhuis operator} on $A$ if it satisfies
   \begin{align*}
       N (a) \cdot N (b) = N ( N(a) \cdot b + a \cdot N(b) - N (a \cdot b)), \text{ for all } a, b \in A. 
   \end{align*}
\end{defn}

\begin{exam}Let $ (A, ~ \! \cdot ~ \! )$ be any associative algebra.
\begin{itemize}
    \item[(i)] Then the identity map $\mathrm{Id}_A : A \rightarrow A$ is a Nijenhuis operator on $A$.
    \item[(ii)] For any element $x \in A$, the left multiplication map $l_x : A \rightarrow A,~ \! a \mapsto x \cdot a$ and the right multiplication map $r_x : A \rightarrow A, ~ \! a \mapsto a \cdot x$ are both Nijenhuis operators on $A$. In particular, for any non-negative integer $k$, the map $N : {\bf k} [x] \rightarrow {\bf k} [x]$ given by $N (x^n) = x^{n+k}$ is a Nijenhuis operator on $A = {\bf k}[x].$ 
    \item[(iii)] If $N : A \rightarrow A$ is a Nijenhuis operator on $A$ then for any scalar $\lambda \in {\bf k}$, the map $\lambda N$ is also a Nijenhuis operator on $A$.
\end{itemize}
\end{exam} 

\begin{exam}\label{exam-rrb}
Let  $A = (A, ~ \! \cdot ~ \! )$ be an associative algebra. An {\em $A$-bimodule} is a vector space $M$ endowed with two bilinear maps $\triangleright : A \times M \rightarrow M, ~ \! (a,u) \mapsto a \triangleright u$ and $\triangleleft : M \times A \rightarrow M, ~ \! (u, a ) \mapsto u \triangleleft a$ that satisfy the following associative-like identities:
\begin{align*}
    (a \cdot b) \triangleright u = a \triangleright (b \triangleright u), \quad (a \triangleright u) \triangleleft b = a \triangleright (u \triangleleft b ), \quad (u \triangleleft a ) \triangleleft b = u \triangleleft (a \cdot b), \text{ for } a, b \in A \text{ and }  u \in M.
\end{align*}
Recall that, in this case, the direct sum $A \oplus M$ carries an associative algebra structure (called the {\em semidirect product}) with multiplication
\begin{align}\label{semid}
    (a, u) \cdot_\ltimes (b, v) := (a \cdot b ~ \! , ~ \!  a \triangleright v + u \triangleleft b), \text{ for } (a, u), (b, v) \in A \oplus M.
\end{align}
%Here we observe that derivations and (relative) Rota-Baxter operators can be seen as Nijenhuis operators on the semidirect product algebra.
\begin{itemize}
\item[(i)] For any linear map $d : A \rightarrow M$, its lift $\widetilde{d} : A \oplus M \rightarrow A \oplus M$ defined by $\widetilde{d} (a, u) = (0, d(a))$ is a Nijenhuis operator on the semidirect product algebra $(A \oplus M, ~ \! \cdot_\ltimes ~ \! )$.

%A {\em derivation} on $A$ with values in the $A$-bimodule $M$ is a linear map $d : A \rightarrow M$ that satisfies
%\begin{align*}
%    d (a \cdot b) = a \triangleright d (b) + d(a) \triangleleft b, \text{ for } a, b \in A.
%\end{align*}
%Note that a linear map $d : A \rightarrow M$ is a derivation if and only if its lift $\widetilde{d} : A \oplus M \rightarrow A \oplus M$ defined by $\widetilde{d} (a, u) = (0, d(a))$ is a Nijenhuis operator on the semidirect product algebra $(A \oplus M, ~ \! \cdot_\ltimes ~ \! )$.
\item[(ii)] A linear map $R : M \rightarrow A$ is called a {\em relative Rota-Baxter operator} (also called an {\em $\mathcal{O}$-operator}) \cite{uchino,das-rota} on $A$ with respect to the $A$-bimodule $M$ if it satisfies
\begin{align*}
    R(u) \cdot R(v) = R ( R (u) \triangleright v + u \triangleleft R (v) ), \text{ for } u, v \in M.
\end{align*}
It turns out that a linear map $R: M \rightarrow A$ is a relative Rota-Baxter operator if and only if its lift $\widetilde{R} : A \oplus M \rightarrow A \oplus M$ defined by $\widetilde{R} (a, u) = (R(u), 0)$ is a Nijenhuis operator on the semidirect product algebra $(A \oplus M, ~ \! \cdot_\ltimes ~ \! )$.
\end{itemize}
\end{exam}

\begin{exam}
 Let $A$ be an associative algebra and $(M, \triangleright, \triangleleft)$ be an $A$-bimodule. Suppose $R_1, R_2 : M \rightarrow A$ are two relative Rota-Baxter operators that are compatible in the sense that their sum $R_1 + R_2$ is also a relative Rota-Baxter operator. Additionally, if $R_2$ is invertible then $N= R_1 R_2^{-1}$ is a Nijenhuis operator on $A$ \cite{liu}.
\end{exam}

\begin{exam}
    A {\em twilled associative algebra} is an associative algebra $(A, ~ \! \cdot ~ \!)$ whose underlying vector space $A$ has a direct sum decomposition $A = A_1 \oplus A_2$ into two subalgebras. Let $p_1: A \rightarrow A$ (resp. $p_2 : A \rightarrow A$) be the projection map onto $A_1$ (resp. $A_2$). Then $p_1$ and $p_2$ are both Nijenhuis operators on $A$. In fact, any linear combination of $p_1$ and $p_2$ is also a Nijenhuis operator on $A.$
\end{exam}

\begin{exam}
    Let $A$ be an associative algebra. A {\em free Nijenhuis algebra} over $A$ is an associative algebra $F_\mathrm{N} (A)$ with a Nijenhuis operator $\mathcal{N} : F_\mathrm{N} (A) \rightarrow F_\mathrm{N} (A)$ and an algebra homomorphism $i: A \rightarrow F_\mathrm{N} (A)$ satisfying some universal property \cite{lei}. In the above reference, the authors have constructed the free Nijenhuis algebra over any associative algebra $A$ and thus obtained a Nijenhuis operator $\mathcal{N}$.
\end{exam}

\begin{exam}
    A {\em dendriform-Nijenhuis bialgebra} \cite{leroux} (see also \cite{peng}) is a triple $(A, ~ \! \cdot ~ \!, \Delta)$ consisting of an associative algebra $(A, ~ \! \cdot ~ \! )$ and a coassociative coalgebra $(A, \Delta)$ both defined on a same vector space satisfying additionally the following compatibility condition:
    \begin{align*}
        \Delta (a \cdot b) = a_{(1)} \otimes a_{(2)} \cdot b ~ \! +~ \! a \cdot b_{(1)} \otimes b_{(2)}  \! ~-~ \! a_{(1)} \cdot a_{(2)} \otimes b, \text{ for all } a, b \in A.
    \end{align*}
    Here we use the Sweedler's notation $\Delta (a) = a_{(1)} \otimes a_{(2)}$ and $\Delta (b) = b_{(1)} \otimes b_{(2)}$ to describe the coproducts. Let $(A, ~ \! \cdot ~ \!, \Delta)$ be a dendriform-Nijenhuis bialgebra. Consider the space $\mathrm{End} (A)$ which has a natural associative algebra structure under the composition $\circ$. With this structure, the map 
    \begin{align*}
    N : \mathrm{End} (A) \rightarrow \mathrm{End} (A) ~~ \text{ defined by } ~~ N (f) (a) := a_{(1)} \cdot f (a_{(2)}), \text{ for } f \in \mathrm{End}(A) \text{ and } a \in A
    \end{align*}
    is a Nijenhuis operator on $(\mathrm{End} (A), \circ)$.
\end{exam}

It is important to remark that Nijenhuis operators on an associative algebra $(A, ~ \! \cdot ~ \!)$ are closely related to linear deformations of the algebra structure \cite{gers} (see also \cite{liu}). Indeed, if $N : A \rightarrow A$ is a Nijenhuis operator then
\begin{align*}
    a \cdot_t b := a \cdot b + t ~ \! ( N(a) \cdot b + a \cdot N (b) - N (a \cdot b)), \text{ for } a, b \in A 
\end{align*}
is a linear deformation of the given associative algebra $(A, ~ \! \cdot ~ \!)$. However, this linear deformation is `trivial' in the sense that it is equivalent to the undeformed one $a \cdot_t' b = a \cdot b$.

Let $(A, \! ~ \cdot ~ \!)$ be an associative algebra and $N: A \rightarrow A$ be a Nijenhuis operator on it. Then the underlying vector space $A$ inherits a new associative algebra structure with the (deformed) multiplication 
\begin{align}\label{deformed-alg}
    a \cdot_N b := N (a) \cdot b + a \cdot N (b) -N (a \cdot b), \text{ for } a, b \in A.
\end{align}
We often denote this deformed associative algebra $(A, \cdot_N )$ simply by $A_N$ when no confusion arises. More generally, one has the following result \cite{gra-bi}.

\begin{prop}
    Let $(A, ~ \! \cdot ~ \! )$ be an associative algebra and $N: A \rightarrow A$ be a Nijenhuis operator on it.
    \begin{itemize}
        \item[(i)] Then for each $k \geq 0$, the map $N^k : A \rightarrow A$ is also a Nijenhuis operator on the associative algebra $(A, ~ \! \cdot ~ \!)$. 
        \item[(ii)] For any $k, l \geq 0$, the map $N^l : A \rightarrow A$ is a Nijenhuis operator on the deformed associative algebra $(A, ~ \! \cdot_{N^k} ~ \!)$. Moreover, the deformed algebras $(A, ( ~ \! \cdot_{N^k})_{N^l})$ and $(A, ~ \! \cdot_{N^{k+l}})$ are the same.
    \end{itemize}
\end{prop}

\medskip

In the following, we revise the Fr\"{o}licher-Nijenhuis bracket \cite{baishya-das3} for a given associative algebra and study the cohomology of a Nijenhuis operator. Let $A$ be any vector space (not necessarily an associative algebra). For any $f \in \mathrm{Hom} (A^{\otimes m}, A)$ and $g \in \mathrm{Hom} (A^{\otimes n}, A)$ with $m, n \geq 1$, the {\em contraction} of $f$ by $g$ is denoted by $i_g f \in \mathrm{Hom} (A^{\otimes m+n-1}, A)$ and it is defined by
\begin{align*}
    (i_g f ) (a_1, \ldots, a_{m+n-1}) := \sum_{i=1}^m (-1)^{ (i-1) (n-1)} ~ \! f ( a_1, \ldots, a_{i-1}, g (a_i, \ldots, a_{i+n-1}), \ldots, a_{m+n-1}),
\end{align*}
for $a_1, \ldots, a_{m+n-1} \in A$. Note that there is a one-to-one correspondence between associative algebra structures on the vector space $A$ and elements $\mu \in \mathrm{Hom}(A^{\otimes 2}, A)$ with $i_\mu \mu = 0 $.

Next, let $(A, ~ \! \cdot ~ \!)$ be an associative algebra with the corresponding multiplication map $\mu \in \mathrm{Hom}(A^{\otimes 2}, A)$, i.e. $\mu (a, b) = a \cdot b$, for all $a, b \in A$. Then the Hochschild coboundary operator of the associative algebra $(A, ~ \! \cdot ~ \!)$ is the map $\delta_\mathrm{Hoch} : \mathrm{Hom} (A^{\otimes n}, A) \rightarrow \mathrm{Hom} (A^{\otimes n+1}, A) $, for $n \geq 0$, given by $\delta_\mathrm{Hoch} f := (-1)^{n-1} ~ \! i_f \mu - i_\mu f$, for $f \in \mathrm{Hom}(A^{\otimes n}, A)$. Explicitly, we have
\begin{align*}
    (\delta_\mathrm{Hoch} f) (a_1, \ldots, a_{n+1}) =~& a_1 \cdot f (a_2, \ldots, a_{n+1}) + \sum_{i=1}^n (-1)^i ~ \! f (a_1, \ldots, a_i \cdot a_{i+1}, \ldots, a_{n+1}) \\
   ~& + (-1)^{n+1} ~ \! f (a_1, \ldots, a_n) \cdot a_{n+1},
\end{align*}
for $f \in \mathrm{Hom}(A^{\otimes n}, A)$ and $a_1, \ldots, a_{n+1} \in A$. Further, there is an associative operation (called the {\em cup-product} operation) $\smile ~ \! ~ \! : \mathrm{Hom}(A^{\otimes m}, A) \times \mathrm{Hom}(A^{\otimes n}, A) \rightarrow \mathrm{Hom}(A^{\otimes m+n}, A)$ given by
\begin{align*}
    (f \smile g) (a_1, \ldots, a_{m+n}) := f(a_1, \ldots, a_m) \cdot g (a_{m+1}, \ldots, a_{m+n})
\end{align*}
which yields a graded Lie bracket $[f, g]_\smile := f \smile g ~ \! - ~ \!  (-1)^{mn}~ \! g \smile f$ on the graded space $\oplus_{n=1}^\infty \mathrm{Hom}(A^{\otimes n} , A)$. With all the above notations, the main result of \cite{baishya-das3} can be stated as follows.

\begin{thm}
    Let $(A, ~ \! \cdot ~ \!)$ be an associative algebra.
    \begin{itemize}
        \item[(i)] For any $m, n \geq 1$, define a bracket $[~,~]_\mathsf{FN} : \mathrm{Hom}(A^{\otimes m}, A) \times \mathrm{Hom}(A^{\otimes n}, A) \rightarrow \mathrm{Hom}(A^{\otimes m+n}, A)$, called the {\bf Fr\"{o}licher-Nijenhuis bracket}, by
        \begin{align*}
            [f, g]_\mathsf{FN} := [f, g]_\smile ~ \!+ ~ \! (-1)^m ~ \! i_{(\delta_{\mathrm{Hoch} } f)} g ~ \! - ~ \! (-1)^{(m+1)n} ~ \!  i_{(\delta_{\mathrm{Hoch} } g)} f,
        \end{align*}
        for $f \in \mathrm{Hom}(A^{\otimes m}, A)$ and $g \in \mathrm{Hom}(A^{\otimes n}, A)$.
        Then $\big(  \oplus_{n=1}^\infty \mathrm{Hom}(A^{\otimes n}, A), [~,~]_\mathsf{FN}   \big)$ is a graded Lie algebra.
        \item[(ii)] A linear map $N : A \rightarrow A$ is a Nijenhuis operator on the associative algebra $(A, ~ \! \cdot ~ \!)$ if and only if $N \in  \mathrm{Hom}(A,A)$ is a Maurer-Cartan element of the graded Lie algebra $\big(  \oplus_{n=1}^\infty \mathrm{Hom}(A^{\otimes n}, A), [~,~]_\mathsf{FN}   \big)$.
    \end{itemize}
\end{thm}

Let $(A, ~ \! \cdot ~ \!)$ be an associative algebra and $N: A \rightarrow A$ be a Nijenhuis operator on it. Note that the linear map $N$ induces a map
\begin{align}
    d_N : \mathrm{Hom}(A^{\otimes n}, A) \rightarrow \mathrm{Hom}(A^{\otimes n+1}, A) ~~~ \text{ given by } ~~~ d_N (f) = [N, f ]_\mathsf{FN}, \text{ for } f \in \mathrm{Hom}(A^{\otimes n}, A).
\end{align}
Explicitly, we have
\begin{align}\label{dn-map1}
    (d_N f) &(a_1, \ldots, a_{n+1}) = N (a_1) \cdot f (a_2, \ldots, a_{n+1}) - (-1)^n ~ \! f (a_1, \ldots, a_n) \cdot N (a_{n+1}) \\
    &+ \sum_{i=1}^n (-1)^{i} ~ \! f (a_1, \ldots, a_{i-1}, N (a_i) \cdot a_{i+1} + a_i \cdot N (a_{i+1}) - N (a_i \cdot a_{i+1}), \ldots, a_{n+1}) \nonumber \\
   & - N \big(  a_1 \cdot f (a_2, \ldots, a_{n+1}) + \sum_{i=1}^n (-1)^i ~ \! f (a_1, \ldots, a_i \cdot a_{i+1}, \ldots, a_{n+1}) + (-1)^{n+1} f (a_1, \ldots, a_n) \cdot a_{n+1}  \big), \nonumber
\end{align}
for $f \in \mathrm{Hom}(A^{\otimes n}, A)$ and $a_1, \ldots, a_{n+1} \in A.$ As $N$ is a Nijenhuis operator (i.e. $[N, N]_\mathsf{FN} = 0$), it turns out that $(d_N)^2 = 0$. We can extend it to a map $d_N : A \rightarrow \mathrm{Hom} (A,A)$ defined by
\begin{align}\label{dn-map2}
    d_N (a) (b) := N (b) \cdot a - a \cdot N (b) - N (b \cdot a - a \cdot b), \text{ for } a, b \in A.
\end{align}
Then it also follows that $\{ \oplus_{n=0}^\infty \mathrm{Hom}(A^{\otimes n}, A), d_N \}$ is a cochain complex which we call the cochain complex associated with the Nijenhuis operator $N$. The corresponding cohomology groups are called the {\em cohomology groups} of the Nijenhuis operator $N$ and are denoted by $H^\bullet (N)$.

\begin{exam}
    Let $(A, ~ \! \cdot ~ \!)$ be any associative algebra. Consider the identity map $\mathrm{Id}_A : A \rightarrow A$ as a Nijenhuis operator on $A$. Then it follows from (\ref{dn-map1}) and (\ref{dn-map2}) that the map $d_{N = \mathrm{Id}_A}$ is the zero map. As a consequence, we get that the cohomology groups of the Nijenhuis operator $N = \mathrm{Id}_A$ are simply given by $H^n (\mathrm{Id}_A) = \mathrm{Hom} (A^{\otimes n}, A)$, for all $n$. This shows that the cohomology of the identity map $\mathrm{Id}_A$ (viewed as a Nijenhuis operator) is independent of the associative multiplication of $A$. This is not surprising as the identity map is a Nijenhuis operator for any associative multiplication.
\end{exam}

\begin{remark}
    It is important to note that the map $d_N$ cannot be expressed as the Hochschild coboundary operator of the deformed associative algebra $A_N = (A, ~ \! \cdot_N ~ \!)$ with coefficients in any bimodule. Equivalently, the cohomology of the Nijenhuis operator $N$ is not the same as the Hochschild cohomology of the deformed associative algebra.
\end{remark}

\begin{remark}
In \cite{das-rota} the present author has introduced the cohomology of a (relative) Rota-Baxter operator while studying its deformations. Let $R: M \rightarrow A$ be a relative Rota-Baxter operator (see Example \ref{exam-rrb} (ii)). Then the cohomology of the operator $R$ is given by the cohomology of the cochain complex $\{ \oplus_{n=0}^\infty \mathrm{Hom} (M^{\otimes n}, A) , d_R \}$, where 
\begin{align*}
d_R (a) (u) :=~& R (u) \cdot a - a \cdot R (u) - R ( u \triangleleft a - a \triangleright u),\\
(d_R f) (u_1, \ldots, u_{n+1}) :=~& R (u_1) \cdot f (u_2, \ldots, u_{n+1}) - (-1)^n ~ \! f (u_1, \ldots, u_n) \cdot R (u_{n+1})\\ 
&+ \sum_{i=1}^n (-1)^i ~ \! f ( u_1, \ldots, u_{i-1}, R (u_i) \triangleright u_{i+1} + u_i \triangleleft R (u_{i+1}), \ldots, u_{n+1}) \\
&- R \big( u_1 \triangleleft f (u_2, \ldots, u_{n+1}) + (-1)^{n+1} ~ \! f (u_1, \ldots, u_n) \triangleright u_{n+1}   \big),
\end{align*}
for $a \in A$, $f \in \mathrm{Hom} (M^{\otimes n}, A)$ and $u, u_1, \ldots, u_{n+1} \in M$. On the other hand, we may consider the (lift) Nijenhuis operator $\widetilde{R} : A \oplus M \rightarrow A \oplus M$ (defined in Example \ref{exam-rrb} (ii)) on the semidirect product algebra and the cochain complex $\{ \oplus_{n=0}^\infty \mathrm{Hom} ( (A \oplus M)^{\otimes n}, A \oplus M ), d_{\widetilde{R}} \}$ for this Nijenhuis operator. Then it is easy to see that 
\begin{align*}
d_{\widetilde{R}} \big(  \mathrm{Hom} (M^{\otimes n}, A)  \big) \subset \mathrm{Hom} (M^{\otimes n+1}, A), \text{ for all } n.
\end{align*}
Hence $\{  \oplus_{n=0}^\infty   \mathrm{Hom} (M^{\otimes n}, A), d_{\widetilde{R}} \}$ is a subcomplex of $\{ \oplus_{n=0}^\infty \mathrm{Hom} ( (A \oplus M)^{\otimes n}, A \oplus M ), d_{\widetilde{R}} \}.$
 Moreover, while restricting to this subcomplex, one can observe that $d_{\widetilde{R}} = d_R$. This shows that the cochain complex of the relative Rota-Baxter operator $R$ can be seen as a subcomplex of the cochain complex of the Nijenhuis operator $\widetilde{R}$.
 \end{remark}
 
 %Here we observe that the above cochain complex of the relative Rota-Baxter operator $R$  First, note that the cochain complex of the Nijenhuis operator is given by $\{ \oplus_{n=0}^\infty \mathrm{Hom} ( (A \oplus M)^{\otimes n}, A \oplus M ), d_{\widetilde{R}} \}.$ 

Let $N$ be a Nijenhuis operator on an associative algebra $(A, ~ \! \cdot ~ \!)$. Note that the differential $d_N$ makes the triple $\big(  \oplus_{n=1}^\infty \mathrm{Hom}(A^{\otimes n}, A), [~,~]_\mathsf{FN}  , d_N \big)$ into a differential graded Lie algebra (dgLa). In the following result, we observe that this dgLa controls the deformations of the Nijenhuis operator $N$.

\begin{prop}
    Let $(A, ~ \! \cdot ~ \!)$ be an associative algebra and $N : A \rightarrow A$ be a Nijenhuis operator on it. Then for any linear map $N' : A \rightarrow A$, the sum $N + N' : A \rightarrow A$ is also a Nijenhuis operator if and only if $N'$ is a Maurer-Cartan element of the dgLa $\big(  \oplus_{n=1}^\infty \mathrm{Hom}(A^{\otimes n}, A), [~,~]_\mathsf{FN}   , d_N \big)$.
\end{prop}

\begin{proof}
    We have
    \begin{align*}
        [N+N', N+N']_\mathsf{FN} =~& [N, N]_\mathsf{FN} + [N, N']_\mathsf{FN} + [N', N]_\mathsf{FN} + [N', N']_\mathsf{FN} \\
        =~& 2 ~ \! [N, N']_\mathsf{FN} + [N', N']_\mathsf{FN} \quad (\because ~ [N, N]_\mathsf{FN} = 0).
    \end{align*}
    This shows that $[N+N', N+N']_\mathsf{FN} = 0$ if and only if $d_N (N') + \frac{1}{2} [N', N']_\mathsf{FN} = 0$. Hence the result follows.
\end{proof}

In the following, we first recall the cohomology of an NS-algebra and find a suitable homomorphism from the cohomology of a Nijenhuis operator $N$ to the cohomology of the induced NS-algebra. As a consequence, we also get a homomorphism from the cohomology of $N$ to the Hochschild cohomology of the deformed associative algebra $A_N = (A, ~ \! \cdot_N ~ \!)$.

\begin{defn}\cite{leroux}
 An {\bf NS-algebra} $(A, \prec, \succ, \curlyvee)$ is a vector space $A$ endowed with three bilinear operations $\prec, \succ, \curlyvee : A \times A \rightarrow A$ that satisfy the following set of identities:
    \begin{align*}
        (a \prec b) \prec c =~& a \prec (b \prec c + b \succ c + b \curlyvee c),\\
        (a \succ b) \prec c =~& a \succ (b \prec c), \\
        (a \prec b + a \succ b + a \curlyvee b) \succ c =~& a \succ (b \succ c),\\
        ( a \prec b + a \succ b + a \curlyvee b) \curlyvee c + (a \curlyvee b) \prec c = ~& a \succ (b \curlyvee c ) + a \curlyvee (b \prec c + b \succ c + b \curlyvee c) , \text{ for } a, b, c \in A.
    \end{align*}
\end{defn}

\begin{remark}
    In \cite{loday-di} Loday introduced the notion of a dendriform algebra in his study of the algebraic structure behind planar binary trees. 
    %Recall that a dendriform algebra $(A, \prec, \succ)$ is a vector space $A$ with binary operations $\prec, \succ : A \times A \rightarrow A$ satisfying
    %\begin{align*}
     %    (a \prec b) \prec c = a \prec (b \prec c + b \succ c ), \quad  (a \succ b) \prec c = a \succ (b \prec c), \quad  (a \prec b + a \succ b ) \succ c =~& a \succ (b \succ c),
    %\end{align*}
    %for $a, b, c \in A$. 
    Note that an NS-algebra $(A, \prec, \succ, \curlyvee)$ for which $\curlyvee = 0$ is nothing but a dendriform algebra.
\end{remark}

Let $(A, \prec, \succ, \curlyvee)$ is an NS-algebra. Then $A$ can be given an associative algebra structure with the multiplication
\begin{align*}
a * b := a \prec b + a \succ b + a \curlyvee b, \text{ for } a, b \in A.
\end{align*}
The associative algebra $(A, ~ \! * ~ \!)$ is said to be the {\em `total associative algebra'} of the given NS-algebra. On the other hand, if $(A, ~ \cdot ~ \!)$ is an associative algebra and $N : A \rightarrow A$ is a Nijenhuis operator on it, then the vector space $A$ inherits an NS-algebra structure with the operations 
\begin{align*}
    a \prec_N b := a \cdot N (b), \quad a \succ_N b := N (a) \cdot b ~~~~ \text{ and } ~~~~ a \curlyvee_N b := - N (a \cdot b), \text{ for } a, b \in A.
\end{align*}
The NS-algebra $(A, \prec_N, \succ_N,\curlyvee_N)$ is denoted by $A_\mathrm{NS}$ and is said to be induced by the Nijenhuis operator $N$. Observe that the corresponding total associative algebra is precisely the deformed algebra $A_N = (A, ~ \! \cdot_N ~ \!)$ given in (\ref{deformed-alg}).

\medskip

Let $C_n$ be the set of first $n$ natural numbers. Since we will treat the elements of $C_n$ as certain symbols, we write $C_n = \{ [1], [2], \ldots, [n] \}$ for conventional purposes. Let $A$ be any vector space (not necessarily having any additional structure). Then for any $n \geq 1$, we define a linear space $\mathcal{O}_A (n)$ by
\begin{align*}
    \mathcal{O}_A (1) := \mathrm{Hom} ( {\bf k} [C_1] \otimes A, A) \cong \mathrm{Hom} (A, A) \quad \text{ and } \quad \mathcal{O}_A (n) = \mathrm{Hom} ( {\bf k} [C_{n+1}] \otimes A^{\otimes n}, A), \text{ for } n \geq 2.
\end{align*}
Given any $f \in \mathcal{O}_A (m)$, $g \in \mathcal{O}_A (n)$ and $1 \leq i \leq m$, we define an element (called the {\em partial composition}) $f \circ_i g \in \mathcal{O}_A (m+n-1)$ by

\begin{align}
    &(f \circ_i g) ([r]; a_1, \ldots, a_{m+n-1}) \label{ns-circ} \\ \medskip
    &= \begin{cases} \medskip
        f (  {\scriptstyle [r]} ; a_1, \ldots, a_{i-1} , g ( {\scriptstyle [1] + \cdots + [n+1]} ; a_i, \ldots, a_{i+n-1}) , a_{i+n}, \ldots, a_{m+n-1}) & \text{ if } 1 \leq r \leq i-1, \\
        \medskip
        f ( {\scriptstyle [i]}; a_1, \ldots, a_{i-1} , g ( {\scriptstyle [r-i+1]}; a_i, \ldots, a_{i+n-1}), a_{i+n}, \ldots, a_{m+n-1})  & \text{ if } i \leq r \leq i+n -1, \\
        \medskip
        f ( {\scriptstyle [r-n+1]}; a_1, \ldots, a_{i-1}, g(  {\scriptstyle [1] + \cdots + [n+1] ; a_i, \ldots, a_{i+n-1}}), a_{i+n}, \ldots, a_{m+n-1} ) & \text{ if } {i+n \leq r \leq m+n-1},\\
        \medskip
        f ( {\scriptstyle [i]} ; a_1, \ldots, a_{i-1} , g ( {\scriptstyle [n+1]}; a_i, \ldots, a_{i+n-1}), a_{i+n}, \ldots, a_{m+n-1}) &\\
         ~~ + f ( {\scriptstyle [m+1]}; a_1, \ldots, a_{i-1} , g ({\scriptstyle [1] + \cdots + [n+1]}; a_i, \ldots, a_{i+n-1}), a_{i+n}, \ldots, a_{m+n-1}) &  \text{ if } r = m+n.
    \end{cases} \nonumber
\end{align}

\medskip

\noindent Then it has been shown in \cite[Theorem 5.1]{das-ns} that the collection of spaces $\{ \mathcal{O}_A (n ) \}_{n \geq 1}$ endowed with the above partial compositions forms a nonsymmetric operad. This in particular implies that the shifted graded space $\mathcal{O}_A (\bullet + 1) = \oplus_{n=0}^\infty \mathcal{O}_A (n+1)$ carries a graded Lie bracket
\begin{align*}
    \llbracket f, g \rrbracket := \sum_{i=1}^{m} (-1)^{(i-1) (n-1)} ~ \! f \circ_i g - (-1)^{(m-1)(n-1)} \sum_{i=1}^{n} (-1)^{(i-1) (m-1)} ~ \! g \circ_i f,
\end{align*}
for $f \in \mathcal{O}_A (m)$ and $g \in \mathcal{O}_A (n)$. Then an NS-algebra structure on the vector space $A$ is equivalent to having a Maurer-Cartan element of the graded Lie algebra $ \big(  \mathcal{O}_A (\bullet + 1), \llbracket ~, ~ \rrbracket \big)$. Explicitly, let the vector space $A$ be endowed with three bilinear operations $\prec, \succ, \curlyvee : A \times A \rightarrow A$. Then we define an element $\pi \in \mathcal{O}_A (2) = \mathrm{Hom} ({\bf k} [C_3] \otimes A^{\otimes 2}, A)$ by
\begin{align*}
    \pi ([1] ; a, b) = a \prec b, \quad \pi ([2]; a, b) = a \succ b ~~~~~ \text{ and } ~~~~~ \pi ([3]; a, b) = a \curlyvee b, \text{ for } a, b \in A.
\end{align*}
Then $(A, \prec, \succ, \curlyvee)$ is an NS-algebra if and only if $\pi \in \mathcal{O}_A (2)$ is a Maurer-Cartan element of the graded Lie algebra $\big(  \mathcal{O}_A (\bullet + 1), \llbracket ~, ~ \rrbracket \big)$. In this case, the Maurer-Cartan element $\pi$ induces a map 
\begin{align*}
    \delta_\pi : \mathcal{O}_A (n) \rightarrow \mathcal{O}_A (n+1) ~~~~ \text{ given by } ~~~~ \delta_\pi (f) := (-1)^{n-1} ~ \! \llbracket \pi, f \rrbracket, \text{ for } f \in \mathcal{O}_A (n). 
\end{align*}
Then it follows that $\{ \oplus_{n=1}^\infty \mathcal{O}_A (n) , \delta_\pi \}$ is a cochain complex. The corresponding cohomology groups are called the cohomologies of the NS-algebra $(A, \prec, \succ, \curlyvee)$ and they are denoted by $H^\bullet_\mathrm{NS} (A).$ 

\medskip

Next, let $(A, ~ \! \cdot ~ \!)$ be an associative algebra and $N : A \rightarrow A$ be a Nijenhuis operator on it. Then we have seen that $A_\mathrm{NS} = (A, \prec_N, \succ_N, \curlyvee_N)$ is an NS-algebra. Let $\pi_N \in \mathcal{O}_A (2) = \mathrm{Hom} ( {\bf k} [C_3]\otimes A^{\otimes 2}, A)$ be the corresponding Maurer-Cartan element in the graded Lie algebra $\big(  \mathcal{O}_A (\bullet + 1), \llbracket ~, ~ \rrbracket \big)$. In other words, $\pi_N ([1]; a, b) = a \cdot N(b)$, $\pi_N ([2]; a, b) = N (a) \cdot b$ and $\pi_N ([3]; a, b) = - N (a\cdot b)$, for $a, b \in A$. For each $n \geq 1$, we now define a map
\begin{align*}
    &\Theta_n : \mathrm{Hom} (A^{\otimes n}, A) \longrightarrow \mathcal{O}_A (n+1) = \mathrm{Hom} ({\bf k} [C_{n+2}] \otimes A^{\otimes n+1}, A) ~~~~ \text{ by } \\ \medskip
    \Theta_n (f)([r];~& a_1, \ldots, a_{n+1}) := \begin{cases}
        (-1)^{n+1} ~ \! a_1 \cdot f (a_2, \ldots, a_{n+1}) & \text{ if } r = 1,\\ \medskip
         \qquad  0 & \text{ if } 2 \leq r \leq n,\\ \medskip
        f (a_1, \ldots, a_{n+1}) \cdot a_{n+1} & \text{ if } r = n+1,\\ \medskip
        \sum_{i=1}^n (-1)^{n+i+1} ~ \! f (a_1, \ldots, a_i \cdot a_{i+1}, \ldots, a_{n+1}) & \text{ if } r = n+2,
    \end{cases}
\end{align*}
for $f \in \mathrm{Hom}(A^{\otimes n}, A)$, $[r] \in C_{n+2}$ and $a_1, \ldots
, a_{n+1} \in A$. Then it is easy to observe that $\Theta_1 (N) = \pi_N.$ Moreover, we have the following result.

\begin{prop}\label{prop-nij-ns}
    With the above notations, the collection of maps $\{ \Theta_n \}_{n \geq 1}$ describes a homomorphism from the reduced cochain complex $\{ \oplus_{n=1}^\infty \mathrm{Hom} (A^{\otimes n}, A) , d_N \}$ of the Nijenhuis operator $N$ to the shifted cochain complex $\{ \oplus_{n=1}^\infty \mathcal{O}_A (n+1) , \delta_{\pi_N} \}$ of the induced NS-algebra $A_\mathrm{NS} = (A, \prec_N, \succ_N, \curlyvee_N)$. As a result, there is a homomorphism $\Theta_\bullet : H^\bullet (N) \rightarrow H^{\bullet +1}_\mathrm{NS} (A_\mathrm{NS})$ between the corresponding cohomology groups.
\end{prop}

\begin{proof}
    For any $f \in \mathrm{Hom} (A^{\otimes n}, A)$ and $a_1, \ldots, a_{n+1} \in A$, we observe that
    \begin{align*}
        &\llbracket \pi_N, \Theta_n (f) \rrbracket ([1]; a_1, \ldots, a_{n+2}) \\
        &= \pi_N \big( [1] ; \Theta_n (f) ([1]; a_1, \ldots, a_{n+1}), a_{n+2} \big) + (-1)^n ~ \! \pi_N \big( [1]; a_1, \Theta_n (f) ( {\scriptstyle  [1] + \cdots + [n+2]}; a_2, \ldots, a_{n+2})  \big)\\
& \qquad - (-1)^n ~ \! \Theta_n (f) \big( [1]; \pi_N ([1]; a_1, a_2), a_3, \ldots, a_{n+2} \big) \\ 
& \qquad - (-1)^n ~ \! \sum_{i=2}^{n+1} (-1)^{i-1} ~ \! \Theta_n (f) \big( [1]; a_1, \ldots, a_{i-1}, \pi_N ( {\scriptstyle  [1] + [2] + [3]}; a_i, a_{i+1}), \ldots, a_{n+2} \big)\\
        &= (-1)^{n+1} ~ \! a_1 \cdot f (a_2, \ldots, a_{n+1}) \cdot N (a_{n+2}) + (-1)^n ~ \! a_1 \cdot N \big(  (-1)^{n+1} (\delta_\mathrm{Hoch} f) (a_2, \ldots, a_{n+2})    \big) \\
    & \qquad + a_1 \cdot N (a_2) \cdot f (a_3, \ldots, a_{n+2}) + \sum_{i=2}^{n+1} (-1)^{i-1} ~ \! a_1 \cdot f (a_2, \ldots, a_{i-1}, a_i \cdot_N a_{i+1}, \ldots, a_{n+2}) \\
    &= a_1 \cdot \big( [N, f]_\mathsf{FN} (a_2, \ldots, a_{n+2}) \big) \\
        &= (-1)^n ~ \!(\Theta_{n+1} [N, f]_\mathsf{FN}) ([1]; a_1, \ldots, a_{n+2}).
    \end{align*}
For any $2 \leq r \leq n+1$, one can easily observe from the definitions that 
    \begin{align*}
    \llbracket \pi_N, \Theta_n (f) \rrbracket ([r]; a_1, \ldots, a_{n+2}) = 0 =  (-1)^n ~ \! (\Theta_{n+1} [N, f]_\mathsf{FN}) ([r]; a_1, \ldots, a_{n+2}).
    \end{align*}
We also have 
    \begin{align*}
         & \llbracket \pi_N, \Theta_n (f) \rrbracket ([n+2]; a_1, \ldots, a_{n+2}) \\
         &= \pi_N \big( [2]; \Theta_n (f) ( {\scriptstyle  [1] +\cdots + [n+2]}; a_1, \ldots, a_{n+1}), a_{n+2}   \big) + (-1)^n ~ \! \pi_N \big( [2]; a_1, \Theta_n (f) ([n+1]; a_2, \ldots, a_{n+2})  \big) \\
         & \qquad -(-1)^n \sum_{i=1}^n (-)^{i-1} ~ \! \Theta_n (f) \big( [n+1]; a_1, \ldots, a_{i-1}, \pi_N ( {\scriptstyle [1]+[2]+[3]}; a_i, a_{i+1}), \ldots, a_{n+2} \big) \\
         & \qquad - \Theta_n (f) \big( [n+1]; a_1, \ldots, a_n, \pi_N ([2]; a_{n+1}, a_{n+2})  \big) \\
         &= (-1)^{n+1} ~ \! N \big(  (\delta_\mathrm{Hoch} f) (a_1, \ldots, a_{n+1})  \big) \cdot a_{n+2} + (-1)^n ~ \! N (a_1) \cdot f (a_2, \ldots, a_{n+1}) \cdot a_{n+2} \\
         & \qquad + (-1)^n \sum_{i=1}^n (-1)^i ~ \! f (a_1, \ldots, a_{i-1}, a_i \cdot_N a_{i+1}, \ldots, a_{n+1}) \cdot a_{n+2} - f(a_1, \ldots, a_n) \cdot N (a_{n+1}) \cdot a_{n+2} \\
         &= (-1)^n \big(  [N, f]_\mathsf{FN} (a_1, \ldots, a_{n+1}) \big) \cdot a_{n+2} \\
         &= (-1)^n ~ \! (\Theta_{n+1} [N, f]_\mathsf{FN}) ([n+2]; a_1, \ldots, a_{n+2}).
    \end{align*}
Finally, it is also tedious but straightforward to see that
    \begin{align*}
          \llbracket \pi_N, \Theta_n (f) \rrbracket ([n+3]; a_1, \ldots, a_{n+2})
         = (-1)^n ~ \! (\Theta_{n+1} [N, f]_\mathsf{FN}) ([n+3]; a_1, \ldots, a_{n+2}).
    \end{align*}
    Combining all these cases, we obtain that $\llbracket \pi_N, \Theta_n (f) \rrbracket = (-1)^n ~ \! \Theta_{n+1} ([N, f]_\mathsf{FN} )$. Hence
    \begin{align*}
        (\delta_{\pi_N} \circ \Theta_n) (f) = (-1)^n ~ \! \llbracket \pi_N, \Theta_n (f) \rrbracket = \Theta_{n+1} [N, f]_\mathsf{FN} = (\Theta_{n+1} \circ d_N) (f).
    \end{align*}
    This concludes the proof.
\end{proof}

It has been shown in \cite[Theorem 5.4]{das-ns} that there is a homomorphism from the cochain complex of an NS-algebra to the Hochschild cochain complex of the corresponding total associative algebra with coefficients in the adjoint bimodule. Combining this fact with Proposition \ref{prop-nij-ns}, we get the following.

\begin{thm}
    Let $(A, ~ \! \cdot ~ \!)$ be an associative algebra and $N: A \rightarrow A$ be a Nijenhuis operator on it. Then the collection of maps $\{    \Psi_n : \mathrm{Hom} (A^{\otimes n}, A) \rightarrow \mathrm{Hom} (A^{\otimes n+1}, A) \}_{n \geq 1}$ given by
    \begin{align*}
        \Psi_n (f) (a_1, \ldots, a_{n+1} ) := (\Theta_n (f)) ({\scriptstyle [1] + \cdots + [n+2]}; a_1, \ldots, a_{n+1}) = (-1)^{n+1} ~ \! (\delta_\mathrm{Hoch} f) (a_1, \ldots, a_{n+1})
    \end{align*}
    induces a homomorphism $\Psi_\bullet : H^\bullet (N) \rightarrow H^{\bullet + 1}_\mathrm{Hoch} (A_N; A_N)$ from the cohomology of the Nijenhuis operator $N$ to the Hochschild cohomology of the deformed associative algebra $A_N$.
\end{thm}

\section{Cohomology of Nijenhuis algebras}\label{sec3}
In this section, we introduce and study the cohomology of Nijenhuis algebras. Among others, we find a long exact sequence connecting the cohomology groups of a Nijenhuis algebra and the cohomologies of the underlying associative algebra and the Nijenhuis operator.

\begin{defn}
    A {\bf Nijenhuis algebra} is an associative algebra $(A, ~ \! \cdot ~ \!)$ endowed with a distinguished Nijenhuis operator $N: A \rightarrow A$ on it.
\end{defn}

A Nijenhuis algebra as above is denoted by the pair $( (A, ~ \! \cdot ~ \!), N)$ or simply by $(A,N)$ when the associative multiplication of $A$ is clear from the context.

Let $( (A, ~ \! \cdot ~ \!), N)$ and $( (A', ~ \! \cdot' ~ \!), N')$ be two Nijenhuis algebras. A {\em homomorphism} of Nijenhuis algebras from $( (A, ~ \! \cdot ~ \!), N)$ to $( (A', ~ \! \cdot' ~ \!), N')$ is an associative algebra morphism $\varphi : A \rightarrow A'$ satisfying $N' \circ \varphi = \varphi \circ N$. Further, it is said to be an isomorphism of Nijenhuis algebras if $\varphi$ is additionally bijective. Given a Nijenhuis algebra $( (A, ~ \! \cdot ~ \!), N)$, we denote the group of all Nijenhuis algebra automorphisms (self isomorphisms) of $( (A, ~ \! \cdot ~ \!), N)$ by the notation $\mathrm{Aut}(A, N).$

\medskip

Let $( (A, ~ \! \cdot ~ \!), N)$ be a Nijenhuis algebra. A {\bf Nijenhuis bimodule} over $( (A, ~ \! \cdot ~ \!), N)$ is given by a pair $((M, \triangleright, \triangleleft), N_M)$ consisting of an $A$-bimodule $(M, \triangleright, \triangleleft)$ with a linear map $N_M: M \rightarrow M$ satisfying
\begin{align*}
    N(a) \triangleright N_M (u) =~& N_M \big(  N (a) \triangleright u + a \triangleright N_M (u) - N_M (a \triangleright u)   \big),\\
    N_M (u) \triangleleft N (a) =~& N_M \big( N_M (u) \triangleleft a + u \triangleleft N (a) - N_M (u \triangleleft a) \big),
\end{align*}
for all $a \in A$ and $u \in M$.

\begin{exam}
    Let $( (A, ~ \! \cdot ~ \!), N)$ be any Nijenhuis algebra. Then there is a natural Nijenhuis bimodule $((A, \triangleright_\mathrm{ad},\triangleleft_\mathrm{ad}), N)$ over the Nijenhuis algebra $( (A, ~ \! \cdot ~ \!), N)$, where $a \triangleright_\mathrm{ad} b = a \triangleleft_\mathrm{ad} b = a \cdot b$ (for $a, b \in A$) defines the adjoint $A$-bimodule structure on the vector space $A$. This is called the adjoint Nijenhuis bimodule.
\end{exam}

\begin{exam} Let $(A, ~ \! \cdot ~ \! )$ be an associative algebra and $(M, \triangleright, \triangleleft)$ be any $A$-bimodule.
    \begin{itemize}
        \item[(i)]  Then $((M, \triangleright, \triangleleft), \mathrm{Id}_M)$ is a Nijenhuis bimodule over the Nijenhuis algebra $((A, ~ \! \cdot ~ \! ), \mathrm{Id}_A)$.
        \item[(ii)]  For any $x \in A$, consider the Nijenhuis algebra $( (A, ~ \! \cdot ~ \!), N = l_x)$. Then it can be shown that $( (M, \triangleright, \triangleleft), N_M = x \triangleright - )$ is a Nijenhuis bimodule over the Nijenhuis algebra $( (A, ~ \! \cdot ~ \!), N = l_x)$. Similarly, the pair $( (M, \triangleright, \triangleleft), N_M = - \triangleleft x)$ is a Nijenhuis bimodule over the Nijenhuis algebra $( (A, ~ \! \cdot ~ \!), N = r_x)$.
    \end{itemize}
\end{exam}

\begin{exam}\label{kth-power}
   Let $((A, ~ \! \cdot ~ \! ), N)$ be a Nijenhuis algebra and $((M, \triangleright, \triangleleft), N_M)$ be a Nijenhuis bimodule over it. Then for any $k \geq 0$, the pair $((M, \triangleright, \triangleleft), (N_M)^k )$ is a Nijenhuis bimodule over the Nijenhuis algebra $((A, ~ \! \cdot ~ \! ), N^k)$.
\end{exam}

\begin{exam}
   Let $((A, ~ \! \cdot ~ \! ), N)$ and $((A', ~ \! \cdot' ~ \! ), N')$ be two Nijenhuis algebras and $\varphi : A \rightarrow A'$ be a homomorphism of Nijenhuis algebras. Then $((A' , \triangleright, \triangleleft), N')$ is a Nijenhuis bimodule over the Nijenhuis algebra $((A, ~ \! \cdot ~ \! ), N)$, where $\triangleright : A \times A' \rightarrow A'$ and $\triangleleft : A' \times A \rightarrow A'$ are respectively given by $a \triangleright u = \varphi (a) \cdot' u$ and $u \triangleleft a = u \cdot' \varphi (a)$, for $a \in A$ and $u \in A'$.
\end{exam}

\begin{exam} (Lifting of Rota-Baxter bimodules) A {\em Rota-Baxter algebra} is a pair $((A, ~ \! \cdot ~ \! ), R)$ consisting of an associative algebra $(A, ~ \! \cdot ~ \!)$ with a distinguished Rota-Baxter operator $R : A \rightarrow A$. Then we have seen that the map $R$ lifts to a Nijenhuis operator on the semidirect product algebra. In other words, $((A \oplus A, ~ \! \cdot_\ltimes ~ \!), \widetilde{R})$ is a Nijenhuis algebra, where
\begin{align*}
    (a, b) \cdot_\ltimes (c, d) = (a \cdot c ~ \!, ~ \! a \cdot d + b \cdot c) ~~~~ \text{ and } ~~~~ \widetilde{R} (a, b) = (R(b), 0),
\end{align*}
for $(a, b), (c, d) \in A \oplus A$. In \cite{lin} the authors have considered Rota-Baxter bimodules over a Rota-Baxter algebra and obtained some representation theoretic results. Recall that a Rota-Baxter bimodule over the Rota-Baxter algebra $((A, ~ \! \cdot ~ \! ), R)$ is a pair $((M, \triangleright, \triangleleft), R_M)$ consisting of an $A$-bimodule $(M, \triangleright, \triangleleft)$ and a linear map $R_M : M \rightarrow M$ satisfying
\begin{align*}
    R(a) \triangleright R_M (u) = R_M ( R(a) \triangleright u + a \triangleright R_M (u)) ~~~ \text{ and } ~~~ R_M (u) \triangleleft R(a) = R_M ( R_M (u) \triangleleft a + u \triangleleft R(a)),
\end{align*}
for all $a \in A$ and $u \in M$. Then it can be checked that $((M \oplus M,  \triangleright_\ltimes , \triangleleft_\ltimes), \widetilde{R_M} )$ is a Nijenhuis bimodule over the Nijenhuis algebra $((A \oplus A, ~ \! \cdot_\ltimes ~ \!), \widetilde{R})$, where 
\begin{align*}
    (a, b) \triangleright_\ltimes (u, v) := (a \triangleright u ~ \!, ~ ~& \! a \triangleright v + b \triangleright u), \qquad (u, v) \triangleleft_\ltimes (a, b) := (u \triangleleft a ~ \!, ~ \! u \triangleleft b + v \triangleleft a) \\
   & \text{ and } ~~ \widetilde{R_M} (u, v) = (R_M (v), 0),
\end{align*}
for $(a, b) \in A \oplus A$ and $(u, v) \in M \oplus M$. Thus, a Rota-Baxter bimodule can be lifted to a Nijenhuis bimodule.
\end{exam}

Let $((A, ~ \! \cdot ~ \! ), N)$ be a Nijenhuis algebra and $(M, \triangleright, \triangleleft)$ be a given $A$-bimodule. A linear map $\beta : M \rightarrow M$ is said to be {\em admissible} \cite{ma} for the above $A$-bimodule if
\begin{align*}
    \beta (N(a) \triangleright u) + a \triangleright \beta^2 (u) =~& N(a) \triangleright \beta (u) + \beta (a \triangleright \beta (u)),\\
    \beta (u \triangleleft N(a)) + \beta^2 (u) \triangleleft a =~& \beta (u) \triangleleft N(a) + \beta (\beta (u) \triangleleft a),
\end{align*}
for all $a \in A$ and $u \in M$. Admissible maps play a prominent role in studying bialgebra theory for Nijenhuis algebras. A Nijenhuis algebra $((A, ~ \! \cdot ~ \! ), N)$ is said to be {\em admissible} if the map $N$ itself is admissible for the adjoint $A$-bimodule. The following result shows the importance of the admissible property (see \cite{ma}).

\begin{prop}
    Let $((A, ~ \! \cdot ~ \! ), N)$ be a Nijenhuis algebra and $(M, \triangleright, \triangleleft)$ be a given $A$-bimodule. If a linear map $\beta : M \rightarrow M$ is admissible for the above $A$-bimodule then the pair $((M^*, \triangleright_*, \triangleleft_*), \beta^*)$ is a Nijenhuis bimodule over the Nijenhuis algebra $((A, ~ \! \cdot ~ \!), N)$, where
    \begin{align*}
         (a \triangleright_* f) (u) = f (u \triangleleft a) ~~~~ \text{ and } ~~~~ (f \triangleleft_* a) (u) = f (a \triangleright u), \text{ for } a \in A, f \in M^*, u \in M.
    \end{align*}

    In particular, if $((A, ~ \! \cdot ~ \! ), N)$ is an admissible Nijenhuis algebra then $((A^*, \triangleright_\mathrm{coad}, \triangleleft_\mathrm{coad}), N^*)$ is a Nijenhuis bimodule over it, where $(a \triangleright_\mathrm{coad} f) (b) = f (b \cdot a)$ and $(f \triangleleft_\mathrm{coad} a) (b) = f (a \cdot b)$, for $a, b \in A$, $f \in A^*$.
\end{prop}

\begin{prop}\label{prop-semid}
    Let $((A, ~ \! \cdot ~ \! ), N)$ be a Nijenhuis algebra and $((M, \triangleright, \triangleleft), N_M)$ be a Nijenhuis bimodule over it. Then the pair $((A \oplus M, ~ \! \cdot_\ltimes ~ \! ), N \oplus N_M )$ is also a Nijenhuis algebra, where $\cdot_\ltimes$ is the semidirect product given in (\ref{semid}). Moreover, the inclusion map $A \hookrightarrow A \oplus M,~ a \mapsto (a, 0)$ is a homomorphism of Nijenhuis algebras from $((A, ~ \! \cdot ~ \! ), N)$ to the semidirect product $((A \oplus M, ~ \! \cdot_\ltimes ~ \! ), N \oplus N_M )$.
\end{prop}

In the next, we show that a Nijenhuis bimodule over a Nijenhuis algebra induces a bimodule over the deformed associative algebra. The explicit result is given by the following whose proof is straightforward.

\begin{prop}
    Let $((A, ~ \! \cdot ~ \!), N)$ be a Nijenhuis algebra and $((M , \triangleright, \triangleleft), N_M)$ be a Nijenhuis bimodule. Define bilinear maps $\triangleright_{N, N_M} : A \times M \rightarrow M$ and $\triangleleft_{N, N_M} : M \times A \rightarrow M$ by
    \begin{align*}
        a \triangleright_{N, N_M} u := N (a) \triangleright u + a \triangleright N_M (u) - N_M (a \triangleright u) ~~~~ \text{ and } ~~~~ u \triangleleft_{N, N_M} a := N_M (u) \triangleleft a + u \triangleleft N(a) - N_M (u \triangleleft a),
    \end{align*}
    for $a \in A$ and $u \in M$. Then $(M, \triangleright_{N, N_M}, \triangleleft_{N, N_M})$ is a bimodule of the deformed associative algebra $A_N  =  (A, ~ \! \cdot_N  ).$ More generally, for any $k \geq 0$, the triple $(M, \triangleright_{N^k, N_M^k} , \triangleleft_{N^k, N_M^k})$ is a bimodule of the algebra $A_{N^k} = (A, ~ \! \cdot_{N^k} ).$ 
\end{prop}

%Note that the second part of the above proposition follows from the first part applying on Example \ref{kth-power}. 

We will now define the cohomology of a Nijenhuis algebra with coefficients in a Nijenhuis bimodule. Let $((A, ~ \! \cdot ~ \! ), N)$ be a Nijenhuis algebra and $((M, \triangleright, \triangleleft), N_M)$ be a given Nijenhuis bimodule. At first, since $(M, \triangleright, \triangleleft)$ is an $A$-bimodule, we may consider the Hochschild cochain complex $\{ \oplus_{n=0}^\infty C^n_\mathrm{Hoch} (A; M) , \delta_\mathrm{Hoch} \}$ of the associative algebra $(A, ~ \! \cdot ~ \!)$ with coefficients in the bimodule $(M, \triangleright, \triangleleft)$. More precisely, we have $C^n_\mathrm{Hoch} (A;M) := \mathrm{Hom} (A^{\otimes n}, M)$ and $\delta_\mathrm{Hoch} : C^n_\mathrm{Hoch} (A;M) \rightarrow C^{n+1}_\mathrm{Hoch} (A;M)$ is given by
\begin{align*}
    (\delta_\mathrm{Hoch} f) (a_1, \ldots, a_{n+1}) =~& a_1 \triangleright f (a_2, \ldots, a_{n+1}) + \sum_{i=1}^n (-1)^i ~ \! f (a_1, \ldots, a_i \cdot a_{i+1}, \ldots, a_{n+1}) \\
   ~& + (-1)^{n+1} ~ \! f (a_1, \ldots, a_n) \triangleleft a_{n+1},
\end{align*}
for $f \in C^n_\mathrm{Hoch} (A; M)$ and $a_1, \ldots, a_{n+1} \in A$. We denote the corresponding Hochschild cohomology groups by $H^\bullet_\mathrm{Hoch} (A; M)$. On the other hand, for each $n \geq 0$, we first set
\begin{align*}
    \qquad C^n (N; N_M) := \mathrm{Hom} (A^{\otimes n}, M) ~~ \text{ and define a map }~ d_{N, N_M} : C^n (N; N_M) \rightarrow C^{n+1} (N; N_M) \text{ by }
   \end{align*}
   \begin{align}
   & \qquad \qquad \quad d_{N, N_M} (u) (a) := N(a) \triangleright u - u \triangleleft N(a) - N_M (a \triangleright u - u \triangleleft a), \label{dnnm}\\
   \medskip &(d_{N, N_M } f) (a_1, \ldots, a_{n+1}) = N (a_1) \triangleright f (a_2, \ldots, a_{n+1}) - (-1)^n ~ \! f (a_1, \ldots, a_n) \triangleleft N (a_{n+1}) \label{dnnm-2} \\
   & \quad + \sum_{i=1}^n (-1)^{i} ~ \! f (a_1, \ldots, a_{i-1}, N (a_i) \cdot a_{i+1} + a_i \cdot N (a_{i+1}) - N (a_i \cdot a_{i+1}), \ldots, a_{n+1}) \nonumber  \\
   & \quad - N_M \big(  a_1 \triangleright f (a_2, \ldots, a_{n+1}) + \sum_{i=1}^n (-1)^i ~ \! f (a_1, \ldots, a_i \cdot a_{i+1}, \ldots, a_{n+1}) + (-1)^{n+1} f (a_1, \ldots, a_n) \triangleleft a_{n+1}  \big), \nonumber
\end{align}
for $u \in M$, $f \in C^{n \geq 1} (N; N_M)$ and $a, a_1, \ldots, a_{n+1} \in A$. Then we have the following.

\begin{prop}
    The map $d_{N, N_M}$ is a differential. In other words, $\{ \oplus_{n=0}^\infty C^n (N; N_M), d_{N, N_M} \}$ is a cochain complex.
\end{prop}

\begin{proof}
    Since $((A, ~ \! \cdot ~ \!), N)$ is a Nijenhuis algebra and $((M, \triangleright, \triangleleft), N_M)$ is a Nijenhuis bimodule, it follows from Proposition \ref{prop-semid} that $((A \oplus M, ~ \! \cdot_\ltimes ~ \!), N \oplus N_M)$ is a Nijenhuis algebra. Hence one may define the cochain complex $\{  \oplus_{n=0}^\infty \mathrm{Hom} ((A \oplus M)^{\otimes n}, A \oplus M ), d_{N \oplus N_M} \}$ associated to the Nijenhuis operator $N \oplus N_M$ on the semidirect product associative algebra $(A \oplus M , ~ \! \cdot_\ltimes ~ \!)$. Then it is straightforward to verify that $d_{N \oplus N_M} \big(  \mathrm{Hom} (A^{\otimes n}, M)   \big) \subset  \mathrm{Hom} (A^{\otimes n+1}, M) $, for all $n$. Moreover, the restriction of $d_{N \oplus N_M}$ on the space $ \mathrm{Hom} (A^{\otimes n}, M) $ coincides with the map $d_{N, N_M}$ given in (\ref{dnnm}), (\ref{dnnm-2}). Hence we get that $(d_{N, N_M})^2 = 0$.
\end{proof}

The cochain complex $\{ \oplus_{n=0}^\infty C^n (N; N_M), d_{N, N_M} \}$ obtained in the above proposition can be regarded as the cochain complex associated with the Nijenhuis operator $N$ relative to the operator $N_M$. We denote the corresponding cohomology groups by $H^\bullet (N; N_M)$. Note that when $ ((M, \triangleright, \triangleleft), N_M) = ((A, \triangleright_\mathrm{ad},  \triangleleft_\mathrm{ad}), N)$ is the adjoint Nijenhuis bimodule, the above cochain complex (or the cohomology) coincides with the complex (or the cohomology) associated to the Nijenhuis operator $N$.

\medskip

For each $n \geq 0$, we now define a map $\partial^{N, N_M} : C^n_\mathrm{Hoch} (A; M) \rightarrow C^n (N; N_M)$ by
\begin{align}
&\qquad \qquad  \partial^{N, N_M} (u) = u, \text{ for } u \in C^0_\mathrm{Hoch} (A; M) = M, \nonumber \\
    \partial^{N, N_M} (f) &(a_1, \ldots, a_n) = f \big(  N (a_1), \ldots, N (a_n)  \big) - \sum_{i=1}^n N_M \big( f \big(  N (a_1), \ldots, a_i, \ldots, N (a_n)   \big)    \big) \\
    &+ \sum_{1 \leq i < j \leq n} N_M^2 \big( f \big(  N (a_1), \ldots, a_i, \ldots, a_j, \ldots, N (a_n)   \big)    \big) - \cdots \cdot + (-1)^n ~ \! N_M^n \big( f (a_1, \ldots, a_n)   \big), \nonumber
\end{align}
for $f \in C^{n \geq 1}_\mathrm{Hoch} (A; M) = \mathrm{Hom}(A^{\otimes n}, M)$ and $a_1, \ldots, a_n \in A$. Then we have the following result whose proof consists of only lengthy calculations.

\begin{prop}\label{prop-mapp}
    For any $f \in  C^n_\mathrm{Hoch} (A; M)$, $n \geq 0$, we have $(d_{N, N_M} \circ \partial^{N, N_M}) (f) = (\partial^{N, N_M} \circ \delta_\mathrm{Hoch})(f)$.
\end{prop}

It follows from the above proposition that the collection $\{ \partial^{N, N_M} : C^n_\mathrm{Hoch} (A; M) \rightarrow C^n (N; N_M) \}_{n \geq 0}$ defines a homomorphism of cochain complexes from the Hochschild complex $\{ \oplus_{n=0}^\infty C^n_\mathrm{Hoch} (A;M), \delta_\mathrm{Hoch} \}$ to the cochain complex $\{ \oplus_{n=0}^\infty C^n(N; N_M), d_{N, N_M} \}$. In the following, we consider the mapping cone induced by this homomorphism. More precisely, we let
\begin{align*}
    C^n_{\mathrm{NAlg}^0} ((A, N); (M, N_M)) := \begin{cases}
         C^0_\mathrm{Hoch} (A; M) = M & \text{ if } n =0, \vspace*{0.1cm} \\
         C^1_\mathrm{Hoch} (A; M) \oplus C^0 (N; N_M) = \mathrm{Hom} (A, M) \oplus M  & \text{ if } n =1, \vspace*{0.1cm} \\
    C^n_\mathrm{Hoch} (A;M) \oplus C^{n-1}(N; N_M)      & \text{ if } n \geq 2 \\
    \quad = \mathrm{Hom} (A^{\otimes n}, M) \oplus \mathrm{Hom} (A^{\otimes n-1}, M)
    \end{cases}
\end{align*}
and define a map $\delta_{\mathrm{NAlg}^0} : C^n_{\mathrm{NAlg}^0} ((A, N); (M, N_M)) \longrightarrow C^{n+1}_{\mathrm{NAlg}^0} ((A, N); (M, N_M))$ by
\begin{align*}
  \delta_{\mathrm{NAlg}^0} (u) := ( \delta_\mathrm{Hoch} (u), u ) \quad \text{ and } \quad \delta_{\mathrm{NAlg}^0} (\chi, F) := \big( \delta_\mathrm{Hoch} (\chi) ~ \!, ~ \! d_{N, N_M} (F) + (-1)^n ~ \! \partial^{N, N_M} (\chi) \big),
\end{align*}
for $u \in M$ and $(\chi, F) \in C^n_\mathrm{Hoch} (A; M) \oplus C^{n-1} (N; N_M) = C^{n \geq 1}_{\mathrm{NAlg}^0} ((A, N); (M, N_M))$.

\begin{prop}
    The map $\delta_{\mathrm{NAlg}^0}$ is a differential, i.e. $(\delta_{\mathrm{NAlg}^0})^2 = 0$.
\end{prop}

\begin{proof}
  For $u \in M$, we have
  \begin{align*}
  (\delta_{\mathrm{NAlg}^0})^2 (u) = \delta_{\mathrm{NAlg}^0} \big(  \delta_\mathrm{Hoch} (u), u \big) = \big(  (\delta_\mathrm{Hoch})^2 (u) ~ \! , ~ \! d_{N, N_M} (u) - \partial^{N, N_M} \delta_\mathrm{Hoch} (u)  \big) = 0.
\end{align*}    
The last equality follows from (\ref{dnnm}) and definition of $\partial^{N, N_M}$. Similarly, for $(\chi, F) \in C^{n \geq 1}_{\mathrm{NAlg}^0} ((A, N); (M, N_M))$, 
\begin{align*}
(\delta_{\mathrm{NAlg}^0})^2 (\chi, F) =~& \delta_{\mathrm{NAlg}^0} \big(  \delta_\mathrm{Hoch} (\chi) ~ \! , ~ \! d_{N, N_M} (F) + (-1)^n ~ \! \partial^{N, N_M} (\chi)   \big) \\
=~& \big(  \underbrace{(\delta_\mathrm{Hoch})^2 (\chi)}_{ = ~ \! 0 } ~ \! , ~ \! \underbrace{(d_{N, N_M})^2 (F)}_{= ~ \! 0} + \underbrace{ (-1)^n ~ \! d_{N, N_M} \partial^{N, N_M} (\chi) + (-1)^{n+1} ~ \! \partial^{N, N_M} \delta_{\mathrm{Hoch}} (\chi)}_{= ~ \! 0 \text{ by Proposition }\ref{prop-mapp} } \big) = 0.
\end{align*}
Hence the proof follows.
\end{proof}

The above proposition shows that $\{ \oplus_{n=0}^\infty C^n_{\mathrm{NAlg}^0} ((A, N); (M, N_M)), \delta_{\mathrm{NAlg}^0}  \}$ is a cochain complex. 
%The corresponding cohomology groups are called the {\em cohomology groups} of the Nijenhuis algebra $((A, ~ \! \cdot ~ \!), N)$ with coefficients in the Nijenhuis bimodule $((M, \triangleright, \triangleleft), N_M).$
The corresponding cohomology groups are denoted by $H^\bullet_{\mathrm{NAlg}^0} ((A, N); (M, N_M)).$ The next result connects this cohomology with the cohomologies of the underlying associative algebra and the Nijenhuis operator.

\begin{thm}
Let $((A, ~ \! \cdot ~ \! ), N)$ be a Nijenhuis algebra and $((M, \triangleright, \triangleleft), N_M)$ be a Nijenhuis bimodule over it. Then there is a short exact sequence of cochain complexes:
\begin{align*}
0 \rightarrow ~& C^{\bullet - 1} (N; N_M)  \xrightarrow{i} C^\bullet_{\mathrm{NAlg}^0} ((A, N); (M, N_M)) \xrightarrow{p} C^\bullet_\mathrm{Hoch} (A; M) \rightarrow 0\\
& \qquad \qquad \text{ where }  i(F) = (0, F) \text{ and }  p (\chi, F) = \chi
\end{align*}
which yields a long exact sequence in the cohomology:
\begin{align}\label{les}
\cdots \longrightarrow H^{n-1} (N; N_M) \longrightarrow H^n_{\mathrm{NAlg}^0} ((A, N); (M, N_M)) \longrightarrow H^n_\mathrm{Hoch} (A; M) \longrightarrow H^n (N; N_M) \longrightarrow \cdots.
\end{align}
\end{thm}

\medskip

\medskip

In the following, we consider a reduced version of the cochain complex defined above to study deformations and abelian extensions of Nijenhuis algebras. More precisely, let $((A, ~ \! \cdot ~ \!), N)$ be a Nijenhuis algebra and $((M, \triangleright, \triangleleft), N_M)$ be a Nijenhuis bimodule over it. 
We set $C^0_\mathrm{NAlg} ((A, N); (M, N_M)) = 0$,
\begin{align*}
C^1_\mathrm{NAlg} ((A, M); (M, N_M)) = C^1_\mathrm{Hoch} (A ; M) ~~~~ \text{ and } ~~~~ C^{n \geq 2}_\mathrm{NAlg} ((A, N); (M, N_M)) = C^n_\mathrm{Hoch} (A; M) \oplus C^{n-1} (N; N_M).
\end{align*}
Define a map $\delta_\mathrm{NAlg} : C^n_\mathrm{NAlg} ((A, N); (M, N_M)) \longrightarrow C^{n+1}_\mathrm{NAlg} ((A, N); (M, N_M))$ by
\begin{align*}
\delta_\mathrm{NAlg} (f) := (\delta_\mathrm{Hoch} (f) ~ \! , ~ \! - \partial^{N, N_M} (f))     ~~~~ \text{ and } ~~~~ \delta_\mathrm{NAlg} (\chi, F) := ( \delta_\mathrm{Hoch} (\chi) ~ \! , ~ \! d_{N, N_M} (F) + (-1)^n ~ \! \partial^{N, N_M} (\chi)),
\end{align*}
for $f \in C^1_\mathrm{Hoch} (A; M)$ and $(\chi, F) \in C^{n \geq 2}_\mathrm{NAlg} ((A, N); (M, N_M))$.
As before, we get that $(\delta_\mathrm{NAlg})^2 = 0$. The cochain complex $ \{ \oplus_{n=0}^\infty C^n_\mathrm{NAlg} ((A, N); (M, N_M)) , \delta_\mathrm{NAlg} \}$ thus obtained is called the {\bf cochain complex} of the Nijenhuis algebra $((A, ~ \! \cdot ~ \!), N)$ with coefficients in the Nijenhuis bimodule $((M, \triangleright, \triangleleft), N_M)$. We denote the space of $n$-cocycles by $Z^n_\mathrm{NAlg} ((A, N); (M, N_M))$ and the space of $n$-coboundaries by $B^n_\mathrm{NAlg} ((A, N); (M, N_M))$. The corresponding cohomology groups
\begin{align*}
H^n_\mathrm{NAlg} ((A, N); (M, N_M)) := \frac{Z^n_\mathrm{NAlg} ((A, N); (M, N_M))}{B^n_\mathrm{NAlg} ((A, N); (M, N_M))}
\end{align*}
are called the cohomology groups of the Nijenhuis algebra $(A, ~ \! \cdot ~ \!), N)$ with coefficients in the Nijenhuis bimodule $((M, \triangleright, \triangleleft), N_M)$.

\begin{remark}
(i) Note that a pair $(\chi, F) \in \mathrm{Hom} (A^{\otimes 2}, M) \oplus \mathrm{Hom}(A, M) = C^2_\mathrm{NAlg} ((A, N); (M, N_M))$ is a $2$-cocycle if and only if the following conditions are hold:
\begin{align}
   & \qquad a \triangleright \chi (b, c) - \chi (a \cdot b, c) + \chi (a, b \cdot c) - \chi (a, b) \triangleleft c = 0, \label{2co1}\\
     N(a) \triangleright F(b) ~ \! +~& F(a) \triangleleft N(b) - F (N(a) \cdot b + a \cdot N(b) - N(a \cdot b)) - N_M ( F(a) \triangleleft b + a \triangleright F(b) - F (a \cdot b) ) \label{2co2} \\
    & + \chi (N (a) , N(b) ) - N_M \big(  \chi (N(a), b) + \chi (a, N (b)) - \chi (a \cdot b)  \big) = 0, \nonumber
\end{align}
for all $a, b, c \in A$. Further, this is a $2$-coboundary if there exists a linear map $g \in \mathrm{Hom} (A, M)$ such that
\begin{align*}
    \chi (a, b) = a \triangleright g(b) - g (a \cdot b) + g (a) \triangleleft b \quad \text{ and } \quad F (a) = (g N - N_M g)(a), \text{ for } a, b \in A.
\end{align*} 

(ii) It follows from the definitions that $H^n_\mathrm{NAlg} ((A, N); (M, N_M)) = H^n_{\mathrm{NAlg}^0} ((A, N); (M, N_M))$, for all $n \geq 3$. Hence the long exact sequence (\ref{les}) still holds if the cohomology group $H^n_{\mathrm{NAlg}^0} ((A, N); (M, N_M))$ is replaced by $H^n_\mathrm{NAlg} ((A, N); (M, N_M))$ for the above range of $n$.
\end{remark}

\medskip

\noindent {\bf Particular case: The cohomology of a Nijenhuis algebra.} Here we write the cochain complex of the Nijenhuis algebra $((A, ~ \! \cdot ~ \!), N)$ with coefficients in the adjoint Nijenhuis bimodule. More precisely, we set
\begin{align*}
C^0_\mathrm{NAlg} ((A, N)) = 0, ~~~~ C^1_\mathrm{NAlg} ((A, N)) = \mathrm{Hom} (A, A), ~~~~ C^{n \geq 2}_\mathrm{NAlg} ((A, N)) = \mathrm{Hom} (A^{\otimes n }, A ) \oplus \mathrm{Hom }(A^{\otimes n-1}, A)
\end{align*}
and define a map $\delta_\mathrm{NAlg} : C^n_\mathrm{NAlg} ((A, N)) \rightarrow C^{n+1}_\mathrm{NAlg} ((A, N))$ by
\begin{align}\label{diff-nalg}
\delta_\mathrm{NAlg} (f) := (\delta_\mathrm{Hoch} (f) ~ \! , ~ \! - \partial^N (f)), \quad
 \delta_\mathrm{NAlg} (\chi, F) := ( \delta_\mathrm{Hoch} (\chi) ~ \! , ~ \! d_{N} (F) + (-1)^n ~ \! \partial^{N} (\chi)),
\end{align}
for $f \in C^1_\mathrm{NAlg} ((A, N))$ and $(\chi, F) \in C^{n \geq 2}_\mathrm{NAlg} ((A, N))$. Here the map $d_N$ is defined in (\ref{dn-map1}) and the map $\partial^N : \mathrm{Hom} (A^{\otimes n}, A) \rightarrow \mathrm{Hom} (A^{\otimes n}, A)$ is given by
\begin{align*}
\partial^{N} (f) &(a_1, \ldots, a_n) = f \big(  N (a_1), \ldots, N (a_n)  \big) - \sum_{i=1}^n N \big( f \big(  N (a_1), \ldots, a_i, \ldots, N (a_n)   \big)    \big) \\
    &+ \sum_{1 \leq i < j \leq n} N^2 \big( f \big(  N (a_1), \ldots, a_i, \ldots, a_j, \ldots, N (a_n)   \big)    \big) - \cdots \cdot + (-1)^n ~ \! N^n \big( f (a_1, \ldots, a_n)   \big),
\end{align*} 
for $f \in \mathrm{Hom} (A^{\otimes n}, A)$ and $a_1, \ldots, a_{n+1} \in A$. Then it turns out that $\{ \oplus_{n=0}^\infty C^n_\mathrm{NAlg} ((A, N)), \delta_\mathrm{NAlg} \}$ is a cochain complex. The corresponding cohomology groups are denoted by $H^\bullet_\mathrm{NAlg} ((A, N))$.

\begin{remark}
In \cite{mapping} Fiorenza and Manetti have shown that the mapping cone of a homomorphism of differential graded Lie algebras can be canonically equipped with an $L_\infty$-algebra structure which lifts the usual differential of the mapping cone. Since the collection $\{ \partial^N :  \mathrm{Hom} (A^{\otimes n}, A) \rightarrow  \mathrm{Hom} (A^{\otimes n}, A) \}_{n \geq 0}$ defines a homomorphism of cochain complexes from the Hochschild complex $\{ \oplus_{n=0}^\infty  \mathrm{Hom} (A^{\otimes n}, A), \delta_\mathrm{Hoch} \}$ to the cochain complex $\{ \oplus_{n=0}^\infty  \mathrm{Hom} (A^{\otimes n}, A), d_N \}$ of the Nijenhuis operator $N$, it turns out that the graded space $\oplus_{n=0}^\infty C^n_\mathrm{NAlg} ((A, N))$ inherits an $L_\infty$-algebra structure whose differential map is given by (\ref{diff-nalg}). It is important to note that the Hochschild complex and the complex $N$ both carry graded Lie brackets. However, at this point, it is unclear to us how these graded Lie brackets play the role in constructing the higher operations of the $L_\infty$-structure on $\oplus_{n=0}^\infty C^n_\mathrm{NAlg} ((A, N))$. We will trace this question in a subsequent work.
\end{remark}

\medskip

\section{Deformations and abelian extensions of Nijenhuis algebras}\label{sec4}

In this section, we first study (formal and infinitesimal) deformations of Nijenhuis algebras. Among others, we show that the set of all equivalence classes of infinitesimal deformations of a Nijenhuis algebra $( (A, ~ \! \cdot ~ \!), N)$ can be classified by the second cohomology group $H^2_{\mathrm{NAlg}} ((A, N))$. As another application of our cohomology, we discuss abelian extensions of a Nijenhuis algebra by a Nijenhuis bimodule over it. We prove that the set of all isomorphism classes of such abelian extensions has a bijection with the second cohomology group of the Nijenhuis algebra with coefficients in the given Nijenhuis bimodule. 

\subsection{Deformations of Nijenhuis algebras} Let $( (A, ~ \! \cdot ~ \!), N)$ be a Nijenhuis algebra. For convenience, we denote the associative multiplication of $A$ by the map $\mu \in \mathrm{Hom}(A^{\otimes 2}, A)$, i.e. $\mu (a, b) = a \cdot b$, for all $a, b \in A$. Now, consider the space $A [  [t ] ]$ of all formal power series in $t$ with coefficients from $A$. That is, 
\begin{align*}
    A[ [t]] = \big\{ \sum_{i=0}^\infty  a_i t^i ~ \! | ~ \! a_i \in A \text{ for all } i = 0 , 1 , 2, \ldots  \big\}.    
\end{align*}
Then $A[[t]]$ is obviously a ${\bf k} [[t]]$-module. A {\em formal deformation} of the Nijenhuis algebra $( (A, ~ \! \cdot ~ \!), N)$ consists of two formal sums
\begin{align*}
    \mu_t =~& \mu_0 + t \mu_1 + t^2 \mu_2 + \cdots ~ \in \mathrm{Hom} (A^{\otimes 2}, A) [[t]] ~~~ \text{ with } \mu_0 = \mu,\\
    N_t =~& N_0 + t N_1 + t^2 N_2 + \cdots ~ \in \mathrm{Hom} (A, A) [[t]] ~~~ \text{ with } N_0 = N
\end{align*}
that make $(A [[t]], \mu_t)$ into an associative algebra over the ring ${\bf k} [[t]]$, and the ${\bf k}[[t]]$-linear extension map $N_t : A[[t]] \rightarrow A[[t]]$ is a Nijenhuis operator on the associative algebra $(A[[t]], \mu_t)$. In other words, $( (A[[t]], \mu_t), N_t)$ is a Nijenhuis algebra over the ring ${\bf k} [[t]]$. We denote a formal deformation as above simply by the Nijenhuis algebra $( (A[[t]], \mu_t), N_t)$ over the ring ${\bf k} [[t]]$.

It follows that $( (A[[t]], \mu_t), N_t)$ is a formal deformation of the Nijenhuis algebra $((A, ~ \! \cdot ~ \! ), N)$ if and only if the following set of equations are hold: for any $p \geq 0$ and $a, b, c \in A$,
\begin{align*}
%\begin{cases}
   \sum_{i+j = p} \mu_i (\mu_j (a, b), c) =~& \sum_{i+j = p} \mu_i (a, \mu_j (b, c)), \\
    \sum_{i+j+k = p} \mu_i (N_j (a), N_k (b)) =~& \sum_{i+j+k = p} \big\{ N_i \big(  \mu_j (N_k (a), b) + \mu_j (a, N_k (b)) - N_k (\mu_j (a, b))   \big) \big\}.
%\end{cases}    
\end{align*}
Note that for $p= 0$, both the above identities are trivially hold as $\mu_0 = \mu$ and $N_0 = N$. However, for $p=1$, the above equations give rise to
\begin{align}
    \mu_1 (a, b) \cdot c + \mu_1 (a \cdot b, c) =~& a \cdot \mu_1 (b, c) + \mu_1 (a, b \cdot c), \label{alg-def}\\
    \mu_1 (N(a), N(b)) + N_1 (a) \cdot N (b) + N (a) \cdot N_1 (b) 
    =~&  N_1 (N (a) \cdot b + a \cdot N (b) - N (a \cdot b)) \label{nij-def} \\
    + N \big( \mu_1 ( N (a), b) + \mu_1 (a, N (b)) & - N \mu_1 (a, b) \big) + N ( N_1 (a) \cdot b + a \cdot N_1 (b) - N_1 (a \cdot b)), \nonumber
\end{align}
for all $a, b, c \in A$. The identity (\ref{alg-def}) simply means that $(\delta_\mathrm{Hoch} \mu_1) (a, b, c) = 0$ while the identity (\ref{nij-def}) can be equivalently rephrased as $ ( d_N (N_1) + \partial^N (\mu_1) ) (a, b) = 0$. Combining these two observations, we get that
\begin{align*}
    \delta_{\mathrm{NAlg}} (\mu_1, N_1) = \big(  \delta_\mathrm{Hoch} (\mu_1) ~ \!, ~ \! d_N (N_1) + \partial^N (\mu_1) \big) = 0.
\end{align*}
In other words, $(\mu_1, N_1) \in \mathrm{Hom}(A^{\otimes 2}, A) \oplus \mathrm{Hom}(A, A) = C^2_\mathrm{NAlg} ((A, N))$ is a $2$-cocycle in the cochain complex of the Nijenhuis algebra $((A, ~ \! \cdot ~ \! ), N)$ with coefficients in the adjoint Nijenhuis bimodule. The $2$-cocycle $(\mu_1, N_1)$ is called the {\em infinitesimal} of the given formal deformation $ ((A[[t]], \mu_t ) , N_t).$ 

\medskip

Let $ ((A[[t]], \mu_t = \sum_{i=0}^\infty t^i \mu_i) , N_t = \sum_{i=0}^\infty t^i N_i)$ and $ ((A[[t]], \mu'_t = \sum_{i=0}^\infty t^i \mu'_i) , N'_t = \sum_{i=0}^\infty t^i N'_i)$ be two formal deformations of the Nijenhuis algebra $((A, ~ \! \cdot ~ \! ), N)$. They are said to be {\em equivalent} if there exists a ${\bf k}[[t]]$-linear map $\varphi_t : A [[t]] \rightarrow A[[t]]$ of the form
\begin{align*}
    \varphi_t = \varphi_0 + t \varphi_1 + t^2 \varphi_2 + \cdots ~ \in \mathrm{Hom} (A, A) [[t]] \quad  (\text{with } \varphi_0 = \mathrm{Id}_A)
\end{align*}
that defines an isomorphism of Nijenhuis algebras from $ ((A[[t]], \mu_t) , N_t )$ to $ ((A[[t]], \mu'_t ) , N'_t )$. In other words, the following set of equations must hold: for any $p \geq 0$ and $a, b \in A$,
\begin{align*}
    \sum_{i+j = p} \varphi_i (\mu_j (a, b)) = \sum_{i+j+k = p} \mu_i' ( \varphi_j (a), \varphi_k (b)) \qquad \text{ and } \qquad \sum_{i+j = p} \varphi_i \circ N_j = \sum_{i+j = p} N_i' \circ \varphi_j.
\end{align*}
For $p= 0$, both the above identities are trivially hold as $\mu_0 = \mu_0' = \mu$, $N_0 = N'_0 = N$ and $\varphi_0 = \mathrm{Id}_A$. However, for $p =1$, we get that
\begin{align*}
    \mu_1 (a, b) - \mu_1' (a, b) =~& a \cdot \varphi_1 (b) - \varphi_1 (a \cdot b) + \varphi_1 (a) \cdot b = (\delta_\mathrm{Hoch} \varphi_1) (a, b), \\
    N_1 - N_1' =~& N \circ \varphi_1 - \varphi_1 \circ N.
\end{align*}
Combining these, we obtain
\begin{align*}
    (\mu_1, N_1) - (\mu'_1, N_1') = (\delta_\mathrm{Hoch} (\varphi_1) ~ \!, ~ \! - \partial^N (\varphi_1))  = \delta_\mathrm{NAlg} (\varphi_1).
\end{align*}

As a conclusion, we have the following result.

\begin{thm}\label{theorem-deformation1}
    Let $((A, ~ \! \cdot ~ \!), N)$ be a Nijenhuis algebra. Then the infinitesimal of any formal deformation of the Nijenhuis algebra $((A, ~ \! \cdot ~ \!), N)$ is a $2$-cocycle in the cochain complex $\{ C^\bullet_\mathrm{NAlg} ((A, N)), \delta_\mathrm{NAlg} \}$. Moreover, the cohomology class of the corresponding $2$-cocycle depends only on the equivalence class of the formal deformation.
\end{thm}

Given a formal deformation of a Nijenhuis algebra, we have seen that the corresponding `infinitesimal' is a $2$-cocycle. However, an arbitrary $2$-cocycle need not be the infinitesimal of some formal deformation. In the following, we consider the notion of `infinitesimal deformations' of a Nijenhuis algebra $((A, ~ \! \cdot ~ \!), N)$, and show that the set of all equivalence classes of infinitesimal deformations has a bijection with the second cohomology group $H^2_\mathrm{NAlg} ((A, N)).$

Let $((A, ~ \! \cdot ~ \!), N)$ be a Nijenhuis algebra. As before, let $\mu$ denote the associative multiplication of $A$. An {\em infinitesimal deformation} of the Nijenhuis algebra $((A, ~ \! \cdot ~ \!), N)$ consists of two parametrized sum $\mu_t = \mu + t \mu_1$ and $N_t = N + t N_1$ that makes $( (A[t]/ (t^2), \mu_t), N_t)$ into a Nijenhuis algebra over the local Artinian ring ${\bf k}[t]/ (t^2)$. Similar to the case of formal deformations, we say that two infinitesimal deformations $( (A[t]/ (t^2), \mu_t = \mu+ t \mu_1), N_t = N + t N_1)$ and $( (A[t]/ (t^2), \mu'_t = \mu+ t \mu'_1), N'_t = N + t N'_1)$ are {\em equivalent} if there exists a ${\bf k}[t]/(t^2)$-linear map $\varphi_t : A[t]/ (t^2) \rightarrow A[t]/(t^2)$ of the form $\varphi_t = \mathrm{Id}_A + t \varphi_1$ that defines an isomorphism of the above Nijenhuis algebras. Then we have the following result.

\begin{thm}\label{theorem-deformation2}
    Let $((A, ~ \! \cdot ~ \!), N)$ be a Nijenhuis algebra. Then there is a bijection
    \begin{align*}
        \bigg\{ \substack{ \text{Set of all equivalence classes of infinitesimal } \\ \text{ deformations of } ((A, ~ \! \cdot ~ \!), N) } \bigg\} \quad \! \!  \! \xrightleftharpoons[\Psi]{\Phi}   \quad \! \! \! \!   H^2_\mathrm{NAlg} ((A, N)).
    \end{align*}
\end{thm}

\begin{proof}
Let $\big( ( A[t] / (t^2), \mu_t = \mu + t \mu_1), N_t = N + t N_1 \big)$ be an infinitesimal deformation. Then one can easily check (similar to the formal deformation case) that $(\mu_1, N_1)$ is a $2$-cocyle. Moreover, equivalent infinitesimal deformations give rise to cohomologous $2$-cocycles. Hence there is a well-defined map $\Phi$ from the set of all equivalence classes of infinitesimal deformations to the second cohomology group $ H^2_\mathrm{NAlg} ((A, N)).$ On the other hand, we first observe that any $2$-cocycle $(\mu_1, N_1)$ gives rise to an infinitesimal deformation $\big( ( A[t] / (t^2), \mu_t = \mu + t \mu_1), N_t = N + t N_1 \big)$ of the given Nijenhuis algebra $((A, ~ \! \cdot ~ \!), N)$. Further, cohomologous $2$-cocycles correspond to equivalent infinitesimal deformations. This construction yields a map $\Psi$ in the other direction. Finally, the maps $\Phi$ and $\Psi$ are inverses to each other. This completes the proof. 
\end{proof}

\subsection{Abelian extensions of Nijenhuis algebras} In this subsection, we study abelian extensions of a Nijenhuis algebra by a Nijenhuis bimodule over it. We find a bijection between the set of all isomorphism classes of such abelian extensions and the second cohomology group of the Nijenhuis algebra.

\medskip

Let $((A, ~ \! \cdot ~ \!), N)$ be a Nijenhuis algebra. Also, let $(M, N_M)$ be a pair consisting of a vector space $M$ endowed with a linear map $N_M: M \rightarrow M$. Note that $M$ may not have any other additional structure. We now regard $M$ as an associative algebra with trivial associative multiplication. Then the pair $((M, 0), N_M)$ becomes a Nijenhuis algebra.
An {\bf abelian extension} of the Nijenhuis algebra $((A, ~ \! \cdot ~ \!), N)$ by the pair $(M, N_M)$ is another Nijenhuis algebra $((E, ~ \! \cdot_E ~ \! ), N_E )$ with a short exact sequence 
\begin{align}\label{abelian-ext}
   \xymatrix{
   0 \ar[r] & ( (M, 0 ), N_M) \ar[r]^i & ((E, ~ \! \cdot_E ~ \! ), N_E ) \ar[r]^p & ((A, ~ \! \cdot ~ \!), N) \ar[r] & 0
   }
\end{align}
of Nijenhuis algebras. We often denote an abelian extension as above simply by the Nijenhuis algebra $((E, ~ \! \cdot_E ~ \! ), N_E )$ when the underlying short exact sequence is understood.

\medskip

Let $((E, ~ \! \cdot_E ~ \! ), N_E )$ and $((E', ~ \! \cdot_{E'} ~ \! ), N_{E'} )$ be two abelian extensions of the Nijenhuis algebra $((A, ~ \! \cdot ~ \! ), N)$ by the pair $(M, N_M)$. These two abelian extensions are said to be {\em isomorphic} if there exists an isomorphism $\varphi : ( (E, ~ \! \cdot_E ~ \! ), N_E) \rightarrow ( (E', ~ \! \cdot_{E'}), N_{E'})$ of Nijenhuis algebras that make the following diagram commutative
\begin{align*}
   \xymatrix{
   0 \ar[r] & ( (M, 0 ), N_M) \ar[d]_{\mathrm{Id}_M} \ar[r]^i & ((E, ~ \! \cdot_E ~ \! ), N_E ) \ar[d]^{\varphi} \ar[r]^p & ((A, ~ \! \cdot ~ \!), N) \ar[r] \ar[d]^{\mathrm{Id}_A} & 0 \\
    0 \ar[r] & ( (M, 0 ), N_M) \ar[r]_{i'} & ((E', ~ \! \cdot_{E'} ~ \! ), N_{E'}) \ar[r]_{p'} & ((A, ~ \! \cdot ~ \!), N) \ar[r] & 0.
   }
\end{align*}

Let $((E, ~ \! \cdot_E ~ \!), N_E)$ be an abelian extension as of (\ref{abelian-ext}). A {\em section} of the map $p$ is a linear map $s: A \rightarrow E$ that satisfies $p \circ s = \mathrm{Id}_A$. Note that a section of $p$ always exists. Let $s: A \rightarrow E$ be any section of the map $p$. We define bilinear maps $\triangleright : A \times M \rightarrow M$ and $\triangleleft : M \times A \rightarrow M$ by
\begin{align*}
    a \triangleright u := s (a) \cdot_E i (u) \quad \text{ and } \quad u \triangleleft a := i (u) \cdot_E s (a), \text{ for } a \in A, u \in M.
\end{align*}
It can be easily checked that the above two maps don't depend on the choice of the section $s$. Moreover, these two maps make the vector space $M$ into an $A$-bimodule. Further, for any $a \in A$ and $u \in M$, we observe that
\begin{align*}
    N(a) \triangleright N_M (u) =~& s N (a) \cdot_E i N_M (u) \\
    =~& N_E s (a) \cdot_E N_E i (u) \quad (\because ~ (sN - N_E s) (a) \in M ~\text{and the multiplication of } M \text{ is trivial}) \\
    =~& N_E \big(  N_E s(a) \cdot_E i (u) + s(a) \cdot N_E i (u) - N_E ( s(a) \cdot_E i(u) ) \big)\\
    =~& N_E \big( sN (a) \cdot_E i(u) + s(a) \cdot_E i N_M (u) - N_E (a \triangleright u)   \big) \quad (\because ~ (sN - N_E s) (a) \in M)\\
    =~& N_M \big(  N (a) \triangleright u + a \triangleright N_M (u) - N_M (a \triangleright u)   \big) \quad (\because~ N_E \big|_M = N_M).
\end{align*}
Similarly, one can show that $N_M (u) \triangleleft N (a) = N_M \big( N_M (u) \triangleleft a + u \triangleleft N (a) - N_M (u \triangleleft a) \big)$. As a consequence, $((M, \triangleright, \triangleleft), N_M)$ becomes a Nijenhuis bimodule over the Nijenhuis algebra $((A , ~ \! \cdot ~ \! ), N)$.

\medskip

Next, let $((M, \triangleright, \triangleleft), N_M)$ be a given Nijenhuis bimodule over the Nijenhuis algebra $((A, ~ \! \cdot ~ \!), N)$. We denote by $\mathcal{E}xt^2 ((A, N), (M, N_M))$ the set of all isomorphism classes of abelian extensions of the Nijenhuis algebra $((A, ~ \! \cdot ~ \!), N)$ by the pair $(M, N_M)$ so that the induced Nijenhuis bimodule structure coincides with the prescribed one. Then we have the following result.

\begin{thm}\label{theorem-abelian}
    Let $((A, ~ \! \cdot ~ \!), N)$ be a Nijenhuis algebra and $((M, \triangleright, \triangleleft), N_M)$ be a given Nijenhuis bimodule over it. Then there is a bijection 
    \begin{align*}
    \mathcal{E}xt^2 ((A, N), (M, N_M)) ~ \cong ~ H^2_\mathrm{NAlg}  ((A, N); (M, N_M)).
    \end{align*}
\end{thm}

\begin{proof}
Let $((E, \cdot_E), N_E)$ be an abelian extension as of (\ref{abelian-ext}). Using any section $s: A \rightarrow E$ of the map $p$, we may identify the space $E$ with the direct sum $A \oplus M$ and with this identification, the map $i$ is the inclusion and $p$ is the projection. Since $p : E \rightarrow A$ is an algebra homomorphism, it follows that $(a, 0) \cdot_E (b, 0) = (a \cdot b ~ \!, ~ \! \chi (a, b))$, for some $\chi \in \mathrm{Hom} (A^{\otimes 2}, M)$. The associativity of the multiplication $\cdot_E$ then implies that $\chi$ is a Hochschild $2$-cocyle (i.e. $\delta_\mathrm{Hoch} \chi = 0$). Further, the map $p$ satisfies $N \circ p = p \circ N_E$ which in turn implies that $N_E (a, 0) = (N(a), F (a))$, for some $F \in \mathrm{Hom} (A, M)$. The Nijenhuis property of $N_E$ then ensures that 
\begin{align*}
  N(a) \triangleright F(b) ~ \! +~& F(a) \triangleleft N(b) - F (N(a) \cdot b + a \cdot N(b) - N(a \cdot b)) - N_M ( F(a) \triangleleft b + a \triangleright F(b) - F (a \cdot b) ) \nonumber \\
    & + \chi (N (a) , N(b) ) - N_M \big(  \chi (N(a), b) + \chi (a, N (b)) - \chi (a \cdot b)  \big) = 0, \text{ for all } a, b \in A.
\end{align*}
As a consequence, we get that the pair $(\chi, F) \in \mathrm{Hom} (A^{\otimes 2}, M) \oplus \mathrm{Hom}(A, M) = C^2_\mathrm{NAlg} ((A, N); (M, N_M))$ is a $2$-cocycle. Similarly, one can easily observe that any two isomorphic abelian extensions are related by a map $\varphi : E = A \oplus M \rightarrow A \oplus M = E'$, $\varphi (a, u) = (a, u + g (a))$, for some $g \in \mathrm{Hom}(A, M)$. Since $\varphi$ is a Nijenhuis algebra (iso)morphism, we get that
\begin{align*}
    \varphi ((a, 0) \cdot_E (b, 0)) = \varphi (a, 0) \cdot_{E'} \varphi (b, 0) ~~~~ \text{ and } ~~~~ (N_{E'} \circ \varphi ) (a, 0) = (\varphi \circ N_E) (a, 0),
\end{align*}
for all $a, b \in A$. As a result, we obtain
\begin{align*}
    (\chi - \chi' ) (a, b) &~= a \triangleright g (b) - g (a \cdot b) + g (a) \triangleleft b \quad  \text{ and } \quad (F- F')(a) = (N_M g - g N) (a), \\
    &\mathrm{i.e.}~ (\chi, F) - (\chi', F') = ( \delta_\mathrm{Hoch} (g) ~ \! , ~ \! - \partial^{N, N_M} (g) ) = \delta_\mathrm{NAlg} (g).
\end{align*}
Here $(\chi', F')$ is the $2$-cocycle induced from the abelian extension $((E', \cdot_{E'}), N_{E'})$. The above construction yields a map $\Phi : \mathcal{E}xt^2 ((A, N), (M, N_M)) \rightarrow H^2_\mathrm{NAlg} ((A, N); (M, N_M))$.

\medskip

Conversely, let $(\chi, F) \in Z^2_\mathrm{NAlg} ((A, N); (M, N_M))$ be a $2$-cocycle. We set $E = A \oplus M$. Then we define a bilinear multiplication $\cdot_E : E \times E \rightarrow E$ and a linear map $N_E : E \rightarrow E$ by
\begin{align*}
    (a, u) \cdot_E (b, v) := ( a \cdot b ~ \!, ~ \! a \triangleright v + u \triangleleft b + \chi (a, b) ) ~~~~ \text{ and } ~~~~
    N_E (a, u) := (N (a) ~ \!, ~ \! N_M (u) + F (a)),
\end{align*}
for $(a, u), (b, v) \in E$. Since $\delta_\mathrm{Hoch} \chi = 0$, it follows that $\cdot_E$ defines an associative algebra structure on $E$. On the other hand, $d_{N, N_M} (F) + \partial^{N, N_M} (\chi) = 0$ implies that $N_E : E \rightarrow E$ is a Nijenhuis operator on the associative algebra $(E, \cdot_E)$. In other words, $((E, \cdot_E), N_E)$ is a Nijenhuis algebra. Moreover, 
\begin{align*}
   \xymatrix{
   0 \ar[r] & ( (M, 0 ), N_M) \ar[r]^i & ((E, ~ \! \cdot_E ~ \! ), N_E ) \ar[r]^p & ((A, ~ \! \cdot ~ \!), N) \ar[r] & 0
   }
\end{align*}
is a short exact sequence of Nijenhuis algebras, where $i$ is the inclusion and $p$ is the projection. Next, let $(\chi', F') \in Z^2_\mathrm{NAlg} ((A, N); (M, N_M))$ be any other $2$-cocycle cohomologous to $(\chi, F)$, i.e.
\begin{align*}
    (\chi, F) - (\chi', F') = \delta_\mathrm{NAlg} (g) = (\delta_\mathrm{Hoch} (g) ~ \! , ~ \! - \partial^{N, N_M} (g)),
\end{align*}
for some linear map $g \in \mathrm{Hom} (A, M)$. If $(( E' = A \oplus M, ~ \! \cdot_{E'} ), N_{E'} )$ is the abelian extension corresponding to the $2$-cocycle $(\chi', F')$, then there is an isomorphism of Nijenhuis algebras
\begin{align*}
    \varphi : ((E= A \oplus M, \cdot_E), N_E) \rightarrow ((E'= A \oplus M, \cdot_{E'}), N_{E'}) ~~~~ \text{ given by } ~~~~ \varphi (a, u) \mapsto (a, u + g (a)).
\end{align*}
Hence we obtain a well-defined map $\Psi : H^2_\mathrm{NAlg} ((A, N); (M, N_M)) \rightarrow \mathcal{E}xt^2 ((A, N), (M, N_M) ).$ Finally, the maps $\Phi$ and $\Psi$ are inverses to each other. This concludes the proof.
\end{proof}

\section{Inducibility of Nijenhuis algebra automorphisms}\label{sec5}
Given an abelian extension of a Nijenhuis algebra by a Nijenhuis bimodule, we study the inducibility of a pair of Nijenhuis algebra automorphisms from an automorphism of the extension. To find an obstruction, we first define the Wells map in the context of Nijenhuis algebras. Our main result shows that a pair of Nijenhuis algebra automorphisms is inducible if and only if the pair is `compatible' and the image of this pair under the Wells map vanishes identically.

Let $((A, ~ \! \cdot ~ \!), N)$ be a Nijenhuis algebra and $((M, \triangleright, \triangleleft), N_M)$ be a Nijenhuis bimodule over it. Suppose
\begin{align}\label{ind-abel}
   \xymatrix{
   0 \ar[r] & ( (M, 0 ), N_M) \ar[r]^i & ((E, ~ \! \cdot_E ~ \! ), N_E ) \ar[r]^p & ((A, ~ \! \cdot ~ \!), N) \ar[r] & 0
   }
\end{align}
is a given abelian extension describing an element of $\mathcal{E}xt^2 ((A,N), (M, N_M))$, i.e. the induced $A$-bimodule structure on $M$ coincides with the prescribed one. Let $\mathrm{Aut}_M (E, N_E)$ be the set of all Nijenhuis algebra automorphisms $\varphi \in \mathrm{Aut} (E, N_E)$ that satisfies $\varphi (M) \subset M$. Then $\mathrm{Aut}_M (E, N_E)$ is obviously a subgroup of $\mathrm{Aut} (E, N_E)$. Let $\varphi \in \mathrm{Aut}_M (E, N_E)$. Then we naturally have $\varphi \big|_M \in \mathrm{Aut}(M, N_M)$. Next, for any section $s: A \rightarrow E$ of the map $p$, we define a map $\overline{\varphi}: A \rightarrow A$ by $\overline{\varphi} (a) := (p \varphi s) (a)$, for $a \in A.$ It can be checked that the map $\overline{\varphi}$ doesn't depend on the choice of the section $s$. Further, the map $\overline{\varphi}$ is a bijection on the set $A$. For any $a, b \in A$, we also have
\begin{align*}
    \overline{\varphi} (a \cdot b) =~& p \varphi ( s(a) \cdot_E s (b) - \chi (a, b)) \\
    =~& p \varphi (  s(a) \cdot_E s (b) ) \quad (\because ~ \! \varphi (M) \subset M \text{ and } p \big|_M = 0) \\
    =~& p \varphi s (a) \cdot p \varphi s (b) = \varphi (a) \cdot \varphi (b) 
\end{align*}
and additionally 
\begin{align*}
  ( N \circ \overline{\varphi} )(a) =  N p \varphi s (a) =~& p \varphi N_E s (a) \quad (\because ~ \! Np = p N_E \text{ and } N_E \varphi = \varphi N_E)\\
  =~& p \varphi s N (a)  \quad (\because ~ \! (N_E s  - s N) (a) \in M, ~ \varphi (M) \subset M \text{ and } p \big|_M = 0) \\
  =~& (\overline{\varphi} \circ N)(a).
\end{align*}
This shows that $\overline{\varphi} \in \mathrm{Aut} (A, N)$ is a Nijenhuis algebra automorphism. Hence we obtain a map
\begin{align*}
    \tau : \mathrm{Aut}_M (E, N_E) \rightarrow \mathrm{Aut} (M, N_M) \times \mathrm{Aut} (A, N) ~~~  \text{ given by } ~~~ \tau (\varphi) = (\varphi \big|_M, \overline{\varphi}),
\end{align*}
for $\varphi \in \mathrm{Aut}_M (E, N_E).$ A pair of Nijenhuis algebra automorphisms $(\beta, \alpha) \in \mathrm{Aut} (M, N_M) \times \mathrm{Aut} (A, N)$ is said to be {\em inducible} if there exists a Nijenhuis algebra automorphism $\varphi \in \mathrm{Aut}_M (E, N_E)$ such that $\tau (\varphi) = (\beta, \alpha)$, i.e. $\varphi \big|_M = \beta$ and $\overline{\varphi} = \alpha$. The inducibility problem then asks to find a necessary and sufficient condition under which a pair $(\beta, \alpha) \in \mathrm{Aut} (M, N_M) \times \mathrm{Aut} (A, N)$ of Nijenhuis algebra automorphisms is inducible.

\begin{prop}\label{pre-final-prop}
    Let (\ref{ind-abel}) be a given abelian extension. For a fixed section $s$, suppose the above abelian extension corresponds to the $2$-cocycle $(\chi, F)$. Then a pair $(\beta, \alpha) \in \mathrm{Aut} (M, N_M) \times \mathrm{Aut}(A, N)$ of Nijenhuis algebra automorphisms is inducible if and only if the following conditions hold:
    \begin{itemize}
        \item[(I)] $\beta (a \triangleright u) = \alpha (a) \triangleright \beta (u)$ and $\beta (u \triangleleft a) = \beta (u) \triangleleft \alpha (a)$, for $a \in A$ and $u \in M$,
        \item[(II)] there exists a linear map $\lambda : A \rightarrow M$ satisfying
        \begin{align}
            \beta (\chi (a, b)) - \chi (\alpha (a), \alpha (b)) =~& \alpha (a) \triangleright \lambda (b) + \lambda (a) \triangleleft \alpha (b) - \lambda (a \cdot b), \label{eqn-8}\\
            \beta (F (a)) - F (\alpha (a)) =~& N_M (\lambda (a) ) - \lambda (N (a)), \text{ for all } a, b \in A. \label{eqn-9}
        \end{align}
    \end{itemize}
\end{prop}

\begin{proof}
    Let $(\beta, \alpha)$ be an inducible pair, i.e. there exists a Nijenhuis algebra automorphism $\varphi \in \mathrm{Aut}_M (E, N_E)$ such that $\varphi \big|_M = \beta $ and $p \varphi s = \alpha$. For any $a \in A$ and $u \in M$, we observe that 
    \begin{align*}
        \beta (a \triangleright u) - \alpha (a) \triangleright \beta (u) =~& \beta (s (a) \cdot_E i(u)) - s \alpha (a) \cdot_E i \beta (u) \\
         =~& \varphi s (a) \cdot_E \varphi i (u) - s \alpha (a) \cdot_E i \beta (u) \quad (\because ~ \! \beta = \varphi|_M)\\
        =~& (\varphi s - s \alpha) (a) \cdot_E i \beta (u) = 0 \quad (\because ~ \! (\varphi s - s \alpha) (a) \in \mathrm{ker~} p = \mathrm{im~} i \cong M).
    \end{align*}
    In the same way, one can show that $\beta (u \triangleleft a) = \beta (u) \triangleleft \alpha (a)$. Hence the identities in (I) follow.

    Next, we define a map $\lambda : A \rightarrow M$ by $\lambda (a) = (\varphi s -  s \alpha)(a)$, for $a \in A$. Then for any $a, b \in A$, we observe that
    \begin{align*}
        &\alpha (a) \triangleright \lambda (b) + \lambda (a) \triangleleft \alpha (b) - \lambda (a \cdot b) \\
        &= s \alpha (a) \cdot_E (\varphi s - s \alpha) (b) + (\varphi s - s \alpha) (a) \cdot_E s \alpha (b) - (\varphi s - s \alpha ) (a \cdot b)\\
        &= \varphi s (a) \cdot_E (\varphi s - s \alpha) (b) + \varphi s (a) \cdot_E s \alpha (b) - s \alpha (a) \cdot_E s \alpha (b) - \varphi s (a \cdot b) + s ( \alpha (a) \cdot \alpha (b) )\\
         & \qquad \qquad \qquad \qquad \qquad \quad (\because ~ \! (\varphi s - s \alpha) (a) \in M \text{ and the multiplication of } M \text{ is trivial}) \\
        &= \varphi s (a) \cdot_E \varphi s (b) - \varphi s (a \cdot b) - s \alpha 
        (a) \cdot_E s \alpha (b) + s (\alpha (a) \cdot \alpha (b)) \quad (\text{after cancellation and rearrangement})\\
        &= \beta ( s(a) \cdot_E s (b) - s (a \cdot b)) - \big( s \alpha(a) \cdot_E s \alpha (b) - s ( \alpha (a) \cdot \alpha (b) )  \big)\\
        &= \beta (\chi (a, b)) - \chi (\alpha (a) , \alpha (b) )
    \end{align*}
    and
    \begin{align*}
        \beta (F (a)) - F (\alpha (a)) 
        =~& \beta  ( N_E s (a) - s N(a)) - (N_E s \alpha (a) - s N \alpha (a))\\
        =~& \varphi N_E s (a) - \varphi s N (a) - N_E s \alpha (a) + sN \alpha (a) \quad (\because ~ \! \beta = \varphi|_M)\\
        =~& N_E (\varphi s - s \alpha) (a) - (\varphi s - s \alpha) N (a) \quad (\because ~ \! \varphi N_E = N_E \varphi \text{ and } \alpha N = N \alpha) \\
        =~& N_M (\lambda (a)) - \lambda (N (a)) \quad (\because ~ \! N_M = N_E \big|_M).
    \end{align*}
    Hence the one direction of the proof follows.

    Conversely, let $(\beta, \alpha)$ satisfy the conditions in (I) and (II). Since $s$ is a section of the map $p$, any element $e \in E$ can be uniquely written as $e = s (a) + u$, for some $a \in A$ and $u \in M$. We now define a map $\varphi : E \rightarrow E$ by
    \begin{align*}
         \varphi (e) = \varphi (s (a) + u ) = s (\alpha (a)) + (\beta (u) + \lambda (a)).
    \end{align*}
    The map $\varphi$ is clearly injective as $\varphi (e) = 0$ means that $s (\alpha (a)) = 0$ and $\beta (u) + \lambda (a) = 0$. Since $s, \alpha, \beta $ are all injectives we get that $a = 0$ and $u = 0$ which in turn implies that $e = s (a) + u = 0$. It is also surjective as any element $e = s(a) + u \in E$ has a unique preimage $e' = s (\alpha^{-1} (a)) + (\beta^{-1}(u) - \beta^{-1} \lambda \alpha^{-1} (a)) \in E$. This shows that $\varphi$ is bijective. Next, we claim that $\varphi: E \rightarrow E$ is a Nijenhuis algebra automorphism. For any two elements $e = s(a) + u$ and $e' = s (b) + v$ from the vector space $E$, we observe that
    \begin{align*}
        &\varphi (e) \cdot_E \varphi (e') \\
        &= \big( s (\alpha (a)) + \beta (u) + \lambda (a)   \big) \cdot_E \big( s (\alpha (b)) + \beta (v) + \lambda (b)   \big)\\
        &= s (\alpha (a)) \cdot_E s (\alpha (b) ) + s (\alpha (a)) \cdot_E \beta (v) + s (\alpha (a)) \cdot_E \lambda (b) + \beta (u) \cdot_E s (\alpha (b)) + \underbrace{\beta (u) \cdot_E \beta (v)}_{= 0}\\
        & \quad \qquad \qquad  + \underbrace{\beta (u) \cdot_E \lambda (b)}_{=0} + \lambda (a) \cdot_E s (\alpha (b)) + \underbrace{\lambda (a) \cdot_E \beta (v)}_{= 0} + \underbrace{\lambda (a) \cdot_E \lambda (b)}_{= 0} \\
        &= s (\alpha (a) \cdot \alpha (b)) + \chi (\alpha (a), \alpha (b)) + \alpha (a) \triangleright \beta (v) + \alpha (a) \triangleleft \lambda (b) + \beta (u) \triangleleft \alpha (b) + \lambda (a) \triangleleft \alpha (b)\\
        &= s ( \alpha (a \cdot b)) + \beta \big( \chi (a, b) + a \triangleright v + u \triangleleft b   \big) + \lambda (a \cdot b) \quad (\text{by using (I) and } (\ref{eqn-8}))\\
        &= \varphi \big(  s (a \cdot b) + \chi (a, b) +  a \triangleright v + u \triangleleft b \big) \quad (\text{from the definition of } \varphi) \\
        &= \varphi \big(  s (a \cdot b) + \chi (a, b) + s (a) \cdot_E v + u \cdot_E s(b) + u \cdot_E v  \big) \quad (\because ~ \! u \cdot_E v = 0)\\
        &= \varphi \big(    s(a) \cdot_E s (b) + s (a) \cdot_E v + u \cdot_E s(b) + u \cdot_E v  \big)\\
        &= \varphi \big( (s (a) + u) \cdot_E ( s (b) + v)   \big) = \varphi (e \cdot_E e').
    \end{align*}
    and
    \begin{align*}
        (\varphi \circ N_E ) (e) =~& (\varphi \circ N_E) (s (a) + u)\\
        =~& \varphi \big( s N(a) + F(a) +N_M (u)   \big) \quad (\because ~ \! N_E s - s N = F \text{ and } N_M = N_E \big|_M)\\
        =~& s (\alpha N (a)) + \beta (F (a) + N_M (u)) + \lambda (N (a))\\
        =~& N_E s (\alpha (a)) + N_M \beta (u) + \big( \beta (F (a)) - F (\alpha (a)) + \lambda (N (a))   \big) \\
        & \qquad \qquad \qquad (\because ~ \! \alpha N = N \alpha, ~ \beta N_M = N_M \beta \text{ and } N_E s - sN = F)\\
        =~& N_E s (\alpha (a)) + N_M \beta (u) + N_M (\lambda (a)) \quad (\text{by }(\ref{eqn-9}))\\
        =~& N_E \big(  s (\alpha (a)) + \beta (u) + \lambda (a)  \big) = (N_E \circ \varphi) (e).
    \end{align*}
    This proves our claim that $\varphi$ is a Nijenhuis algebra automorphism (i.e. $\varphi \in \mathrm{Aut} (E, N_E)$). Since $\varphi (M) \subset M $, we get that $\varphi \in \mathrm{Aut}_M (E, N_E)$. Finally, for any $a \in A$ and $u \in M$,
    \begin{align*}
    \overline{\varphi} (a) =~&  p \varphi (s (a) + 0) = p \big( s (\alpha (a)) + \lambda (a) \big)   =\alpha (a) \quad (\because ~ \! ps = \mathrm{Id}_A \text{ and } p \big|_M = 0),\\
         \varphi (u) =~&  \varphi (s (0) + u)   = \beta (u).
    \end{align*}
    Hence $\tau (\varphi) = (\varphi \big|_M , \overline{ \varphi}) = (\beta, \alpha)$ which shows that the pair $(\beta, \alpha)$ is inducible.
\end{proof}

It is important to observe that the necessary and sufficient condition obtained in the above proposition depends on the section $s$. In the following, we aim to formulate a condition that doesn't depend on the choice of any section. Let $s$ be any section that induces the $2$-cocycle $(\chi, F)$ corresponding to the given abelian extension. For any pair $(\beta, \alpha) \in \mathrm{Aut} (M, N_M) \times \mathrm{Aut} (A, N)$, we define maps $\chi_{(\beta, \alpha)} \in \mathrm{Hom} (A^{\otimes 2}, M)$ and $F_{(\beta, \alpha)} \in \mathrm{Hom}(A, M)$ by 
\begin{align*}
 \chi_{(\beta, \alpha)} (a, b) = \beta \circ \chi (\alpha^{-1} (a), \alpha^{-1} (b)) ~~~~ \text{ and } ~~~~  F_{(\beta, \alpha)} (a) = \beta \circ F (\alpha^{-1} (a)), \text{ for } a, b \in A.
\end{align*}

With the above notations, we have the following.

\begin{prop}
   \begin{itemize}
       \item[(i)] The pair $(\chi_{(\beta, \alpha)}, F_{(\beta, \alpha)})$ is also a $2$-cocycle.
       \item[(ii)] Moreover, the cohomology class of the difference $(\chi_{(\beta, \alpha)}, F_{(\beta, \alpha)}) - (\chi, F)$ is independent of the choice of $s$.
   \end{itemize}
\end{prop}

\begin{proof}
    (i) Since $(\chi, F)$ is a $2$-cocycle, the identities (\ref{2co1}) and (\ref{2co2}) are hold. In both of these identities, if we replace $a, b, c $ by $\alpha^{-1} (a), \alpha^{-1}(b), \alpha^{-1} (c)$ respectively, we obtain the corresponding identities for the pair $(\chi_{(\beta, \alpha)}, F_{ (\beta, \alpha)})$. This shows that $(\chi_{(\beta, \alpha)}, F_{ (\beta, \alpha)})$ is a $2$-cocycle.
    
    \medskip
    
    (ii) Let $s' : A \rightarrow M$ be any other section of the map $p$. If $(\chi', F')$ is the $2$-cocycle corresponding to the abelian extension induced by the section $s'$, then the $2$-cocycles $(\chi, F)$ and $(\chi', F')$ are cohomologous. In fact, we have $(\chi, F) - (\chi', F') = \delta_\mathrm{NAlg} (g)$, where $g := s - s'$. Then it follows from the definitions of $(\chi_{(\beta, \alpha)}, F_{(\beta, \alpha)})$ and $(\chi'_{(\beta, \alpha)}, F'_{(\beta, \alpha)})$ that $ (\chi_{(\beta, \alpha)}, F_{(\beta, \alpha)}) - (\chi'_{(\beta, \alpha)}, F'_{(\beta, \alpha)}) = \delta_\mathrm{NAlg} (\beta g \alpha^{-1})$. Thus, we obtain that
    \begin{align*}
        \big\{ (\chi_{(\beta, \alpha)}, F_{(\beta, \alpha)}) - (\chi, F)   \big\} - \big\{  (\chi'_{(\beta, \alpha)}, F'_{(\beta, \alpha)}) - (\chi', F')  \big\} = \delta_\mathrm{NAlg} (\beta g \alpha^{-1} - g) \text{ is a } 2\text{-coboundary}.
    \end{align*}
    Hence the result follows.
\end{proof}

%Let ................  be an abelian extension of the Nijenhuis algebra .......... by ............ Given any section $s$ of the map $p$, let the above extension corresponds to the $2$-cocycle $(\chi, F) \in Z^2_\mathrm{NAlg} ((A, N), (M, N_M))$.
With the help of the above notations, we will now define a set-map (which we refer as the {\bf Wells map} in the context of Nijenhuis algebras)
\begin{align*}
 \mathcal{W} : \mathrm{Aut} (M, N_M) \times \mathrm{Aut} (A, N) \rightarrow H^2_\mathrm{NAlg} ((A, N); (M, N_M)) ~~~ \text{ by } ~~~   \mathcal{W} ((\beta, \alpha )) := \big[ (\chi_{(\beta, \alpha)}, F_{(\beta, \alpha)}) - (\chi, F)   \big],
\end{align*}
for $(\beta, \alpha) \in \mathrm{Aut} (M, N_M) \times \mathrm{Aut} (A, N)$. Then it follows from the previous proposition that the Wells map doesn't depend on the choice of the section $s$. In terms of the Wells map, Proposition \ref{pre-final-prop} can be rephrased as follows.

%We are now ready to prove the main result of this section.

\begin{thm}\label{theorem-wells}
    Let (\ref{ind-abel}) be an abelian extension of the Nijenhuis algebra $((A, ~ \! \cdot ~ \! ), N)$ by a Nijenhuis bimodule $((M, \triangleright, \triangleleft), N_M)$. Then a pair $(\beta, \alpha) \in  \mathrm{Aut} (M, N_M) \times \mathrm{Aut} (A, N)$ is inducible if and only if the pair $(\beta, \alpha)$ is `compatible' in the sense that
    \begin{align}\label{final-thm-no}
        \beta (a \triangleright u) = \alpha (a) \triangleright \beta (u) ~~~ \text{  and } ~~~ \beta (u \triangleleft a) = \beta (u) \triangleleft \alpha (a), \text{ for all } a \in A, u \in M
    \end{align}
    and the cohomology class $\mathcal{W} ((\beta, \alpha )) \in H^2_\mathrm{NAlg} ((A, N); (M, N_M))$ is trivial.
\end{thm}

\begin{proof}
    Let $(\beta, \alpha)$ be an inducible pair. Then by Proposition \ref{pre-final-prop}, we get that the identities in (\ref{final-thm-no}) are hold. Also, there exists a linear map $\lambda :A \rightarrow M$ satisfying (\ref{eqn-8}) and (\ref{eqn-9}). In these two identities, if we replace $a, b$ respectively by $\alpha^{-1} (a), \alpha^{-1} (b)$, we obtain
    \begin{align*}
        \chi_{(\beta, \alpha)} (a, b) - \chi (a, b) =~& a \triangleright \lambda \alpha^{-1} (b) + \lambda \alpha^{-1} (a) \triangleleft b - \lambda \alpha^{-1} (a \cdot b),\\
        F_{(\beta, \alpha)} (a) - F (a) =~& N_M (\lambda \alpha^{-1} (a)) - \lambda \alpha^{-1} (N(a)).
    \end{align*}
    This shows that $  ( \chi_{(\beta, \alpha)} , F_{(\beta, \alpha)} ) - (\chi, F) = \delta_\mathrm{NAlg} (\lambda \alpha^{-1})$ is a $2$-coboundary. Hence the cohomology class $\mathcal{W} ((\beta, \alpha)) = [   ( \chi_{(\beta, \alpha)} , F_{(\beta, \alpha)} ) - (\chi, F) ]$ is trivial.

    Conversely, suppose the identities in (\ref{final-thm-no}) are hold and $\mathcal{W} ((\beta, \alpha)) = 0$. Let $s$ be any section and $(\chi, F)$ be the $2$-cocycle corresponding to the given abelian extension induced by the section $s$. Since $\mathcal{W} ((\beta, \alpha)) = 0$, the $2$-cocycles $(\chi_{(\beta, \alpha)}, F_{(\beta, \alpha)})$ and $(\chi, F)$ are cohomologous, say $(\chi_{(\beta, \alpha)}, F_{(\beta, \alpha)}) - (\chi, F) = \delta_\mathrm{NAlg} (g)$, for some $g \in \mathrm{Hom} (A, M)$. Then it can be checked that the map $\lambda := g \alpha : A \rightarrow M$ satisfy the conditions (\ref{eqn-8}) and (\ref{eqn-9}). Hence by Proposition \ref{pre-final-prop}, the pair $(\beta, \alpha)$ is inducible. 
\end{proof}

Given an abelian extension of the Nijenhuis algebra $((A, ~ \! \cdot ~ \!), N)$ by the Nijenhuis bimodule $((M, \triangleright, \triangleleft), N_M)$, we define
\begin{align*}
C_{\triangleright, \triangleleft} = \{  (\beta, \alpha) \in \mathrm{Aut} (M, N_M) \times \mathrm{Aut} (A, N) ~ \! | ~    \substack{ \beta  (a ~ \! \triangleright ~ \! u) ~ \! = ~ \! \alpha
 (a) ~ \! \triangleright ~ \! \beta (u), \\
  \beta (u ~ \! \triangleleft ~ \! a) ~ \! = ~ \! \beta (u) ~ \! \triangleleft ~ \! \alpha (a),  } \text{ for all } a \in A, u \in M  \}.
\end{align*}
The space $C_{\triangleright, \triangleleft}$ is called the space of {\em compatible pairs} of Nijenhuis algebra automorphisms. Then $C_{\triangleright, \triangleleft}$ is a subgroup of $\mathrm{Aut} (M, N_M) \times \mathrm{Aut} (A, N)$ and it follows from the first part of the proof of Proposition \ref{pre-final-prop} that the image of the map $\tau$ lies inside $C_{\triangleright, \triangleleft}$, i.e. $\tau \big( \mathrm{Aut}_M (E, N_E) \big) \subset C_{\triangleright, \triangleleft}$. Moreover, by Theorem \ref{theorem-wells}, we get that a pair $(\beta, \alpha) \in \mathrm{Aut} (M, N_M) \times \mathrm{Aut} (A, N)$ of Nijenhuis algebra automorphisms is inducible if and only if $(\beta, \alpha) \in C_{\triangleright, \triangleleft}$ and $\mathcal{W} ((\beta, \alpha)) = 0$.

\medskip

Like classical cases, here we show that the Wells map defined above fits into a short exact sequence.

\begin{thm}
Let (\ref{ind-abel}) be an abelian extension of the Nijenhuis algebra $((A, ~ \! \cdot ~ \! ), N)$ by the Nijenhuis bimodule $((M, \triangleright, \triangleleft), N_M)$. Then there is an exact sequence (called the {\bf Wells exact sequence})
\begin{align}\label{wells-seq}
1 \rightarrow \mathrm{Aut}_{M, A} (E, N_E)
 \xrightarrow{\iota} \mathrm{Aut}_M (E, N_E) \xrightarrow{\tau} C_{\triangleright, \triangleleft} \xrightarrow{\mathcal{W}} H^2_\mathrm{NAlg} ((A, N); (M, N_M)),
\end{align}
where $\mathrm{Aut}_{M, A} (E, N_E) = \{ \varphi \in \mathrm{Aut} (E, N_E) ~ \! | ~ \! \tau (\varphi) = (\mathrm{Id}_M, \mathrm{Id}_A) \}.$
\end{thm}

\begin{proof}
The sequence (\ref{wells-seq}) is exact in the first place as the inclusion map $\iota : \mathrm{Aut}_{M, A} (E, N_E) \rightarrow \mathrm{Aut}_M (E, N_E)$ is an injection. Next, to show that the sequence is exact in the second place, we first take an element $\varphi \in \mathrm{ker} (\tau)$. That is, $\varphi \in \mathrm{Aut}_M (E, N_E)$ with the property that $\varphi \big|_M = \mathrm{Id}_M$ and $\overline{\varphi} = \mathrm{Id}_A$. This shows that $\varphi \in \mathrm{Aut}_{M, A} (E, N_E)$. On the other hand, $\varphi \in \mathrm{Aut}_{M, A} (E, N_E)$ implies that $\varphi \in \mathrm{ker} (\tau)$. Therefore, we get that $\mathrm{ker} (\tau) = \mathrm{Aut}_{M, A} (E, N_E) = \mathrm{im} (\iota)$.

Finally, we take an element $(\beta, \alpha) \in C_{\triangleright, \triangleleft}$ with $\mathcal{W} ((\beta, \alpha)) = 0$. Then it follows from Theorem \ref{theorem-wells} that the pair $(\beta, \alpha)$ is inducible. Thus, there exists a Nijenhuis algebra automorphism $\varphi \in \mathrm{Aut}_M (E, N_E)$ such that $\tau (\varphi) = (\beta, \alpha)$. This shows that $(\beta, \alpha) \in \mathrm{im} (\tau).$ Conversely, if the pair $(\beta, \alpha) \in \mathrm{im} (\tau)$ then by definition the pair $(\beta
, \alpha)$ is inducible. Hence again by Theorem \ref{theorem-wells}, we get that $(\beta, \alpha) \in C_{\triangleright, \triangleleft}$ and $\mathcal{W} ((\beta, \alpha)) = 0$, i.e. $(\beta, \alpha) \in \mathrm{ker} (\mathcal{W})$. Hence $\mathrm{ker} (\mathcal{W}) = \mathrm{im} (\tau)$ which shows that the sequence is also exact in the third place. This concludes the proof.
\end{proof}

In the following, we show that the group $\mathrm{Aut}_{M, A} (E, N_E)$ that appears in the Wells exact sequence is isomorphic to the space $Z^1_\mathrm{NAlg} ((A, N); (M, N_M))$ of $1$-cocycles of the Nijenhuis algebra $((A, ~ \! \cdot ~ \!), N)$ with coefficients in the Nijenhuis bimodule $((M, \triangleright, \triangleleft), N_M)$. More precisely, we have the following result.

\begin{prop}
With the hypothesis of the above theorem, there is a group isomorphism
\begin{align*}
\mathrm{Aut}_{M, A} (E, N_E) ~ \! \cong ~ \! Z^1_\mathrm{NAlg} ((A, N); (M, N_M).
\end{align*}
As a consequence, the group $\mathrm{Aut}_{M, A} (E, N_E)$ is abelian.
\end{prop}

\begin{proof}
Let $\varphi \in \mathrm{Aut}_{M, A} (E, N_E)$. Then for any $a \in A$, we observe that $p (( \varphi s - s ) (a)) = ( p \varphi s - ps ) (a) = 0$ which implies that $(\varphi s - s) (a) \in \mathrm{ ker} ( p ) = \mathrm{im } (i) \cong M$. We now define a map
$\widetilde{\varphi} : A \rightarrow M$ by
\begin{align*}
\widetilde{\varphi} (a) := (\varphi s - s ) (a) , \text{ for } a \in A.
\end{align*}
For any $a, b \in A$, we have
\begin{align*}
\widetilde{\varphi} (a \cdot b) =~& \varphi s (a \cdot b) - s (a \cdot b)\\
=~& \varphi \big(  s (a) \cdot_E s (b) - \chi (a, b)   \big) - \big(    s (a) \cdot_E s (b) - \chi (a, b)    \big)\\
=~& \varphi s (a) \cdot_E \varphi s (b) - s (a) \cdot_E s (b) \quad (\because ~ \! \varphi \big|_M = \mathrm{Id}_M) \\
=~& (\widetilde{\varphi} + s )(a) \cdot_E (\widetilde{\varphi} + s )(b) - s(a) \cdot_E s (b) \quad (\because ~ \! \varphi s = \widetilde{\varphi} + s)\\
=~& s (a) \cdot_E \widetilde{\varphi} (b) + \widetilde{\varphi} (a) \cdot_E s(b) = a \triangleright \widetilde{\varphi} (b) + \widetilde{\varphi} (a) \triangleleft b
\end{align*}
and
\begin{align*}
(N_M \circ \widetilde{\varphi} - \widetilde{\varphi} \circ N) (a) =~& N_M ( \widetilde{\varphi} (a)) - (\varphi s N (a) - sN (a))\\
=~& N_M ( \widetilde{\varphi} (a)) - \varphi (N_E s - F) (a) + (N_E s - F)(a) \quad (\because ~ \! F = N_E s - sN)\\
 =~& N_M ( \widetilde{\varphi} (a)) - N_E (\varphi s (a)) + N_E s (a) \quad (\because  ~ \! \varphi \big|_M = \mathrm{Id}_M \text{ and } \varphi N_E = N_E \varphi)\\
 =~& N_M  ( \widetilde{\varphi} (a)) - N_E (\varphi s - s ) (a) = N_M  ( \widetilde{\varphi} (a))- N_M (\widetilde{\varphi} (a)) = 0.
\end{align*}
This shows that $\widetilde{\varphi} \in Z^1_\mathrm{NAlg} ((A, N); (M, N_M))$. Finally, for any $\varphi, \psi \in \mathrm{Aut}_{M , A} (E, N_E)$,
\begin{align*}
\widetilde{\varphi \psi} = \varphi \psi s - s = \varphi s - s + \varphi (\psi s - s) = \widetilde{\varphi} + \varphi \widetilde{\psi} = \widetilde{ \varphi} + \widetilde{\psi} \quad (\because ~ \mathrm{im}(\widetilde{ \psi} ) \subset M \text{ and } \varphi \big|_M = \mathrm{Id}_M)
\end{align*}
which implies that the map $\widetilde{ ( ~)} : \mathrm{Aut}_{M, A} (E, N_E) \rightarrow Z^1_\mathrm{NAlg} ((A, N); (M, N_M))$ is a group homomorphism.

Next, we claim that the map $\widetilde{ ( ~)}$ is bijective. First, we take an element $\varphi \in \mathrm{Aut}_{M, A} (E, N_E)$ with $\widetilde{\varphi} = 0$. Then for any $e = s(a) + u \in E$, we observe that
\begin{align*}
\varphi (e) = \varphi (s (a) + u) = (\varphi s - s ) (a) + s (a) + \varphi (u) = s (a) + u = e \quad (\because ~ \! \widetilde{\varphi} = 0 \text{ and } \varphi \big|_M = \mathrm{Id}_M)
\end{align*}
which implies that $\varphi = \mathrm{Id}_E$. This shows that the map $\widetilde{ ( ~)}$ is injective. Next, to show that it is surjective, we take an arbitrary element $d \in Z^1_\mathrm{NAlg} ((A, N); (M, N_M))$. Hence $d : A \rightarrow M$ is a linear map satisfying the usual derivation property and $N_M \circ d = d \circ N$. We now define a map $\underline{d} : E \rightarrow E$ by
\begin{align*}
\underline{d} (e) = \underline{d} (s (a) + u) := s (a) + (u + d (a)), \text{ for } e =  s(a) + u \in E.
\end{align*}
Then it is easy to see that the map $\underline{d}$ is bijective and satisfies additionally $\underline{d} (e \cdot_E e') = e \cdot_E \underline{d} (e') + \underline{d}(e) \cdot_E e'$ and $N_E \circ \underline{d} = \underline{d} \circ N_E$, for all $e, e' \in E$. Hence $\underline{d} \in \mathrm{Aut}(E, N_E)$. Moreover, for all $a \in A$, $u \in M$, we have
\begin{align*}
p \underline{d} s (a) = p ( s(a) + d (a)) = ps (a) = a ~~~~ \text{ and } ~~~~ \underline{d} (u) = \underline{d} (s (0) + u) = u.
\end{align*}
Hence $\tau (\underline{d}) = (\underline{d} \big|_M , p \underline{d} s) = (\mathrm{Id}_M, \mathrm{Id}_A)$ which in turn implies that $\underline{d} \in \mathrm{Aut}_{M, A} (E, N_E)$. Finally, 
\begin{align*}
\widetilde{ (\underline{d})} (a) = (\underline{d}s - s)(a) = s(a) + d(a) - s(a) = d(a), \text{ for all } a \in A. 
\end{align*}
This shows that $\widetilde{ (\underline{d})} = d$ proving that the map $\widetilde{ (~)}$ is surjective. Hence the proof follows.
\end{proof}

\begin{corollary}
Given the above proposition, the Wells exact sequence can be alternatively read as an exact sequence
\begin{align*}
0 \rightarrow Z^1_\mathrm{NAlg} ((A, N); (M, N_M)) \rightarrow \mathrm{Aut}_M (E, N_E) \xrightarrow{\tau} C_{\triangleright, \triangleleft} \xrightarrow{ \mathcal{W}} H^2_\mathrm{NAlg} ((A, N); (M, N_M)).
\end{align*}
\end{corollary}

\medskip

\section{Skeletal and strict 2-term homotopy Nijenhuis algebras and homotopy NS-algebras}\label{sec6}
In this section, we first introduce homotopy Nijenhuis operators on $2$-term $A_\infty$-algebras. Next, we consider skeletal and strict $2$-term homotopy Nijenhuis algebras and characterize them by third cocycles of Nijenhuis algebras and crossed modules of Nijenhuis algebras, respectively. Finally, we also consider strict homotopy Nijenhuis operators on arbitrary $A_\infty$-algebras and show that such operators induce $NS_\infty$-algebras.

\begin{defn} \cite{keller}
An {\bf $A_\infty$-algebra} is a pair $(\mathcal{A}, \{ \mu_n \}_{n \geq 1})$ consisting of a graded vector space $\mathcal{A} = \oplus_{i \in \mathbb{Z}} \mathcal{A}_i$ equipped with a collection $\{ \mu_n : \mathcal{A}^{\otimes n} \rightarrow \mathcal{A} \}_{n \geq 1}$ of graded linear maps with $\mathrm{deg} (\mu_n) = n-2$ (for $n \geq 1$) that satisfy the following set of identities:

For each $k \in \mathbb{N}$ and homogeneous elements $a_1, \ldots, a_k \in \mathcal{A}$,
\begin{align}\label{ainf-iden}
\sum_{m+n = k+1} \sum_{i=1}^m (-1)^{ i (n+1) + n ( |a_{1}| + \cdots + |a_{i-1}|)} ~ \! \mu_m \big(  a_1, \ldots, a_{i-1}, \mu_n (a_i, \ldots, a_{i+n-1}), a_{i+n}, \ldots, a_k  \big) = 0.
\end{align}
\end{defn}

It is easy to observe that a differential graded associative algebra is an $A_\infty$-algebra $(\mathcal{A}, \{ \mu_n \}_{n \geq 1})$ for which $\mu_n = 0$, $n \geq 3$. Additionally, if $\mu_1 = 0$, one obtains a graded associative algebra.

An $A_\infty$-algebra whose underlying graded vector space $\mathcal{A}$ is concentrated in arities $0$ and $1$ is called a $2$-term $A_\infty$-algebra. The explicit description is given in the following definition.

\begin{defn}
    A {\bf $2$-term $A_\infty$-algebra} $ ( \mathcal{A}_1 \xrightarrow{ \partial} \mathcal{A}_0, \mu_2, \mu_3)$ is a $2$-term chain complex $\mathcal{A}_1 \xrightarrow{\partial} \mathcal{A}_0$ endowed with bilinear maps $\mu_2 : \mathcal{A}_i \times \mathcal{A}_j \rightarrow \mathcal{A}_{i+j}$ (for $0 \leq i, ~ \! j , ~ \! i+j \leq 1$) and a trilinear map $\mu_3 : \mathcal{A}_0 \times \mathcal{A}_0 \times \mathcal{A}_0 \rightarrow \mathcal{A}_1$ subject to satisfy the following set of identities: for $a, b, c, d \in \mathcal{A}_0$ and $u, v \in \mathcal{A}_1$,
    \begin{align}
        \partial \mu_2 (a, u) =~& \mu_2 (a, \partial u), \label{ainf-1}\\
        \partial \mu_2 (u, a ) =~& \mu_2 (\partial u, a), \label{ainf-2}\\
        \mu_2 (\partial u, v ) =~& \mu_2 (u, \partial v), \label{ainf-3}\\
        \partial \mu_3 (a, b, c) =~& \mu_2 (\mu_2 (a, b), c) - \mu_2 (a, \mu_2 (b, c)), \label{ainf-4}\\
        \mu_3 (a, b, \partial u) =~& \mu_2 (\mu_2 (a, b), u) - \mu_2 (a, \mu_2 (b, u)), \label{ainf-5}\\
        \mu_3 (a, \partial u, b) =~& \mu_2 (\mu_2 (a, u), b) - \mu_2 (a, \mu_2 (u, b)), \label{ainf-6}\\
        \mu_3 (\partial u, a, b) =~& \mu_2 ( \mu_2 (u, a), b) - \mu_2 (u, \mu_2 (a, b)), \label{ainf-7}\\
        \mu_2 (a, \mu_3 (b, c, d)) - \mu_3 (\mu_2 (a, b), c, d) ~+& ~ \mu_3 (a, \mu_2 (b, c), d) - \mu_3 (a, b, \mu_2 (c, d)) + \mu_2 (\mu_3 (a, b, c), d) = 0. \label{ainf-8}
    \end{align}
\end{defn}

A $2$-term $A_\infty$-algebra $ ( \mathcal{A}_1 \xrightarrow{ \partial} \mathcal{A}_0, \mu_2, \mu_3)$ is said to be 
{\em skeletal} if $\partial = 0$ and
{\em strict} if $\mu_3 = 0$. It is well-known that skeletal $2$-term $A_\infty$-algebras are closely related to third Hochschild cocycles of associative algebras and strict $2$-term $A_\infty$-algebras are intimately related to crossed modules of associative algebras.

\begin{defn}
    (i) Let $( \mathcal{A}_1 \xrightarrow{ \partial} \mathcal{A}_0, \mu_2, \mu_3)$ be a $2$-term $A_\infty$-algebra. A {\bf homotopy Nijenhuis operator} on $( \mathcal{A}_1 \xrightarrow{ \partial} \mathcal{A}_0, \mu_2, \mu_3)$ is a triple $(\mathcal{N}_0, \mathcal{N}_1, \mathcal{N}_2)$ consisting of (bi)linear maps
        \begin{align*}
            \mathcal{N}_0 : \mathcal{A}_0 \rightarrow \mathcal{A}_0 , \quad  \mathcal{N}_1 : \mathcal{A}_1 \rightarrow \mathcal{A}_1 ~~~~ \text{ and } ~~~~ \mathcal{N}_2 : \mathcal{A}_0 \times \mathcal{A}_0 \rightarrow \mathcal{A}_1
        \end{align*}
        that satisfy the following list of identities: for all $a, b, c \in \mathcal{A}_0$ and $u \in \mathcal{A}_1$,
        \begin{align}
            \partial \circ \mathcal{N}_1 = \mathcal{N}_0 \circ \partial, \qquad \qquad \qquad \qquad \qquad & \label{infnij-1}\\
            \mathcal{N}_0 \big(  \mu_2 ( \mathcal{N}_0 (a), b) + \mu_2 (a, \mathcal{N}_0 (b)) - \mathcal{N}_0 (\mu_2 (a, b))    \big) - \mu_2 ( \mathcal{N}_0 (a) , \mathcal{N}_0 (b)) =~& \partial \big(  \mathcal{N}_2 (a, b)  \big), \label{infnij-2}\\
            \mathcal{N}_1 \big(  \mu_2 (\mathcal{N}_0 (a), u ) + \mu_2 (a, \mathcal{N}_1 (u)) - \mathcal{N}_1 (\mu_2 (a, u))   \big) - \mu_2 (\mathcal{N}_0 (a), \mathcal{N}_1 (u) ) =~& \mathcal{N}_2 (a, \partial u), \label{infnij-3}\\
            \mathcal{N}_1 \big(  \mu_2 (\mathcal{N}_1 (u), a ) + \mu_2 (u, \mathcal{N}_0 (a)) - \mathcal{N}_1 (\mu_2 (u,a))   \big) - \mu_2 (\mathcal{N}_1 (u), \mathcal{N}_0 (a) ) =~& \mathcal{N}_2 ( \partial u, a), \label{infnij-4}
\end{align}            
           \begin{align} 
           \mu_2 \big( \mathcal{N}_0 (a), \mathcal{N}_2 (b, c)   \big) ~ - ~& \mu_2 \big(  \mathcal{N}_2 (a, b), \mathcal{N}_0 (c)   \big) - \mathcal{N}_2 \big( \mu_2 ( \mathcal{N}_0 (a), b) + \mu_2 (a, \mathcal{N}_0 (b)) - \mathcal{N}_0 ( \mu_2 (a, b)) ~ \! , ~ \! c    \big) \label{infnij-5} \\
          & + \mathcal{N}_2 \big( a ~ \!, ~ \! \mu_2 ( \mathcal{N}_0 (b), c) + \mu_2 (b, \mathcal{N}_0 (c)) - \mathcal{N}_0 ( \mu_2 (b, c)) \big) \nonumber \\
           - \mathcal{N}_1  \big(  \mu_2 & (a, \mathcal{N}_2 (b, c)) - \mathcal{N}_2 ( \mu_2 (a, b), c) + \mathcal{N}_2 (a, \mu_2 (b, c)) - \mu_2 (\mathcal{N}_2 (a, b) , c)     \big) \nonumber \\
  = \mu_3 \big( \mathcal{N}_0 (a), \mathcal{N}_0 (b), \mathcal{N}_0 (c)  \big) - & \mathcal{N}_1 \mu_3 \big(  \mathcal{N}_0 (a), \mathcal{N}_0 (b) , c  \big) - \mathcal{N}_1 \mu_3 \big( \mathcal{N}_0 (a), b, \mathcal{N}_0 (c)    \big) - \mathcal{N}_1 \mu_3 \big(  a, \mathcal{N}_0 (b), \mathcal{N}_0 (c)  \big) \nonumber \\
 + \mathcal{N}_1^2 \mu_3 \big(  \mathcal{N}_0 (a),& b, c  \big) + \mathcal{N}_1^2 \mu_3 \big(  a, \mathcal{N}_0 (b), c  \big) + \mathcal{N}_1^2 \mu_3 \big(  a, b, \mathcal{N}_0 (c)  \big) - \mathcal{N}_1^3 \mu_3 (a, b, c). \nonumber
        \end{align}

(ii) A pair $(  ( \mathcal{A}_1 \xrightarrow{ \partial} \mathcal{A}_0, \mu_2, \mu_3), (\mathcal{N}_0, \mathcal{N}_1, \mathcal{N}_2))$ consisting of a $2$-term $A_\infty$-algebra $ ( \mathcal{A}_1 \xrightarrow{ \partial} \mathcal{A}_0, \mu_2, \mu_3)$ with a homotopy Nijenhuis operator  $(\mathcal{N}_0, \mathcal{N}_1, \mathcal{N}_2)$ is called a {\bf $2$-term homotopy Nijenhuis algebra}.
\end{defn}

A $2$-term homotopy Nijenhuis algebra $(  ( \mathcal{A}_1 \xrightarrow{ \partial} \mathcal{A}_0, \mu_2, \mu_3), (\mathcal{N}_0, \mathcal{N}_1, \mathcal{N}_2))$ is said to be 
\begin{itemize}
    \item {\em skeletal} if the underlying $2$-term $A_\infty$-algebra is skeletal (i.e. $\partial = 0$),
    \item {\em strict} if the underlying $2$-term $A_\infty$-algebra is strict (i.e. $\mu_3 = 0$) and $\mathcal{N}_2 = 0$.
\end{itemize}

\begin{thm}
    There is a 1-1 correspondence between skeletal $2$-term homotopy Nijenhuis algebras and third cocycles of Nijenhuis algebras with coefficients in Nijenhuis bimodules.
\end{thm}

\begin{proof}
    Let $(  ( \mathcal{A}_1 \xrightarrow{ \partial = 0} \mathcal{A}_0, \mu_2, \mu_3), (\mathcal{N}_0, \mathcal{N}_1, \mathcal{N}_2))$ be a skeletal $2$-term homotopy Nijenhuis algebra. Since $\partial = 0$, it follows from (\ref{ainf-4}) and (\ref{infnij-2}) that $( (\mathcal{A}_0, \mu_2), \mathcal{N}_0)$ is a Nijenhuis algebra. On the other hand, from the identities (\ref{ainf-5})-(\ref{ainf-7}) and (\ref{infnij-3})-(\ref{infnij-4}), we get that $( (\mathcal{A}_1, \triangleright = \triangleleft = \mu_2) , \mathcal{N}_1)$ is a Nijenhuis bimodule over the Nijenhuis algebra $( (\mathcal{A}_0, \mu_2), \mathcal{N}_0)$. Finally, the identity (\ref{ainf-8}) simply says that $(\delta_\mathrm{Hoch} \mu_3) (a, b, c, d) = 0$ while the identity (\ref{infnij-5}) can be equivalently expressed as $ d_{\mathcal{N}_0, \mathcal{N}_1} (\mathcal{N}_2) (a, b, c) = \partial^{\mathcal{N}_0, \mathcal{N}_1} (\mu_3) (a, b, c)$. Here $\delta_\mathrm{Hoch}$ is the Hochschild coboundary operator of the associative algebra $(\mathcal{A}_0, \mu_2)$ with coefficients in the bimodule $(\mathcal{A}_1, \triangleright = \triangleleft = \mu_2)$. Hence 
    \begin{align*}
        \delta_\mathrm{NAlg} (\mu_3, \mathcal{N}_2) = \big( \delta_\mathrm{Hoch} (\mu_3) ~ \! , ~ \!     d_{\mathcal{N}_0, \mathcal{N}_1} (\mathcal{N}_2) - \partial^{\mathcal{N}_0, \mathcal{N}_1} (\mu_3) \big) = 0
    \end{align*}
    which shows that $(\mu_3, \mathcal{N}_2) \in {Z}^3_\mathrm{NAlg} ( (\mathcal{A}_0, \mathcal{N}_0); ( \mathcal{A}_1, \mathcal{N}_1))$ is a $3$-cocycle. Hence we obtain a Nijenhuis algebra, a Nijenhuis bimodule over it and a $3$-cocycle of the Nijenhuis algebra with coefficients in the Nijenhuis bimodule.

    Conversely, let $((A, ~ \! \cdot ~ \!), N)$ be a Nijenhuis algebra, $( (M, \triangleright, \triangleleft), N_M )$ be a Nijenhuis bimodule and a $3$-cocycle $(\chi, F) \in Z^3_\mathrm{NAlg} ((A, N); (M, N_M))$. Then it is easy to check that
    \begin{align*}
        \big(  (M \xrightarrow{\partial = 0} A, \mu_2, \mu_3 = \chi), (N, N_M, F)  \big)
    \end{align*}
    is a skeletal $2$-term homotopy Nijenhuis algebra, where the map $\mu_2$ is given by
    \begin{align*}
        \mu_2 (a, b) = a \cdot b, \quad \mu_2 (a, u) = a \triangleright u ~~~~ \text{ and } ~~~~ \mu_2 (u, a) = u \triangleleft a, \text{ for } a, b \in A \text{ and } u \in A.
    \end{align*}
    The above two constructions are inverses to each other. Hence the proof follows.
\end{proof}

To obtain a characterization result for strict $2$-term homotopy Nijenhuis algebras, here we first consider the notion of crossed modules of Nijenhuis algebras.

\begin{defn}
    A {\bf crossed module of Nijenhuis algebras} is a tuple $\big(  ((A, ~ \! \cdot ~ \!), N), ((A_1, ~ \! \cdot_1 ~ \!), N_1), \varphi, \triangleright, \triangleleft  \big)$ consisting of Nijenhuis algebras $((A, ~ \! \cdot ~ \!), N)$ and $((A_1, ~ \! \cdot_1 ~ \!), N_1)$ together with a homomorphism $\varphi : A_1 \rightarrow A$ of Nijenhuis algebras  and two bilinear maps $\triangleright : A \times A_1 \rightarrow A_1$ and $\triangleleft : A_1 \times A \rightarrow A_1$ such that
    \begin{itemize}
        \item[(i)] $( (A_1, \triangleright, \triangleleft), N_1)$ is a Nijenhuis bimodule over the Nijenhuis algebra $((A, ~ \! \cdot ~ \!), N)$,
        \item[(ii)] the following compatibility conditions also hold: 
        \begin{align}
           & a \triangleright  (u \cdot_1 v) = (a \triangleright u) \cdot_1 v, \qquad (u \triangleleft a) \cdot_1 v = u \cdot_1 (a \triangleright v), \qquad (u \cdot_1 v) \triangleleft a = u \cdot_1 (v \triangleleft a), \label{crossed1}\\
           & \qquad  \varphi (a \triangleright u) = a \cdot \varphi (u), \qquad \varphi (u \triangleleft a) = \varphi (u) \cdot a, \qquad \varphi (u ) \triangleright v = u \cdot_1 v = u \triangleleft \varphi (v), \label{crossed2}
        \end{align}
        for all $a \in A$ and $u, v \in A_1$.
    \end{itemize}
\end{defn}

\begin{exam}
Let $((A, ~ \! \cdot ~ \!), N)$ be any Nijenhuis algebra. Then $ \big( ((A, ~ \! \cdot ~ \!), N), ((A, ~ \! \cdot ~ \!), N), \varphi = \mathrm{Id}_A, \triangleright_\mathrm{ad}, \triangleleft_\mathrm{ad} \big)$ is a crossed module of Nijenhuis algebras.
\end{exam}

\begin{exam}
Let $( (A, ~ \! \cdot ~ \!), (A_1, ~ \! \cdot_1 ), \varphi, \triangleright, \triangleleft )$ be a crossed module of associative algebras. 
\begin{itemize}
\item[(i)] Then the tuple $\big( ((A, ~ \! \cdot ~ \!), N=\mathrm{Id}_A), ((A_1, ~ \! \cdot_1 ), N_1 = \mathrm{Id}_{A_1}), \varphi, \triangleright, \triangleleft \big)$ is obviously a crossed module of Nijenhuis algebras.
\item[(ii)] For any $u \in A_1$, consider the left multiplications $l_{ \varphi (u)} : A \rightarrow A$, $a \mapsto \varphi (u) \cdot a$ and $l_{1, u} : A_1 \rightarrow A_1$, $v \mapsto u \cdot_1 v$. We have seen that they are Nijenhuis operators on the respective algebras. Then the tuple $   ( ((A, ~ \! \cdot ~ \!),   l_{ \varphi (u)}  ), ((A_1, ~ \! \cdot_1 ), l_{1, u} ), \varphi, \triangleright, \triangleleft )$ is a crossed module of Nijenhuis algebras.
\end{itemize}
\end{exam}

\begin{exam}
Let $\big(  ((A, ~ \! \cdot ~ \!), N), ((A_1, ~ \! \cdot_1 ~ \!), N_1), \varphi, \triangleright, \triangleleft  \big)$ be any crossed module of Nijenhuis algebras. Then for any $k \geq 0$, the tuple $\big(  ((A, ~ \! \cdot ~ \!), N^k), ((A_1, ~ \! \cdot_1 ~ \!), N_1^k), \varphi, \triangleright, \triangleleft  \big)$ is also a crossed module of Nijenhuis algebras.
\end{exam}

We are now ready to prove the characterization result for strict $2$-term homotopy Nijenhuis algebras.

\begin{thm}
    There is a 1-1 correspondence between strict $2$-term homotopy Nijenhuis algebras and crossed modules of Nijenhuis algebras.
\end{thm}

\begin{proof}
    Let $(  ( \mathcal{A}_1 \xrightarrow{ \partial} \mathcal{A}_0, \mu_2, \mu_3 = 0), (\mathcal{N}_0, \mathcal{N}_1, \mathcal{N}_2 = 0))$ be a strict $2$-term homotopy Nijenhuis algebra. Then it follows from (\ref{ainf-4}) and (\ref{infnij-2}) that $((\mathcal{A}_0, \mu_2), \mathcal{N}_0)$ is a Nijenhuis algebra. We define a bilinear map $\cdot_1 : \mathcal{A}_1 \times \mathcal{A}_1 \rightarrow \mathcal{A}_1$ by $u \cdot_1 v := \mu_2 (\partial u, v) = \mu_2 (u, \partial v)$, for $u, v \in \mathcal{A}_1$. Then it turns out that $(\mathcal{A}_1, \cdot_1)$ is an associative algebra (can be deduced either from  (\ref{ainf-5}), (\ref{ainf-6}), (\ref{ainf-7})) and $\mathcal{N}_1 : \mathcal{A}_1 \rightarrow \mathcal{A}_1$ is a Nijenhuis operator on it (follows either from (\ref{infnij-3}), (\ref{infnij-4})). Hence $((\mathcal{A}_1, \cdot_1), \mathcal{N}_1)$ is also a Nijenhuis algebra. Further, from (\ref{ainf-1}) (or (\ref{ainf-2})) and (\ref{infnij-1}), we get that $\partial : \mathcal{A}_1 \rightarrow \mathcal{A}_0$ is a homomorphism of Nijenhuis algebras. Next, we define maps $\triangleright : \mathcal{A}_0 \times \mathcal{A}_1 \rightarrow \mathcal{A}_1$ and $\triangleleft : \mathcal{A}_1 \times \mathcal{A}_0 \rightarrow \mathcal{A}_1$ by $a \triangleright u : = \mu_2 (a, u)$ and $u \triangleleft a := \mu_2 (u, a)$, for $a \in \mathcal{A}_0$ and $u \in \mathcal{A}_1$. Then from (\ref{ainf-5})-(\ref{ainf-7}) and (\ref{infnij-3})-(\ref{infnij-4}) that $((\mathcal{A}_1, \triangleright, \triangleleft), \mathcal{N}_1)$ is a Nijenhuis bimodule over the Nijenhuis algebra $((\mathcal{A}_0, \mu_2 ), \mathcal{N}_0 )$. Moreover, it is easy to see that all the identities in (\ref{crossed1}) and (\ref{crossed2}) are held. Hence the tuple $\big( ((\mathcal{A}_0, \mu_2 ), \mathcal{N}_0 ), ((\mathcal{A}_1, \cdot_1), \mathcal{N}_1), \partial, \triangleright, \triangleleft \big)$ is a crossed module of Nijenhuis algebras.
    
    \medskip

    Conversely, let $\big(  ((A, ~ \! \cdot ~ \!), N), ((A_1, ~ \! \cdot_1 ~ \!), N_1), \varphi, \triangleright, \triangleleft  \big)$ be a crossed module of Nijenhuis algebras. Then one can verify that 
    $ ( (A_1 \xrightarrow{\varphi} A, \mu_2, \mu_3 = 0), (N, N_1, \mathcal{N}_2 = 0))$ is a strict $2$-term homotopy Nijenhuis algebra, where
    \begin{align*}
        \mu_2 (a, b) = a \cdot b, \quad \mu_2 (a, u) = a \triangleright u ~~~ \text{ and } ~~~ \mu_2 (u, a) = u \triangleleft a, \text{ for } a \in A, u \in A_1.
    \end{align*}
    Finally, the above two constructions are inverses to each other. This completes the proof.
\end{proof}

\medskip

We will now introduce strict homotopy Nijenhuis operators on arbitrary $A_\infty$-algebras. However, a generic homotopy Nijenhuis operator on an $A_\infty$-algebra is yet to be found. We also consider $NS_\infty$-algebras and show that strict homotopy Nijenhuis operators on $A_\infty$-algebras induce $NS_\infty$-algebras.

\begin{defn}
Let $(\mathcal{A}, \{ \mu_n \}_{n \geq 1})$ be an $A_\infty$-algebra. A {\bf strict homotopy Nijenhuis operator} on $(\mathcal{A}, \{ \mu_n \}_{n \geq 1})$ is a degree $0$ linear map $\mathcal{N} : \mathcal{A} \rightarrow \mathcal{A}$ that satisfies
\begin{align}\label{strict-hn}
&\mu_n \big( \mathcal{N}(a_1), \ldots, \mathcal{N} (a_n) \big) = \mathcal{N} \bigg(   \sum_{i=1}^n \mu_n \big(  \mathcal{N}(a_1), \ldots, a_i, \ldots, \mathcal{N} (a_n)  \big) \\
& \qquad \qquad - \sum_{1 \leq i < j \leq n} \mathcal{N} \big(  \mu_n \big(  \mathcal{N}(a_1), \ldots, a_i, \ldots, a_j, \ldots, \mathcal{N}(a_n) \big)   \big) + \cdots
+ (-1)^{n-1} \mathcal{N}^{n-1} \mu_n (a_1, \ldots, a_n) \bigg), \nonumber
\end{align}
for all $n \geq 1$ and homogeneous elements $a_1, \ldots, a_n \in \mathcal{A}$.
\end{defn}

An $A_\infty$-algebra whose underlying graded vector space is concentrated in arity $0$ is nothing but an associative algebra. In such a case, a strict homotopy Nijenhuis operator simply becomes a Nijenhuis operator on the associative algebra. Moreover, for any arbitrary $A_\infty$-algebra $(\mathcal{A} , \{ \mu_n \}_{n \geq 1})$, the identity map $\mathrm{Id}_\mathcal{A} : \mathcal{A} \rightarrow \mathcal{A}$ is a strict homotopy Nijenhuis operator on $(\mathcal{A} , \{ \mu_n \}_{n \geq 1})$.

\medskip

In \cite{das-dend} the author has introduced the notion of a strict homotopy relative Rota-Baxter operator on an $A_\infty$-algebra with respect to a representation. Let $(\mathcal{A}, \{ \mu_n \}_{n \geq 1})$ be an $A_\infty$-algebra. Recall that a {\em representation} of this $A_\infty$-algebra is given by a pair $(\mathcal{M}, \{ \nu_n \}_{n \geq 1})$ of a graded vector space $\mathcal{M} = \oplus_{i \in \mathbb{Z}} \mathcal{M}_i$ with a collection of graded linear maps 
\begin{align*}
\{ \nu_n : \oplus_{r=1}^n (\mathcal{A}^{\otimes r-1} \otimes \mathcal{M} \otimes \mathcal{A}^{n-r}) \rightarrow \mathcal{M} \}_{n \geq 1} ~\text{ with }~ \mathrm{deg} (\nu_n) = n-2 ~~~ (\text{for } n \geq 1 )
\end{align*}
that satisfy the identities in (\ref{ainf-iden}) when exactly one of $a_1, \ldots, a_k$ is taken from $\mathcal{M}$ and the corresponding graded linear operation $\mu_m$ or $\mu_n$ is replaced by $\nu_m$ or $\nu_n$. 
%Any $A_\infty$-algebra $(\mathcal{A}, \{ \mu_n \}_{n \geq 1})$ can be regarded as a representation of itself, where $\nu_n = \mu_n$, for all $n$.
In this case, the direct sum $\mathcal{A} \oplus \mathcal{M}$ can be given an $A_\infty$-algebra structure $(\mathcal{A} \oplus \mathcal{M}, \{ \mu_n^\ltimes \}_{n \geq 1})$, where
\begin{align*}
\mu_n^\ltimes \big( (a_1, u_1), \ldots,  (a_n, u_n) \big) := \big(  \mu_n (a_1, \ldots, a_n) ~ \! , ~ \! \sum_{r=1}^n \nu_n (a_1, \ldots, u_r, \ldots, a_n)  \big),
\end{align*}
for homogeneous $(a_1, u_1), \ldots, (a_n, u_n) \in \mathcal{A} \oplus \mathcal{M}$. This is called the semidirect product.

\medskip

Let $(\mathcal{A}, \{ \mu_n \}_{n \geq 1})$ be an $A_\infty$-algebra and $(\mathcal{M} , \{ \nu_n \}_{n \geq 1})$ be a representation of it. A degree $0$ linear map $\mathcal{R} : \mathcal{M} \rightarrow \mathcal{A}$ is called a {\em strict homotopy relative Rota-Baxter operator} \cite{das-dend} if for all $n \geq 1$ and homogeneous elements $u_1, \ldots, u_n \in \mathcal{M}$,
\begin{align*}
\mu_n \big( \mathcal{R} (u_1), \ldots, \mathcal{R} (u_n) \big) = \sum_{r=1}^n \mathcal{R} \big( \nu_n \big(  \mathcal{R} (u_1), \ldots, u_r, \ldots, \mathcal{R} (u_n)   \big)    \big).
\end{align*}

The next result shows that a strict homotopy relative Rota-Baxter operator lifts to a strict homotopy Nijenhuis operator on the semidirect product. This generalizes Example \ref{exam-rrb} (ii) in the homotopy context.

\begin{prop}
With the above notations, a degree $0$ linear map $\mathcal{R} : \mathcal{M} \rightarrow \mathcal{A}$ is a strict homotopy relative Rota-Baxter operator if and only if the map $\widetilde{\mathcal{R}} : \mathcal{A} \oplus \mathcal{M} \rightarrow \mathcal{A} \oplus \mathcal{M}$ given by $\widetilde{\mathcal{R}} (a, u) := ( \mathcal{R} (u), 0)$ is a strict homotopy Nijenhuis operator on the semidirect product $A_\infty$-algebra $(\mathcal{A} \oplus \mathcal{M}, \{ \mu_n^\ltimes \}_{n \geq 1})$.
\end{prop}

In the following, we introduce $NS_\infty$-algebras as the strongly homotopy analogue of NS-algebras. We observe that $NS_\infty$-algebras are splitting of $A_\infty$-algebras and they can be obtained from strict homotopy Nijenhuis operators as defined above.

\begin{defn}
An {\bf $NS_\infty$-algebra} (also called a {\em strongly homotopy NS-algebra}) is a pair $(\mathcal{A}, \{ \eta_n \}_{n \geq 1})$ of a graded vector space $\mathcal{A} = \oplus_{i \in \mathbb{Z}} \mathcal{A}_i$ with a collection
\begin{align*}
\{ \eta_1 : {\bf k}[C_1] \otimes \mathcal{A} \rightarrow \mathcal{A} ~ \!, ~ \! \eta_{n \geq 2} : {\bf k}[C_{n+1}] \otimes \mathcal{A}^{\otimes n} \rightarrow \mathcal{A} \}
\end{align*}
of graded linear maps with $\mathrm{deg} (\eta_n) = n-2$ (for $n \geq 1$) subject to satisfy the following set of identities:

For each $k \in \mathbb{N}$, $[r] \in C_{k+1}$ and homogeneous elements $a_1, \ldots, a_k \in \mathcal{A}$,
\begin{align}\label{ns-inf}
\sum_{m+n = k+1} \sum_{i=1}^m (-1)^{i(n+1)} ~ \! (\eta_m \circ_i \eta_n) ([r]; a_1, \ldots, a_k) = 0,
\end{align}
where the partial compositions $\eta_m \circ_i \eta_n$ are the graded version of the operations defined in (\ref{ns-circ}).
\end{defn}

An $NS_\infty$-algebra whose underlying graded vector space is concentrated in arity $0$ is nothing but an NS-algebra. On the other hand, an $NS_\infty$-algebra $(\mathcal{A}, \{ \eta_n \}_{n \geq 1})$ for which $\eta_n = 0$, for $n \neq 2$, is nothing but a graded NS-algebra (i.e. an NS-algebra in the category of graded vector spaces). In \cite{das-dend} the present author has considered the explicit description of a $Dend_\infty$-algebra (strongly homotopy dendriform algebra). It turns out that an $NS_\infty$-algebra $(\mathcal{A}, \{ \eta_n \}_{n \geq 1})$ for which
\begin{align*}
\eta_n ([n+1]; a_1, \ldots, a_n) = 0, \text{ for all } n \geq 2 \text{ and } a_1, \ldots, a_n \in \mathcal{A}
\end{align*}
is nothing but a $Dend_\infty$-algebra. Thus, an $NS_\infty$-algebra generalizes $Dend_\infty$-algebras.

\begin{prop}\label{ns-ainf}
Let $(\mathcal{A}, \{ \eta_n \}_{n \geq 1})$ be an $NS_\infty$-algebra. Then $(\mathcal{A}, \{ \overline{\eta}_n \}_{n \geq 1})$ is an $A_\infty$-algebra, where
\begin{align*}
\overline{\eta}_1 (a) := \eta_1 ([1]; a) \quad \text{ and } \quad \overline{\eta}_n (a_1, \ldots, a_n) := \eta_n ( {\scriptstyle [1] + \cdots + [n+1]} ; a_1, \ldots, a_{n}), \text{ for } n \geq 2.
\end{align*}
\end{prop}

\begin{proof}
 Since $(\mathcal{A}, \{ \eta_n \}_{n \geq 1})$ is an $NS_\infty$-algebra, the identities in (\ref{ns-inf}) are hold. Hence for any $k \in \mathbb{N}$ and homogeneous elements $a_1, \ldots, a_k \in \mathcal{A}$, we have
 \begin{align*}
\sum_{m+n = k+1} \sum_{i=1}^m (-1)^{i(n+1)} \sum_{ r=1}^{ k+1} ~ \! (\eta_m \circ_i \eta_n) ([r]; a_1, \ldots, a_k) = 0.
\end{align*}
By expanding the partial compositions (which are the graded version of the operations (\ref{ns-circ})), we get that
\begin{align*}
\sum_{m+n = k+1} &\sum_{i=1}^m (-1)^{i (n+1) + n (|a_1| + \cdots + |a_{i-1}|)} \\ & \qquad  ~ \! \eta_m ( {\scriptstyle [1] + \cdots + [m+1]} ; a_1, \ldots, a_{i-1}, \eta_n ( {\scriptstyle [1] + \cdots + [n+1]} ; a_i , \ldots, a_{i+n -1}), a_{i+n}, \ldots, a_k   ) = 0.
\end{align*}
This shows the desired result.
\end{proof}

\begin{thm}
Let $(\mathcal{A}, \{ \mu_n \}_{n \geq 1})$ be an $A_\infty$-algebra and $\mathcal{N} : \mathcal{A} \rightarrow \mathcal{A}$ be a strict homotopy Nijenhuis operator on it. Then the pair $(A, \{ \eta_n \}_{n \geq 1})$ is an $NS_\infty$-algebra, where
\begin{align*}
\eta_n ([r]; a_1, \ldots, a_n) = \begin{cases} \medskip \medskip
\mu_n \big( \mathcal{N} (a_1), \ldots, a_r, \ldots, \mathcal{N} (a_n) \big) & \text{ for } r=1,\ldots, n, \\ \medskip
- \sum_{1 \leq i < j \leq n} \mathcal{N} \big( \mu_n \big( \mathcal{N} (a_1), \ldots, a_i, \ldots,a_j , \ldots, \mathcal{N} (a_n) \big)  \big) & \text{ for } r= n+1.\\
\qquad + \cdots + (-1)^{n-1} \mathcal{N}^{n-1} \big(  \mu_n (a_1, \ldots, a_n)  \big)
\end{cases}
\end{align*}
As a consequence, $(\mathcal{A}, \{ \mu_{n, \mathcal{N}} \}_{n \geq 1} )$ is an $A_\infty$-algebra (called the deformed $A_\infty$-algebra), where $\mu_{1, \mathcal{N}} := \mu_1$ and for $n \geq 2$,
\begin{align*}
\mu_{n, \mathcal{N}}  &(a_1, \ldots, a_n) := \sum_{r=1}^n \mu_n \big( \mathcal{N} (a_1), \ldots, a_r, \ldots, \mathcal{N} (a_n) \big)   \\ & ~~~  - \sum_{1 \leq i < j \leq n} \mathcal{N} \big( \mu_n \big( \mathcal{N} (a_1), \ldots, a_i, \ldots,a_j , \ldots, \mathcal{N} (a_n) \big)  \big)  + \cdots + (-1)^{n-1} \mathcal{N}^{n-1} \big(  \mu_n (a_1, \ldots, a_n)  \big).
\end{align*}
\end{thm}

\begin{proof}
Since $\mathcal{N} : A \rightarrow \mathcal{A}$ is a strict homotopy Nijenhuis operator, it follows from (\ref{strict-hn}) that
\begin{align*}
\mu_n \big(  \mathcal{N} (a_1), \ldots, \mathcal{N} (a_n) \big) = \mathcal{N} \big(  \eta_n ( {\scriptstyle [1] + \cdots + [n+1]} ; a_1, \ldots, a_n)  \big), \text{ for all } n.
\end{align*}
On the other hand, $(\mathcal{A}, \{ \mu_n \}_{n \geq 1})$ is an $A_\infty$-algebra implies that 
\begin{align}\label{ainf-new}
\sum_{m+n = k+1} \sum_{i=1}^m (-1)^{i (n+1) + n ( |a_1| + \cdots + |a_{i-1}|)} ~ \! \mu_m (a_1, \ldots, a_{i-1}, \mu_n (a_i, \ldots, a_{i+n -1}), a_{i+n}, \ldots, a_k) = 0,
\end{align} 
for all $k \geq 1$. In the identity (\ref{ainf-new}), replace the tuple $(a_1, \ldots, a_k)$ of homogeneous elements of $\mathcal{A}$ by the tuple $(\mathcal{N} (a_1), \ldots, a_r, \ldots, \mathcal{N} (a_k))$, for some $1 \leq r \leq k$. For any fixed $m, n$ and $i$, if $r \leq i-1$, then we get that
%the term inside the summation becomes
\begin{align*}
 &(-1)^{n ( |a_1| + \cdots + |a_{i-1}|)} ~ \! \mu_m \big(  \mathcal{N} (a_1), \ldots, a_r, \ldots, \mathcal{N} (a_{i-1}) , \mu_n \big( \mathcal{N} (a_i), \ldots, \mathcal{N} (a_{i+n-1}) \big), \ldots, \mathcal{N} (a_k) \big) \\
&=  (-1)^{n ( |a_1| + \cdots + |a_{i-1}|)} ~ \!  \mu_m \big(  \mathcal{N} (a_1), \ldots, a_r, \ldots, \mathcal{N} (a_{i-1}), \mathcal{N} \big(  \eta_n ( {\scriptstyle [1] + \cdots + [n+1]}; a_i, \ldots, a_{i+n-1}) \big), \ldots, \mathcal{N} (a_k)       \big) \\
&=  (-1)^{n ( |a_1| + \cdots + |a_{i-1}|)} ~ \! \eta_m \big(   [r] ; a_1, \ldots, a_{i-1} , \eta_n ( {\scriptstyle [1] + \cdots + [n+1]} ; a_i, \ldots, a_{i+n-1} ), \ldots, a_k \big) \\
&= (\eta_m \circ_i \eta_n) ([r]; a_1, \ldots, a_k).
\end{align*}
On the other hand, if $i \leq r \leq i+ n -1$ then 
\begin{align*}
 &(-1)^{n ( |a_1| + \cdots + |a_{i-1}|)} ~ \! \mu_m \big(  \mathcal{N} (a_1), \ldots, \mathcal{N} (a_{i-1}) , \mu_n \big( \mathcal{N} (a_i), \ldots, a_r, \ldots, \mathcal{N} (a_{i+n-1}) \big), \ldots, \mathcal{N} (a_k) \big) \\
&=  (-1)^{n ( |a_1| + \cdots + |a_{i-1}|)} ~ \!  \mu_m \big(  \mathcal{N} (a_1), \ldots, \mathcal{N} (a_{i-1}),  \eta_n ( [r-i+1]; a_i, \ldots, a_{i+n-1}), \ldots, \mathcal{N} (a_k)       \big) \\
&=  (-1)^{n ( |a_1| + \cdots + |a_{i-1}|)} ~ \! \eta_m \big(   [i] ; a_1, \ldots, a_{i-1} , \eta_n ( [r-i+1] ; a_i, \ldots, a_{i+n-1} ), \ldots, a_k \big) \\
&= (\eta_m \circ_i \eta_n) ([r]; a_1, \ldots, a_k).
\end{align*}
Similarly, if $i+n \leq r \leq k$, we get that
\begin{align*}
&(-1)^{n ( |a_1| + \cdots + |a_{i-1}|)} ~ \! \mu_m \big(  \mathcal{N} (a_1), \ldots, \mathcal{N} (a_{i-1}) , \mu_n \big( \mathcal{N} (a_i), \ldots, \mathcal{N} (a_{i+n-1}) \big), \ldots, a_r, \ldots, \mathcal{N} (a_k) \big) \\
& = (\eta_m \circ_i \eta_n) ([r]; a_1, \ldots, a_k).
\end{align*}
This implies that the $NS_\infty$-algebra identities (\ref{ns-inf}) are hold for $[r]= [1],\ldots, [k]$. Finally, a straightforward but tedious computation (using the identity (\ref{strict-hn})) shows that
\begin{align*}
\sum_{i=1}^m (-1)^{i (n+1) + n ( |a_1| + \cdots + |a_{i-1}|)} ~ \!  \bigg\{  &\eta_m \big( [i] ; a_1, \ldots, a_{i-1}, \eta_n ( [n+1] ; a_i, \ldots, a_{i+n-1}), \ldots, a_k    \big) \\
+~& \eta_m \big(  [m+1]; a_1, \ldots, a_{i-1} , \eta_n (  {\scriptstyle [1] + \cdots + [n+1]}; a_i, \ldots, a_{i+n-1}), \ldots, a_k   \big)  \bigg\} = 0
\end{align*}
which imply that $\sum_{m+n= k +1} \sum_{i=1}^m (-1)^{i (n+1)} ~ \! (\eta_m \circ_i \eta_n) ( [k+1] ; a_1, \ldots, a_k) = 0$.
This verifies the identity (\ref{ns-inf}) for $[r]= [k+1]$. Hence the proof of the first part follows.

The second part of the statement follows as a consequence of Proposition \ref{ns-ainf}.
\end{proof}

It is important to remark that a strict homotopy Nijenhuis operator on an $A_\infty$-algebra is a generalization of Nijenhuis operators on $n$-ary algebras \cite{liu-sheng}. Motivated by a result from this reference, here we obtain the following.

\begin{thm}
Let $(\mathcal{A}, \{ \mu_n \}_{n \geq 1})$ be an $A_\infty$-algebra and $\mathcal{N} : \mathcal{A} \rightarrow \mathcal{A}$ be a strict homotopy Nijenhuis operator.
\begin{itemize}
\item[(i)] Then for any $k \geq 0$, the map $\mathcal{N}^k : \mathcal{A} \rightarrow \mathcal{A}$ is also a strict homotopy Nijenhuis operator on the $A_\infty$-algebra  $(\mathcal{A}, \{ \mu_n \}_{n \geq 1})$. Hence one obtains the deformed $A_\infty$-algebra $(\mathcal{A}, \{ \mu_{n, \mathcal{N}^k} \}_{n \geq 1})$.
\item[(ii)] Further, for any $k , l \geq 0$, the map $\mathcal{N}^l : \mathcal{A} \rightarrow \mathcal{A}$ is a strict homotopy Nijenhuis operator on the deformed $A_\infty$-algebra $(\mathcal{A}, \{ \mu_{n, \mathcal{N}^k} \}_{n \geq 1})$.
\end{itemize}
\end{thm}

\begin{proof}
(i) We will prove this result by using the mathematical induction on $k$. First, observe that the statement is true for $k = 0 ,1 $. Assume that the statement is true for some natural number $k$. Choose any homogeneous elements $a_1, \ldots, a_n \in \mathcal{A}$. Since $\mathcal{N}^k$ is a strict homotopy Nijenhuis operator (by assumption), we observe that
\begin{align}
&\mu_n \big(  \mathcal{N}^{k+1} (a_1), \ldots, \mathcal{N}^{k+1} (a_n) \big) \nonumber \\
&= \sum_{i=1}^n \mathcal{N}^k \mu_n \big(    \mathcal{N}^{k+1} (a_1) , \ldots, \mathcal{N} (a_i), \ldots, \mathcal{N}^{k+1} (a_n) \big) \label{nk1} \\
& \quad - \sum_{1 \leq i < j \leq n} \mathcal{N}^{2k} \mu_n \big(   \mathcal{N}^{k+1} (a_1) , \ldots, \mathcal{N} (a_i), \ldots, \mathcal{N} (a_j), \ldots, \mathcal{N}^{k+1} (a_n)    \big) \nonumber \\
& \qquad + \cdots + (-1)^{n-1} \mathcal{N}^{nk}  \mu_n \big(  \mathcal{N} (a_1), \ldots, \mathcal{N} (a_n) \big). \nonumber
\end{align}
By repeatedly using the fact that $\mathcal{N}$ is a strict homotopy Nijenhuis operator and then cancellation of terms, the whole expression in (\ref{nk1}) becomes
\begin{align*}
\sum_{i=1}^n & \mathcal{N}^{k+1} \mu_n \big(    \mathcal{N}^{k+1} (a_1) , \ldots, a_i, \ldots, \mathcal{N}^{k+1} (a_n) \big) \\
& - \sum_{1 \leq i < j \leq n} \mathcal{N}^{2k+2} \mu_n \big(   \mathcal{N}^{k+1} (a_1) , \ldots, a_i, \ldots, a_j, \ldots, \mathcal{N}^{k+1} (a_n) \big) + \cdots + (-1)^{n-1} \mathcal{N}^{nk + n}  \mu_n \big(  a_1, \ldots, a_n \big).
\end{align*}
This shows that $\mathcal{N}^{k+1}$ is a strict homotopy Nijenhuis operator (i.e. the statement is valid for $k+1$). Hence the result follows by the mathematical induction.

(ii) This can be proved by using the mathematical induction on $l$.
\end{proof}

\medskip

\medskip

\noindent {\bf Acknowledgements.} The author thanks the Department of Mathematics, IIT Kharagpur for providing the beautiful academic atmosphere where the research has been carried out.
%\vspace*{1cm}

%\medskip

%\noindent {\bf Funding.} Ramkrishna Mandal would like to thank the Government of India for supporting his work through the Prime Minister Research Fellowship.

\medskip

\noindent {\bf Data Availability Statement.} Data sharing does not apply to this article as no new data were created or analyzed in this study.


\begin{thebibliography}{BFGM03}


\bibitem{azimi} M. J. Azimi, C. Laurent-Gengoux and J. M. Nunes da Costa, Nijenhuis forms on $L_\infty$-algebras and Poisson geometry, {\em Diff. Geom. Appl.} 38 (2015), 69-113.

%\bibitem{aamar} M. Ammar and N. Poncin, Coalgebraic approach to the Loday infinity category, stem differential for $2n$-ary graded and homotopy algebras, {\em Annales de l'institut Fourier, Tome} 60, no. 1 (2010) 355-387.

%\bibitem{agore} A. L. Agore, Classifying complements for associative algebras, {\em Linear Alg. Appl.} 446 (2014), 345-355.

%\bibitem{agore3} A. L. Agore and G. Militaru, Classifying complements for Hopf algebras and Lie algebras, {\em J. Algebra} 391 (2013), 321-341.

%\bibitem{agore2} A. L. Agore and G. Militaru, The extending structures problem for algebras, Available at: \url{https://arxiv.org/abs/1305.6022v3}

%\bibitem{agore} A. L. Agore and G. Militaru, Unified products for Leibniz algebras. Applications, {\em Linear Algebra Appl.} 439 (2013), 2609-2633.

%\bibitem{agore4} A. L. Agore and G. Militaru, Hochschild products and global non-abelian cohomology for algebras. Applications, {\em J. Pure Applied Algebra} 221 (2017), 366-392.

%\bibitem{aguiar} M. Aguiar, Pre-Poisson algebras, {\em Lett. Math. Phys.} 54 (2000), 263-277.

%\bibitem{Bai-Guo-Ni} C. Bai, L. Guo, and X. Ni, $\mathcal{O}$-operators on associative algebras and associative Yang–Baxter equations, {\em Pacific J. Math.} 256 (2012), 257-289.

%\bibitem{guo-commun} C. Bai, L. Guo and X. Ni, Nonabelian generalized Lax pairs, the classical Yang-Baxter equation and PostLie algebras, {\em Commun. Math. Phys.} 297 (2010), 553-596.

%\bibitem{Bai-Guo-Ni} C. Bai, L. Guo, and X. Ni, Relative Rota-Baxter algebras and tridendriform algebras, {\em J. Algebra Appl.} 12, No. 07 (2013), 1350027.

%\bibitem{baishya-das} A. Baishya and A. Das, Cup product, Fr\"{o}licher-Nijenhuis bracket and the derived bracket associated to Hom-Lie algebras, arXiv:2409.01865

\bibitem{baishya-das2} A. Baishya and A. Das, Nijenhuis deformations of Poisson algebras and $F$-manifold algebras, Online available at: https://arxiv.org/abs/2403.18496

\bibitem{baishya-das3} A. Baishya and A. Das, Fr\"{o}licher-Nijenhuis bracket and derived bracket associated to a nonsymmetric operad with multiplication, In preparation.

%\bibitem{bala} D. Balavoine, Deformations of algebras over a quadratic operad, {\em Contemp. Math.} 202 (1967), 207-234.

%\bibitem{nij-geom} A. V. Balsinov, A. Y. Konyaev and V. S. Matveev, Nijenhuis geometry, {\em Adv. Math.} 394 (2022), 108001.

\bibitem{bar-singh} V. G. Bardakov and M. Singh, Extensions and automorphisms of Lie algebras, {\em J. Algebra Appl.} 16 (2017), 1750162.


%\bibitem{bloh} A. Bloh, On a generalization of the concept of Lie algebra, {\em Dokl. Akad. Nauk. SSSR.} 165(3) (1965), 471-473.

%\bibitem{bor} M. Bordemann and F. Wagemann, Global integration of Leibniz algebras, {\em J. Lie Theory} 27 (2017), no. 2, 555–567.

\bibitem{gra-bi} J. F. Cari\~{n}ena, J. Grabowski and G. Marmo, Quantum Bi-Hamiltonian systems, {\em Int. J. Mod. Phys. A} 15 (2000), 4797-4810. 

%\bibitem{cao} W. Cao, An algebraic study of averaging operators, Ph.D. Thesis, Rutgers University at Newark (2000).

%\bibitem{conn} A. Connes and D. Kreimer, Renormalization in quantum field theory and the Riemann-Hilbert problem I: the Hopf algebra structure of graphs and the main theorem, {\em Commun. Math. Phys.} 210 (2000), 249-273.

%\bibitem{covez} S. Covez, The local integration of Leibniz algebras, {\em Ann. Inst. Fourier (Grenoble)} 63 (2013), no. 1, 1-35.

\bibitem{das-rota} A. Das, Deformations of associative Rota-Baxter operators, {\em J. Algebra} 560 (2020) 144-180.

%\bibitem{das-twisted} A. Das,  Twisted Rota-Baxter operators and Reynolds operators on Lie algebras and NS-Lie algebras, {\em J. Math. Phys.} Vol. 62, Issue 9  (2021) 091701.

%\bibitem{das-leib-hrs} A. Das, Leibniz algebras with derivations, {\em J. Homotopy Relat. Struct.} 16 (2021) 245-274.

\bibitem{das-ns} A. Das, Cohomology and deformations of twisted Rota-Baxter operators and NS-algebras, {\em  J. Homotopy Relat. Struct.} 17 (2022) 233-262.

%\bibitem{das-weighted} A. Das, Cohomology and deformations of weighted Rota-Baxter operators, {\em J. Math. Phys.} Vol. 63, Issue 9 (2022) 091703.

%\bibitem{das-crossed} A. Das, Cohomology and deformations of crossed homomorphisms, {\em Bull. Belgian Math. Soc. Simon Stevin} Vol. 28, Issue 3 (2022) 381-397.

\bibitem{das-dend} A. Das, Cohomology and deformations of dendriform algebras, and $Dend_\infty$-algebras, {\em Comm. Algebra} 50 (2022), 1544-1567.

%\bibitem{das-modified} A, Das, A cohomological study of modified Rota-Baxter algebras, Available at: \url{https://arxiv.org/abs/2207.02273}

%\bibitem{das-avg} A. Das, Controlling structures, deformations and homotopy theory for averaging algebras, Available at: \url{https://arxiv.org/abs/2303.17798}

%\bibitem{das-leib} A. Das, Relative Rota-Baxter Leibniz algebras, their characterization and cohomology, {\em Linear Multilinear Algebra} Vol. 71, Issue 17 (2023) 2796-2822.

%\bibitem{das-leib-pub} A. Das, Weighted relative Rota-Baxter operators on Leibniz algebras and Post-Leibniz algebra structures, {\em Publ. Math. Debrecen} 103, Issue 3-4 (2023) 385-406.

%\bibitem{das-guo} A. Das and S. Guo, Cohomology and deformations of generalized Reynolds operators on Leibniz algebras, {\em Rocky Mountain J. Math.} 54 (1) (2024), 161-178.

%\bibitem{das-mandal} A. Das and R. Mandal, Quasi-twilled associative algebras, deformation maps and their governing algebras, Available at: \url{https://arxiv.org/abs/2409.00443}

\bibitem{das-mishra-jmp} A. Das and S. K. Mishra, The $L_\infty$-deformations of associative Rota-Baxter algebras and homotopy Rota-Baxter operators, {\em J. Math. Phys.} Vol. 63, Issue 5 (2022) 051703.

%\bibitem{das-mishra2} A. Das and S. K. Mishra, Bimodules over relative Rota-Baxter algebras and cohomologies, {\em Algebr. Represent. Theor.} 26 (2023), 1823–1848.

%\bibitem{dherin} B. Dherin and F. Wagemann, Deformation quantization of Leibniz algebras, {\em Adv. Math.} 270 (2015), 21–48.

%\bibitem{drinfeld} V. Drinfeld, Quasi-Hopf algebras, {\em Leningrad Math. J.} 1 (1989), 1419-1457.

\bibitem{ebrahimi-loday} K. Ebrahimi-Fard, Loday-type algebras and the Rota-Baxter relation, {\em Lett. Math. Phys.} 61, no. 2 (2002), 139-147.

\bibitem{ebrahimi} K. Ebrahimi-Fard, On the associative Nijenhuis relation, {\em Elec. J. Combin.} 11 (2004), R38.

\bibitem{ebrahimi-leroux} K. Ebrahimi-Fard and P. Leroux, Generalized shuffles related to Nijenhuis and TD-algebras, {\em Comm. Algebra} 37 (2009), 3064-3094.

\bibitem{mapping} D. Fiorenza and M. Manetti, $L_\infty$-structures on mapping cones, {\em Algebra Number Theory} 1 (3) (2007), 301-330.

\bibitem{yau} Y. Fr\'{e}gier, M. Markl and D. Yau, The $L_\infty$-deformation complex of diagrams of algebras, {\em New York J. Math.} 15 (2009), 353-392.

\bibitem{fre} Y. Fr\'{e}gier and M. Zambon, Simultaneous deformations of algebras and morphisms via derived brackets, {\em J. Pure Appl. Algebra} 219 (2015), 5344-5362.

%\bibitem{gers-ring} M. Gerstenhaber, The cohomology structure of an associative ring, {\em Ann. Math. (2)} 78 (1963), 267-288.

%\bibitem{gers}  M. Gerstenhaber, On the deformation of rings and algebras, {\em Ann. Math. (2)} 79 (1964), 59-103.

%\bibitem{gers0} M. Gerstenhaber, The cohomology structure of an associative ring, {\em Ann. of Math.} (2) 78 (1963), 267-288.

\bibitem{gers}  M. Gerstenhaber, On the deformation of rings and algebras, {\em Ann. of Math. (2)} 79 (1964), 59-103.

\bibitem{gers-sch} M. Gerstenhaber and S. D. Schack, On the deformation of algebra morphisms and diagrams, {\em Trans. Amer. Math. Soc.} 279 (1983), no. 1, 1-50.

\bibitem{guo-book} L. Guo, What is ... a Rota-Baxter algebra? {\em Notices of the AMS} 56 (2009), 1436-1437.

\bibitem{lin} L. Guo and Z. Lin, Representations and modules of Rota-Baxter algebras, {\em Asian J. Math.} 25 (2021), 841-870.

%\bibitem{getzler} E. Getzler, Lie theory for nilpotent $L_{\infty}$-algebras, {\em Ann. Math. (2)} 170 (2009), 271-301.

%\bibitem{gouray} J.-B. Gouray, A differential graded Lie algebra approach to non-abelian extensions of associative algebras, Available at: \url{https://arxiv.org/abs/1802.04641}


%\bibitem{guo-book} L. Guo, An introduction to Rota-Baxter algebra, {\em International Press, Somerville, MA; Higher Education Press, Beijing}, 2012. xii+226 pp. ISBN:978-1-57146-253-4


\bibitem{hoch} G. Hochschild, On the cohomology groups of an associative algebra, {\em Ann. Math. (2)} 46 (1945), 58-67.


%\bibitem{jiang} J. Jiang and Y. Sheng, Deformations, cohomologies and integrations of relative difference Lie algebras, {\em J. Algerba} 614 (2023), 535-563.

%\bibitem{jiang-sheng-tang} J. Jiang, Y. Sheng and R. Tang, Deformation maps of quasi-twilled Lie algebras, Available at: \url{https://arxiv.org/abs/2405.02532}

\bibitem{jin} P. Jin and H. Liu, The Wells exact sequence for the automorphism group of a group extension, {\em J. Algebra} 324 (2010), 1219-1228.

%\bibitem{kaj-stas} H. Kajiura and J. Stasheff, Homotopy algebras inspired by classical open-closed string field theory, {\em Commun. Math. Phys.} 263 (2006), 553-581.

\bibitem{keller} B. Keller, Introduction to $A_\infty$-algebras and modules, {\em Homology Homotopy Appl.} 3 (2001), 1-35.

%\bibitem{khuda} D. Khudaverdyan, N. Poncin and J. Qiu, On the infinity category of homotopy Leibniz algebras, {\em Theory Appl. Categ.} Vol. 29, No. 12 (2014), pp 332-370.

\bibitem{koss} Y. Kosmann-Schwarzbach and F. Magri, Poisson-Nijenhuis structures,  {\em Annales de l'I.H.P. Physique th\'{e}orique} 53, no. 1 (1990), 35-81.

%\bibitem{kotov} A. Kotov and T. Strobl, The embedding tensor, Leibniz-Loday algebras, and their higher gauge theories, {\em Commun. Math. Phys.} 376 (2020), no. 1, 235–258.

%\bibitem{jonas} A. Kraft and J. Schnitzer, An introduction to $L_\infty$-algebras and their homotopy theory for the working mathematician, {\em Rev. Math. Phys.} 36, No. 01 (2024), 2330006.

%\bibitem{kuper} B. A. Kupershmidt, What a classical {\em r}-matrix really is, {\em J. Nonlinear Math. Phys.} 6, No. 4 (1998), 448-488.

%\bibitem{lada-markl} T. Lada and M. Markl, Strongly homotopy Lie algebras, {\it Comm. Algebra} 23 (1995), 2147-2161.

%\bibitem{loday-dialgebra} J.-L. Loday, Dialgebra, in Dialgebras and related operad, {\em Lecture Notes in Math.} 1763 (2002), 7-66.

%\bibitem{lazarev} A. Lazarev, Y. Sheng and R. Tang, Deformations and homotopy theory of relative Rota-Baxter Lie algebras, {\em Commun. Math. Phys.} 383 (2021), 595-631.

%\bibitem{li-wang} Y. Li and D. Wang, Cohomology and deformation theory of crossed homomorphisms of Leibniz algebras, {\em J. Algebra Appl.} to appear. DOI: https://doi.org/10.1142/S0219498825501956

%\bibitem{Liu-Bai-sheng} J. Liu, C. Bai and Y. Sheng, Compatible $\mathcal{O}$-operators on bimodules over associative algebras, {\em J. Algebra }532 (2019), 80-118.

%\bibitem{loday-ronco} J.-L. Loday and M. Ronco, Trialgebras and families of polytopes, {\em Contemp. Math.} 346 (2004), 369-398.

%\bibitem{loday-morph} J.-L. Loday,  On the operad of associative algebras with derivation, {\em Georgian Math. J.} 17 (2010), no. 2, 347-372.

%\bibitem{lod-val-book} J.-L. Loday and B. Vallette, Algebraic operads, Grundlehren mathematischen Wissenschaften, Volume 346, Springer-Verlag (2012), xviii+512 pp.


\bibitem{lazarev} A. Lazarev, Y. Sheng and R. Tang, Deformations and homotopy theory of relative Rota–Baxter Lie algebras, {\em Commun. Math. Phys.} 383 (2021), 595-631.

\bibitem{lei} P. Lei and L. Guo, Nijenhuis algebras, NS algebras, and N-dendriform algebras, {\em Front. Math. China} 7 (2012), 827–846.

\bibitem{leroux} P. Leroux, Construction of Nijenhuis operators and dendriform trialgebras, {\em Int. J. Math. Math. Sci.} 49 (2004), 2595-2615.

\bibitem{liu} J. Liu, C. Bai and Y. Sheng, Compatible $\mathcal{O}$-operators on bimodules over associative algebras, {\em J. Algebra} 532 (2019), 80-118.

\bibitem{liu-sheng} J. Liu, Y. Sheng, Y. Zhou and C. Bai, Nijenhuis operators on $n$-Lie algebras, {\em Commun. Theor. Phys.} 65 (2016), pp. 659.

%\bibitem{loday} J.-L. Loday, Une version non commutative des alg\'{e}bres de Lie: les alg\'{e}bres de Leibniz, {\em Enseign. Math.} (2) 39 (1993), no. 3-4, 269-293.

%\bibitem{loday-pira}  J.-L. Loday and T. Pirashvili, Universal enveloping algebras of Leibniz algebras and (co)homology, {\em Math. Ann.} 296 (1993), 139-158.

\bibitem{loday-di} J.-L. Loday, Dialgebras, {\em Dialgebras and related operads}, 7-66, Lecture Notes in Math., 1763, {\em Springer, Berlin}, 2001.

\bibitem{loday-der} J.-L. Loday, On the operad of associative algebras with derivation, {\em Georgian Math. J.} 17 (2010), 347-372.


\bibitem{ma} T. Ma and L. Long, Nijenhuis operators and associative $D$-bialgebras, {\em J. Algebra} 639 (2024), 150-186. 

%\bibitem{mandal} A. Mandal, Deformations of Leibniz algebra morphisms, {\em Homology Homotopy Appl.} 9(2007), no. 1, 439–450.

%\bibitem{men} I. Mencattini and A. Quesney, Crossed homomorphisms, integration of Post-Lie algebras and the Post-Lie magnus expansion, {\em Comm. Algebra} 49 (2021), 3507-3533.

\bibitem{das-hazra-mishra} S. K. Mishra, A. Das and S. K. Hazra,
Non-abelian extensions of Rota-Baxter Lie algebras and inducibility of automorphisms, {\em Linear Alg. Appl.} 669 (2023), 147-174.

%\bibitem{saha} B. Mondal and R. Saha, Cohomology of modified Rota-Baxter Leibniz algebra of weight $\lambda$, {\em J. Algebra Appl.} to appear. DOI: https://doi.org/10.1142/S0219498825501579

%\bibitem{nij-ric} A. Nijenhuis and R. Richardson, Cohomology and deformations in graded Lie algebras, {\em Bull. Amer. Math. Soc.} 72 (1966), 1-29.

%\bibitem{nij-ric2} A. Nijenhuis and R. Richardson, Deformations of homomorphisms of Lie groups and Lie algebras, {\em Bull. Amer. Math. Soc.} 73 (1967), 175-179.

%\bibitem{pei} Y. Pei, Y. Sheng, R. Tang and K. Zhao, Actions of monoidal categories and representations of Cartan type Lie algebras, {\em J. Inst. Math. Jussieu} 22 (2023), 2367-2402.

\bibitem{saha} B. Mondal and R. Saha, Nijenhuis operators on Leibniz algebras, {\em J. Geom. Phys.} 196 (2024), 105057.

\bibitem{nij-ric} A. Nijenhuis and R. W. Richardson, Cohomology and deformations in graded Lie algebras, {\em Bull. Amer. Math. Soc.} 72 (1966), 1-29.

\bibitem{peng} X. S. Peng, Y. Zhang, X. Gao and Y. F. Luo, Dendriform-Nijenhuis bialgebras and DN-associative Yang-Baxter equations, {\em J. Algebra} 575 (2021), 78-126.

%\bibitem{sheng} Y. Sheng, A survey on deformations, cohomologies and homotopies of relative Rota-Baxter Lie algebras, {\em Bull. London Math. Soc.} 54 (2022), 2045-2065.

%\bibitem{sheng-embed} Y. Sheng, R. Tang and C. Zhu, The controlling $L_\infty$-algebra, cohomology and homotopy of embedding tensors and Lie-Leibniz triples, {\em Commun. Math. Phys.} 386 (2021), 269-304.

%\bibitem{tang} R. Tang, C. Bai, L. Guo and Y. Sheng, Deformations and their controlling cohomologies of $\mathcal{O}$-operators, {\em Commun. Math. Phys.} 368 (2019), 665-700.

%\bibitem{tang-sheng} R. Tang and Y. Sheng, Leibniz bialgebras, relative Rota–Baxter operators, and the classical Leibniz Yang–Baxter equation, {\em J. Noncommut. Geom.} 16 (2022), 1179–1211.

\bibitem{O-op} R. Tang, C. Bai, L. Guo and Y. Sheng, Deformations and their controlling cohomologies of $\mathcal{O}$-operators, {\em Commun. Math. Phys.} 368 (2019), 665-700.

%\bibitem{tang-sheng-zhou} R. Tang, Y. Sheng and Y. Zhou, Deformations of relative Rota–Baxter operators on Leibniz algebras, {\em Int. J. Geom. Methods Mod. Phys.} Vol. 17, No. 12 (2020), 2050174.

%\bibitem{tang-xu-sheng} R. Tang, N. Xu and Y. Sheng, Symplectic structures, product structures and complex structures on Leibniz algebras, {\em J. Algebra} 647 (2024), 710-743.

\bibitem{uchino}  K. Uchino, Quantum analogy of Poisson geometry, related dendriform algebras and Rota-Baxter operators, {\em Lett. Math. Phys.} 85, no. 2-3 (2008), 91-109.

%\bibitem{uchino-t}  K. Uchino, Twisting on associative algebras and Rota–Baxter type operators, {\em J. Noncommut. Geom.} 4 (2010), 349–379.

%\bibitem{voro} Th. Voronov, Higher derived brackets and homotopy algebras, {\em J. Pure Appl. Algebra} 202 (2005), 133-153.

\bibitem{wang} Q. Wang, Y. Sheng, C. Bai and J. Liu, Nijenhuis operators on pre-Lie algebras, {\em Commun. Contemp. Math.} 21, No. 07 (2019), 1850050.
  

%\bibitem{wang-zhou2} K. Wang and G. Zhou, Cohomology theory of averaging algebras, $L_\infty$-structures and homotopy averaging algebras, Available at: \url{https://arxiv.org/abs/2009.11618}

%\bibitem{wang-zhou} K. Wang and G. Zhou, Deformations and homotopy theory of Rota-Baxter algebras of any weight, Online available at: \url{https://arxiv.org/abs/2108.06744}

\bibitem{wells} C. Wells, Automorphisms of group extensions, {\em Trans. Amer. Math. Soc.} 155 (1971), 189-194.

\bibitem{yuan} L. Yuan, $\mathcal{O}$-operators and Nijenhuis operators of associative conformal algebras, {\em J. Algebra} 609 (2022), 245-291.

%\bibitem{zhang-gao-guo} T. Zhang, X. Gao and L. Guo, Reynolds algebras and their free objects from bracketed words and rooted trees, {\em J. Pure Applied Algebra} 225 (2021) 106766.

\end{thebibliography}
\end{document}